\documentclass[a4paper,11pt]{article}
\usepackage{amssymb,amsmath,amsfonts,amsbsy}
\usepackage{mathrsfs}
\usepackage[all]{xy}
\usepackage[latin1]{inputenc}
\usepackage{graphics, amsthm, amscd, mathrsfs, verbatim}
\usepackage{epsfig}
\usepackage{color}
\usepackage{enumerate}
\usepackage{euscript}
\usepackage{anysize}
\usepackage{caption}
\usepackage{hyperref}

\marginsize{2.6cm}{2.6cm}{2cm}{2.4cm}
\input diagxy

\begin{document}

\bibliographystyle{utphys}

\newtheorem{theorem}{Theorem}[section]
\newtheorem{conjecture}{Conjecture}
\newtheorem{lemma}[theorem]{Lemma}
\newtheorem{proposition}[theorem]{Proposition}
\newtheorem{corollary}[theorem]{Corollary}
\newtheorem{notation}[theorem]{Notation}
\theoremstyle{definition}
\newtheorem{remark}[theorem]{Remark}
\newtheorem{definition}[theorem]{Definition}

\newcommand{\func}[3]{{#1}\co{#2}\to{#3}}

\newcommand{\ip}[2]{\langle{#1},{#2}\rangle }
\newcommand{\dd}[2]{\frac{\partial {#1}}{\partial {#2}}}
\newcommand{\ddsq}[2]{\frac{\partial^2 {#1}}{\partial {#2}^2}}

\newcommand{\id}[1]{\operatorname{id}_{#1}}
\newcommand{\grmap}[1]{\operatorname{gr}_{#1}}
\newcommand{\crit}[1]{\operatorname{Crit}({#1})}
\newcommand{\critmbl}[1]{\operatorname{MBL}({#1})}
\newcommand{\proj}[1]{\operatorname{proj}_{#1}}
\newcommand{\norm}[1]{\left|{#1}\right|}
\newcommand{\vecnorm}[1]{\left|\left|{#1}\right|\right|}

\newcommand{\cskh}{\operatorname{CKH}_{\rm symp}}
\newcommand{\skh}{\operatorname{KH}_{\rm symp}}
\newcommand{\kh}{\operatorname{KH}}
\newcommand{\sympc}{\operatorname{Symp}}
\newcommand{\proja}{\operatorname{proj}_a}
\newcommand{\projX}{\operatorname{proj}_X}
\newcommand{\projbc}{\operatorname{proj}_{bc}}
\newcommand{\nonreg}{{\text{non-reg}}}
\newcommand{\tame}{{\text{tame}}}

\newcommand{\sfibre}[2]{\EuScript{Y}_{{#1},{#2}}}
\newcommand{\match}[2]{\mathfrak{M}^{#1}_{#2}}
\newcommand{\bra}[1]{\left<{#1}\right|}
\newcommand{\ket}[1]{\left|{#1}\right>}
\newcommand{\braket}[2]{\left<{#1}\vert{#2}\right>}
\newcommand{\interval}[0]{\left[-1,1\right]}
\newcommand{\im}[0]{\operatorname{im} }
\newcommand{\nbd}[0]{\mathbf{N}}
\newcommand{\co}[0]{\colon\thinspace }
\newcommand{\lcan}[0]{\lambda_{\mbox{can}}}
\newcommand{\trans}[0]{ \pitchfork }
\newcommand{\ntrans}[0]{ \!{\not{\!\pitchfork}} }
\newcommand{\ham}[0]{ \operatorname{Ham} }
\newcommand{\conf}[1]{\operatorname{Conf}^0_{#1}(\mathbb{C})}
\newcommand{\confbar}[1]{\overline{\operatorname{Conf}}^0_{#1}(\mathbb{C})}
\newcommand{\reg}{\operatorname{Reg}}

\newcommand{\minushorizresolution}{\;\;\includegraphics{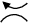}\;}
\newcommand{\plushorizresolution}{\;\;\includegraphics{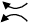}\;}
\newcommand{\minusvertresolution}{\;\;\includegraphics{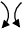}\;}
\newcommand{\plusvertresolution}{\;\;\includegraphics{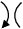}\;}
\newcommand{\minuscrossing}{\;\;\includegraphics{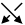}\;}
\newcommand{\pluscrossing}{\;\;\includegraphics{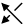}\;}

\title{An invariant of link cobordisms from\\ symplectic Khovanov homology}
\author{Jack W. Waldron}

\maketitle

\begin{abstract}
Symplectic Khovanov homology is an invariant of oriented links defined by
Seidel and Smith \cite{seidelsmith:khsymp} and conjectured to be isomorphic to
Khovanov homology. We define morphisms (up to a global sign ambiguity) between
symplectic Khovanov homology groups, corresponding to isotopy classes of smooth
link cobordisms in $\mathbb{R}^3\!\times\![0,1]$ between a fixed pair of links.
These morphisms define a functor from the category of links and such cobordisms
to the category of
abelian groups and group homomorphisms up to a sign ambiguity.
This provides an extra structure for symplectic Khovanov homology and more generally
an isotopy invariant of smooth surfaces in $\mathbb{R}^4$; a
first step in proving the conjectured isomorphism of
symplectic Khovanov homology and Khovanov homology. The maps themselves are
defined using a generalisation of Seidel's relative invariant of exact Lefschetz
fibrations \cite{seidel:les} to exact Morse-Bott-Lefschetz fibrations with
non-compact singular loci.
\end{abstract}

\tableofcontents

\section{Introduction}

The Jones polynomial $V_\mathfrak{L}$, introduced in \cite{jones:poly}, is an invariant of oriented
links $\mathfrak{L}\subset S^3$, motivated originally by a connection
between the Yang-Baxter equation in statistical mechanics and the braid group.
Importantly, $V_\mathfrak{L}$ satisfies skein relations:
\[t^{-1/2} V_{\!\scalebox{0.6}{\plushorizresolution}}
  +t^{3v/2} V_{\!\scalebox{0.6}{\plusvertresolution}}
  +t^{-1} V_{\!\scalebox{0.6}{\pluscrossing}} = 0
\]
\[
 t^{3v/2} V_{\!\scalebox{0.6}{\minushorizresolution}}
  +t^{1/2} V_{\!\scalebox{0.6}{\minusvertresolution}}
  +t V_{\!\scalebox{0.6}{\minuscrossing}} = 0
\]
The pictures in the skein relations above depict a region of a crossing diagram
of three oriented links, which are identical elsewhere. $v$ is a term, arising as a signed
count of certain crossings, which compensates for the fact that there is no canonical
choice of orientation on one resolution of any crossing.

Together with the normalisation $V_{\rm unknot}=1$, this allows the Jones polynomial to be algorithmically
computed in a straightforward, though computationally intensive, manner from a crossing diagram of a link.
However, the geometric meaning of the Jones polynomial is still poorly understood.
One hopes that an intrinsically geometric definition of the Jones polynomial would
advance geometric applications to knot theory.

Khovanov \cite{khovanov:catjones,khovanov:ftangle} defines a bigraded abelian group
$\kh ^{*,*}(\mathfrak{L})$,
which categorifies the Jones polynomial. It is widely known as the \emph{Khovanov homology}
of $\mathfrak{L}$, although technically a cohomology theory.
The skein relations for the Jones polynomial are replaced by long exact sequences for
Khovanov homology (bi-degrees are indicated on the arrows):
\[
\bfig
 \node a(600,500)[\kh ^{*,*}\!\left(\raisebox{-0.5mm}{\!\!\scalebox{0.8}{\pluscrossing}}\!\right)]
 \node b(1200,0)[\kh ^{*,*}\!\left(\raisebox{-0.5mm}{\!\!\scalebox{0.8}{\plushorizresolution}}\!\right)]
 \node c(0,0)[\kh ^{*,*}\!\left(\raisebox{-1mm}{\!\!\scalebox{0.8}{\plusvertresolution}}\!\right)]
 \arrow[a`b;(0,-1)]
 \arrow[b`c;(-v,-3v-1)]
 \arrow[c`a;(v+1,3v+2)]
\efig
\]
\[
\bfig
 \node a(600,500)[\kh ^{*,*}\!\left(\raisebox{-0.5mm}{\!\!\scalebox{0.8}{\minuscrossing}}\!\right)]
 \node b(1200,0)[\kh ^{*,*}\!\left(\raisebox{-0.5mm}{\!\!\scalebox{0.8}{\minushorizresolution}}\!\right)]
 \node c(0,0)[\kh ^{*,*}\!\left(\raisebox{-1mm}{\!\!\scalebox{0.8}{\minusvertresolution}}\!\right)]
 \arrow[a`b;(-v+1,-3v+2)]
 \arrow[b`c;(v,3v+1)]
 \arrow[c`a;(0,1)]
\efig
\]

A consequence of this is that the Jones polynomial is recovered from a change of
variables on the bigraded Poincar\'e polynomial of $\kh ^{*,*}(\mathfrak{L})$.
\[V_\mathfrak{L} =
	\left[
		\frac
		{\sum_{i,j}(-1)^iq^j\operatorname{dim}(\kh ^{i,j}(\mathfrak{L})\otimes\mathbb{Q})}
		{q+q^{-1}}
	\right]_{q=t^{1/2}}
\]

Khovanov homology is also defined in terms of algorithmically computable algebra
related to crossing diagrams. Its other strengths include, in particular, that
it fits into a topological quantum field theory for knotted-surfaces in $\mathbb{R}^4$.
This property was used by Rasmussen \cite{rasmussen:s} to give a purely combinatorial
proof of Milnor's conjecture on the smooth slice genus of torus knots. The same work
can also be used to construct an exotic smooth structure on $\mathbb{R}^4$.

The topological field theory is described as follows.
Let $\operatorname{\mathbf{Cob}}$ be the category whose objects are oriented
links in $S^3$ and whose morphisms $\mathfrak{L}_1\to\mathfrak{L}_1$
are isotopy classes of smooth link cobordisms
from $\mathfrak{L}_1$ to $\mathfrak{L}_2$ in $S^3\times [0,1]$.
Let $\operatorname{\mathbf{Ab}_2^\pm}$ be the category of finitely generated, bigraded abelian
groups and bigraded homomorphisms defined only up to an overall sign ambiguity.
It is shown \cite{jacobsson:cob} that there is a non-trivial
functor from $\operatorname{\mathbf{Cob}}$
to $\operatorname{\mathbf{Ab}_2^\pm}$, which maps a link $\mathfrak{L}$ to
the Khovanov homology $\kh ^{*,*}(\mathfrak{L})$. To be more precise,
it takes $\mathfrak{L}$ to the Khovanov homology of the crossing
diagram created by projecting a $C^1$--small perturbation of $\mathfrak{L}$
orthogonally to $\{(0,0,z)\in\mathbb{R}^3\}$. One often phrases this result
as ``\emph{Khovanov homology is functorial with respect to smooth link cobordisms}''.

\emph{Symplectic Khovanov homology} $\skh ^*(\mathfrak{L})$ is another
invariant of oriented links, due to
Seidel and Smith \cite{seidelsmith:khsymp}. More recently, it has also been extended
to an invariant of tangles by Rezazadegan \cite{reza}.
For links, it is a singly-graded abelian group, defined using
Lagrangian Floer cohomology in an auxiliary symplectic manifold
$(\sfibre{m}{P},\Omega)$, and Lagrangian submanifolds in $\sfibre{m}{P}$ derived from a presentation of
$\mathfrak{L}$ as a braid closure. Section~\ref{section:khsymp} introduces
symplectic Khovanov homology in more detail, as well as extending the definition
to work with bridge diagrams of links. In doing so, we ignore orientations of the links, which amounts
to replacing the absolute grading on symplectic Khovanov homology by a relative grading.
For the purposes of this paper, we do not require a more precise understanding of the absolute grading.
It suffices to know that the absolute grading always exists, given a choice of orientation.

Seidel and Smith conjecture that symplectic Khovanov homology
is in fact the same as Khovanov homology. Since $\skh ^*(\mathfrak{L})$
has only one grading, the conjecture is usually phrased as:

\begin{conjecture}
\label{conjecture:khiskhsymp}
There is a canonical isomorphism
$\skh ^k(\mathfrak{L})\cong\bigoplus_{i-j=k}\kh ^{i,j}(\mathfrak{L})$
\end{conjecture}

Here, one has also to be careful with the meaning of \emph{canonical}. Section~\ref{section:khsympcross}
defines the symplectic Khovanov homology of a crossing
diagram up to canonical isomorphism, thus allowing one to be precise about the meaning
of the above conjecture.

The main aim of this paper is to exhibit the same functoriality with respect to smooth
link cobordisms for symplectic Khovanov homology as exists in the setting of Khovanov
homology. Recent work of Rezazadegan \cite{reza:quilts} has independently exhibited similar homomorphisms between symplectic Khovanov
homology groups (for tangles) corresponding to elementary cap, cup and saddle cobordisms.
He also exhibits some important structure, relevant to Conjecture~\ref{conjecture:khiskhsymp}, but
does not show that these homomorphisms give an invariant of link cobordisms.

\begin{theorem}
\label{theorem:main}
Let $\operatorname{\mathbf{Ab}_1^\pm}$ be the category of finitely generated, singly graded abelian
groups and graded homomorphisms, defined only up to an overall sign ambiguity.
There is a non-trivial
functor from $\operatorname{\mathbf{Cob}}$
to $\operatorname{\mathbf{Ab}_1^\pm}$, which maps a link $\mathfrak{L}$ to
the symplectic Khovanov homology $\skh ^*(\mathfrak{L})$.
\end{theorem}

In particular, some care is needed to define the symplectic Khovanov homology
up to canonical isomorphism for a general link in $\mathbb{R}^3$. In fact, the
same trick of defining symplectic Khovanov homology for crossing diagrams together
with a small perturbation of the link works here too.

One can consider Theorem~\ref{theorem:main} as a first step in proving
Conjecture~\ref{conjecture:khiskhsymp}, since Khovanov homology
is defined in terms of elementary cobordisms between unlinked unions of unknots
(and the homomorphisms corresponding to general cobordisms arise directly from these).

Combining the theorem with a simple generalisation of Wehrheim and Woodward's exact sequence \cite{wewo:les}
(one has to deal with non-compactness issues similar to those in Section~\ref{section:sfib})
one should already be able to exhibit skein exact triangles for
symplectic Khovanov homology. In particular, Theorem~\ref{theorem:main} should imply the existence of a spectral sequence
with $\mathbb{Z}/2$--coefficients from Khovanov homology to symplectic Khovanov homology
(cf. \cite{ozsz:ss}).

We also prove the less surprising, though previously unproven, result that the symplectic
Khovanov homology of an unlinked union of links splits as a tensor product at the level
of chain complexes, so is described by the K\"unneth isomorphism on the level of homology.
This
result was previously only shown in the case that one of the two links was an
unknot \cite[Proposition 56]{seidelsmith:khsymp}.

\begin{theorem}
\label{theorem:splitintroversion}
Let $\mathfrak{L}$ be an unlinked union of links $\mathfrak{L_1}$ and $\mathfrak{L_2}$.
Then $\skh ^*(\mathfrak{L})$ is the K\"unneth product of
$\skh ^*(\mathfrak{L_1})$ and $\skh^*(\mathfrak{L_2})$ (i.e. tensor product at the level of
chain complexes) up to an overall
grading shift.
\end{theorem}

The underlying motivation for the proofs of both results is in the geometry of
configuration spaces $\conf{2n}$ and $\confbar{2n}$. In the construction of
symplectic Khovanov homology, one uses the following correspondences to present
links in terms of these configuration spaces:
\begin{itemize}
 \item A braid joining a configuration $P\in\conf{2n}$ of points in the plane
to another configuration $Q$ corresponds to a path $P$ to $Q$ in the configuration space.
 \item Let $\gamma:[0,1]\to\confbar{2n}$ be a \emph{vanishing path}; that is, a
smooth path hitting $\confbar{2n}\setminus\conf{2n}$ transversely, only at
$\gamma(0)=(0,0,\mu_1,\ldots,\mu_{2n-2})$ such that the $\mu_i$ are pairwise distinct. Then
$\gamma$ corresponds to a $(2n-2,2n)$ tangle between configurations $\gamma(0)$ and $\gamma(1)$.
\end{itemize}

Suppose we have a map $u:\overline{\mathbb{D}}\to \confbar{2n}$
intersecting $\confbar{2n}\setminus\conf{2n}$ transversely only on $\mathbb{D}$ and only
in configurations $(0,0,\mu_1,\ldots,\mu_{2n-2})$, such that the $\mu_i$ are pairwise distinct.
Then from this one can construct a smooth braid cobordism from the trivial braid at $u(1)$ to
the braid described by $u(\partial{\overline{\mathbb{D}}})$. In fact, braid cobordisms
up to isotopy, correspond to isotopy classes of such maps. For details on this correspondence
see Section~\ref{section:cobordisms}.

With the further assumption that $u$ is a holomorphic embedding near where it intersects
$\confbar{2n}\setminus\conf{2n}$, there is a natural construction of a singular symplectic fibration over
$\overline{\mathbb{D}}$. These have fibres over $\pm 1$ which are the
the auxiliary symplectic manifolds $\sfibre{n}{u(\pm 1)}$, in which $\skh^*$ is defined as a Floer cohomology group.
In a manner motivated by Seidel's relative invariant of exact Lefschetz fibrations \cite{seidel:les},
one would then like to define morphisms between symplectic Khovanov homology groups by
counting holomorphic sections of these fibrations.

In fact, the above describes essentially the method of this paper. The fibrations are
all of a particularly nice form (exact Morse-Bott-Lefschetz fibrations) and relative
invariants can be defined in an analogous manner. However, essential non-compactness issues
(arising from non-compactness of the singular loci)
cause problems for convexity, gluing and even defining symplectic parallel transport in
these fibrations.

Section~\ref{section:sfib} develops general tools for
studying relative invariants in these fibrations. In particular, it
is also shown in general how to construct exact Morse-Bott-Lefschetz fibrations
together with relative invariants from maps of surfaces into singular holomorphic fibrations
of Stein manifolds.
Section~\ref{section:twistprod} then develops some specific tools for the calculation of Floer cohomology
and relative invariants of symplectic associated bundles
necessary for proof of Theorem~\ref{theorem:main} given in the later sections.

\subsection*{Acknowledgments}
The author would like to thank the anonymous referee who pointed out an error in an earlier version of the paper,
and also Ivan Smith and Dominic Joyce for important discussions relating to the correction.

\section {Floer cohomology and singular symplectic fibrations}
\label{section:sfib}

Given a pair $L$, $L^\prime$ of closed, connected Lagrangian submanifolds
of a symplectic manifold $M$ satisfying certain conditions, one can define
a relatively graded abelian group $HF(L,L^\prime)$. This is the Floer
cohomology, or in our case, the ``Lagrangian
intersection Floer cohomology''. The cohomology is obtained from a chain
complex made of formal sums (over $\mathbb{Z}$) of transverse intersections
of the Lagrangians and a differential defined by counting holomorphic
strips with boundary on the Lagrangians which interpolate between
intersections.

\begin{remark}
\label{remark:fh}
An example of conditions in which the Floer cohomology is defined is given
by the following:
\begin{itemize}
\item $M$ is a K\"ahler manifold on which the symplectic form is exact
and the underlying complex structure makes $M$ a Stein manifold
\item $c_1(M)=0$ and $H^1(M)=0$
\item $H_1(L)=H_1(L^\prime)=0$ 
\item $w_2(L)=w_2(L^\prime)=0$ (equivalently $L$, $L^\prime$ are spin)
\end{itemize}
These conditions are satisfied where Floer cohomology is used for the
definition of symplectic Khovanov homology in \cite{seidelsmith:khsymp}.
Unless otherwise mentioned these are the conditions under which we will 
use Floer cohomology in this paper.
\end{remark}

Floer cohomology is defined up to canonical isomorphism, even
when each Lagrangian is specified only up to compactly supported
Hamiltonian isotopy. These isotopies and canonical isomorphisms
are used to define the cohomology even when
the Lagrangians do not intersect transversely.

\begin{remark}
The condition that $L$, $L^\prime$ be spin is necessary only to use $\mathbb{Z}$
coefficients for $HF(L,L^\prime)$. Without it Floer cohomology is still defined with
$\mathbb{Z}/2$ coefficients (provided $L,L^\prime$ are orientable.

$L$, $L^\prime$ being spin implies the orientability of the moduli spaces of holomorphic
strips (see Lemma~22.11 of \cite{fooo}), the counting of which defines the differential.
The orientation gives a consistent choice of signs for this counting process.
\end{remark}

A symplectic vector bundle (cf. \cite{mcduff:sfib}) is a vector bundle
$E\to B$ with a
smooth choice of skew symmetric bilinear form on each fibre (i.e.\ a
section $\Omega$ of $\Lambda^2 E$), which is non-degenerate on each
fibre.
This is the local (first order) model for a symplectic fibration (where
$\Omega$ is instead a closed 2--form on the total space).

To be more precise, an \emph{exact symplectic fibration} is a manifold E
with corners and a smooth fibration $\func{\pi}{E}{B}$ equipped with an
exact 2--form
$\Omega = d\Theta$ on E whose restriction to fibres of E is a
symplectic form. We shall also require that the corners of E are
precisely the boundary points of the fibres over $\partial B$.

Non-degeneracy of $\Omega$ on the vertical tangent spaces
$TE^v=\ker{D\pi}$ means that we can define horizontal tangents to be
$TE^h=(\ker{D\pi})^{\perp_\Omega}$. This defines the \emph{symplectic
connection} and \emph{symplectic parallel transport} over any path $\gamma$ in the base.
As long as points do not flow under symplectic parallel transport off of
the boundary of $E$, the symplectic parallel transport defines maps between
the fibres over the start and end points of $\gamma$.
These maps are symplectomorphisms between the fibres.
Isotopic paths in the base yield parallel transport maps which
differ by exact symplectomorphisms.

In \cite{seidel:les} these fibrations are generalised to
\emph{exact Lefschetz fibrations} (over surfaces),
by allowing complex non-degenerate singularities of $\pi$.
The monodromy by parallel transport once
anticlockwise around such a singular value in the base is then a Dehn
twist $\sigma$ in the Lagrangian vanishing cycle associated to the
singular point. Take an exact Lefschetz fibration over the infinite
strip $\mathbb{R}\times[0,1]$ which has trivialised symplectic
parallel transport over the ends (giving well-defined fibres at
$\pm\infty$). One assigns to the fibre at $+\infty$ a pair of exact
Lagrangian submanifolds $L_{+\infty}^0$, $L_{+\infty}^1$.
Extending these by symplectic parallel transport over the boundaries
$\{0\}\times\mathbb{R}$, $\{1\}\times\mathbb{R}$ respectively, Seidel
defines a map from the Floer cohomology in the fibre at $-\infty$ to
that in the fibre at $+\infty$.

\[HF(L_{-\infty}^{top}=\sigma L_{+\infty}^{top},L_{-\infty}^{bottom})
\to HF(L_{+\infty}^{top},L_{+\infty}^{bottom})\]

\begin{figure}[h]
\centering\scalebox{0.6}{\input{Fibration1.pstex_t}}
\caption{A basic Lefschetz fibration over the infinite strip with one
singular fibre. The orientation indicated on the boundary is such as to
make the total monodromy a single positive Dehn twist $\sigma$
in the vanishing cycle. The fibres at $\pm\infty$ have been
identified by symplectic parallel transport along the lower boundary,
so the monodromy occurs entirely on the upper boundary.}
\label{fig:fibration1}
\end{figure}

\subsection{Exact Morse-Bott-Lefschetz fibrations}
\label{section:mblfibrations}

Working with $\skh $, we have a natural construction of singular
symplectic fibrations from braid cobordisms. The singularities that arise
in this construction have a slightly more general form than those of the
exact Lefschetz fibrations, considered in \cite{seidel:les}. In this section, I
describe the corresponding construction for exact
Morse-Bott-Lefschetz (MBL) fibrations over surfaces.

\begin{definition}
\label{definition:exactmbl}
An exact \emph{MBL--fibration} is
a collection $(E,\pi,\Omega,\Theta,J_0,B,j)$ such that:

\begin{enumerate}[(1)]
\item $E$ is a smooth, not necessarily compact, manifold with boundary $\partial E$.
\item $B$ is a Riemann surface with complex structure $j$, homeomorphic to
$\overline{\mathbb{D}}$ with finitely many boundary points removed.
\item $\func{\pi}{E}{B}$ is a smooth map with
$\partial E=\pi^{-1}(\partial B)$ and such that
$\func{\pi|_{\partial E}}{\partial E}{\partial B}$
is a smooth fibre bundle.
\item $\Omega=d\Theta$ is an exact 2--form on $E$, non-degenerate
on $TE^v:=\ker{D\pi}$ at every point in $E$.
\item $\pi$ has finitely many critical values, all in the interior of $B$.
\item $J_0$ is an almost complex structure defined on some subset of $E$ which contains
a neighbourhood of the set $\crit\pi$ of critical points and the complement $U$ of some fibrewise compact subset of $E$.
\item 
\label{condition:omegajcompat} $\pi$ is $(J_0,j)$--holomorphic and $\Omega(.,J_0.)|_{(TE^v)^{\otimes 2}}$
is everywhere symmetric and positive definite (where $J_0$ is defined).
\item $J_0$ preserves $TE^h$ on $U$.
\item $\Omega$ is a K\"ahler form for $J_0$ on some open neighbourhood of
$\crit{\pi}$.
\item $\crit\pi$ is smooth and the complex Hessian of $\pi$ is
non-degenerate on complex complements of $T\crit\pi$ in $TE$.
\end{enumerate}
\end{definition}

Seidel's exact Lefschetz fibrations have boundary in the fibre direction near which
there is a trivialisation of $E$ compatible with $\Omega$ and $\Theta$ (c.f. \cite{seidel:les}).
One cannot expect such trivialisations at boundaries to exist for exact MBL-fibrations, since
the singular locus can escape to infinity in a fibre. For this reason, we don't require
there to be trivialisations. We compensate for this by taking significantly more
care with convexity and gluing of exact MBL-fibrations. This is the main difficult content of
Section~\ref{section:flattening}.

We will consider exact MBL-fibrations with bases $B$ which are of a particular form.
Namely $B$ should be a Riemann surface with finite sets $I^\pm$ of \emph{ends} (see below),
not both empty.
The ends may be of two forms:

\begin{definition}[Striplike ends (cf. \cite{seidel:les})]
A \emph{striplike end} $e\in I^\pm$ of a surface $B$ is a
proper holomorphic embedding
\[\func{\gamma_e}{[0,\infty)\times[0,1]}{B}\]
(with the standard complex structure on $[0,\infty)\times[0,1]\subset\mathbb{C}$)
such that $\gamma_e^{-1}(\partial B)=[0,\infty)\times\{0,1\}$.

An exact MBL--fibration is \emph{trivial over the striplike end $e$}
if over the image of $\gamma_e$ it is non singular and isomorphic as an exact
symplectic fibration to $[0,\infty)\times[0;1]\times E_z$ for
some fibre $E_z$. Here one takes $\Omega$ and $\Theta$ pulled back
by the projection to $E_z$, and $J_0$ split as the sum of an almost complex structure
on the $E_z$--factor and the standard almost complex structure on the
$[0,\infty)\times[0,1]$--factor.
\end{definition}

\begin{definition}
A \emph{boundary marked point} $z\in \partial B$ together with a proper holomorphic embedding
\[\func{\gamma_e}{[0,\infty)\times[0,1]}{B\setminus z}\]
such that $\gamma_e^{-1}(\partial B)=[0,\infty)\times\{0,1\}$
and $\gamma_e(x,t)\to z$ as $x\to\infty$ may also be considered
an end. Exact MBL-fibrations are not required to be trivial over these ends.
\end{definition}

Ends given by boundary marked points can be viewed as striplike ends without
the trivialisation and striplike ends can be completed, by adding a single
\emph{fibre at infinity}, to give boundary marked points. Switching between these 
two settings will be important later on.

\begin{definition}
By the \emph{fibre at an end} $e$ of $B$ we mean:
\begin{itemize}
 \item the fibre over the boundary marked point
 \item the \emph{fibre at infinity} of a striplike end
\end{itemize}
\end{definition}

We require also, that the ends of $B$ be pairwise disjoint and that
the complement of the ends (i.e.\ of the images of the $\gamma_e$ and any boundary
marked points) be compact. This means in particular that
the boundary of $B$ with boundary marked points removed decomposes into as many
open intervals as there are ends.

The benefit of striplike ends (with accompanying trivialisations)
is that it is easy to \emph{compose} exact MBL--fibrations at striplike ends.
Namely, one forms the composite by gluing oppositely oriented, but
otherwise identical trivialisations of two separate exact
MBL--fibrations together.

In contrast, the benefit of boundary marked points is twofold. They arise more naturally (see
Section~\ref{section:flattening}) and holomorphic convexity in the fibre direction is easier to attain.

We now define what we mean by \emph{exact Lagrangian boundary conditions} for
an exact MBL-fibration $(E,\pi)$ over a surface $B$ with ends.

\begin{definition}
\label{def:lagboundary}

Let $P$ be the set of boundary marked points of $B$.
An \emph{exact Lagrangian boundary condition} on $(E,\pi)$
is a subbundle $Q$ of $E$ over $\partial B\setminus P$
together with a function $\func{K_Q}{Q}{\mathbb{R}}$ such that:

\begin{enumerate}
 \item $\Omega|_Q=0$
 \item for any $z\in\partial B$ the restriction $(Q_z,K_Q|_{Q_z})$
is a closed, connected exact Lagrangian submanifold of $E_z$ (i.e.\ a Lagrangian
submanifold such that also $d(K_Q|_{Q_z})=\Theta|_{Q_z}$).
 \item $(Q_z,K_Q|_{Q_z})$ extends smoothly, along each component of $\partial B\setminus P$,
to the fibres over boundary marked points. This extension is allowed to depend on the
side from which one approaches a boundary marked point.
 \item $(Q_z,K_Q|_{Q_z})$ is constant w.r.t. trivialisations over the
striplike ends
\end{enumerate}

\end{definition}

Given an exact MBL--fibration over a surface with striplike ends, one can construct an exact MBL--fibration
over a surface with boundary marked points. Namely, the base can be compactified by
adding a single point \emph{at infinity} at each end. One then adds to the total space the fibres at infinity.

Condition (1) implies that $Q$ is preserved by symplectic parallel transport
over $\partial B$ and that $\Theta|_Q=dK_Q+\pi^*\kappa_Q$ for some
$\kappa_Q\in\Omega^1(\partial B)$ (cf. \cite{seidel:les} Lemma~1.3). Condition (2) and
triviality of the striplike ends gives $\kappa_Q=0$ there. In fact $\kappa_Q=0$
whenever symplectic parallel transport preserves $K_Q$.

$Q$ specifies a pair of Lagrangian submanifolds in the fibre over each marked
point and in each fibre at infinity (i.e. `in the fibre at each end').
We will refer to $Q$ as \emph{transverse}, if these pairs of Lagrangians are each transverse.

\begin{remark}
Assume we are given a choice of exact Lagrangian submanifold $(L,K_L)$
in the fibre at infinity or fibre over a boundary marked point at one
end of each edge of $\partial B$.
Then either symplectic parallel transport maps restricted to one of these Lagrangians fail to be
defined over the entire edge on which it lies, or else they are defined
and the condition $\kappa_Q=0$ uniquely specifies a Lagrangian boundary condition.
This is the manner in which all exact Lagrangian boundary conditions in this paper are constructed.
\end{remark}

We will be interested in counting compact moduli spaces of holomorphic sections with boundary
in $Q$, so it makes sense to require holomorphic convexity of a neighbourhood in
$E$ containing $Q$ as follows:

\begin{definition}
\label{def:enclosedlag}
An \emph{enclosed exact Lagrangian boundary condition} $(Q,\rho,R)$ is an exact Lagrangian boundary
condition $Q$, together with a smooth map $\func{\rho}{E}{\mathbb{R}^{\geq 0}}$ and $R>0$,
such that:
\begin{enumerate}[(i)]
 \item $\rho$ splits w.r.t. the trivialisations over all striplike ends
 \item $\rho^{-1}\left[0,R\right]$ is fibrewise compact and $\rho^{-1}[0,R)$ contains $Q$
 \item $\exists\varepsilon>0$ such that on $\rho^{-1}\left(R-\varepsilon,R\right)$
  \begin{itemize}
  \item $J_0$ is defined

  \item $\rho$ is subharmonic w.r.t. $J_0$ and plurisubharmonic on fibres
  \end{itemize}
\end{enumerate}
We refer to the pair $(\rho,R)$ as an \emph{enclosure}.
\end{definition}

\begin{remark}
We do not require $J_0$ to be integrable. $\func{\rho}{E}{\mathbb{R}^{\geq 0}}$ is
defined to be subharmonic or plurisubharmonic when $-d(d\rho\circ J_0)(.,J_0.)$ is
$\geq 0$ or $>0$, respectively. Given any $(i,J)_0$--holomorphic map from a subset of $\mathbb{C}$
into $E$, the composite $\rho\circ u$ is subharmonic or plurisubharmonic in the 
conventional sense.
\end{remark}

Suppose we extend $J_0$ to an almost complex structure on $E$ which 
makes the projection $\pi$ holomorphic.
Then subharmonicity of $\rho$ gives a maximum principle for holomorphic sections
of $E$. Namely, they cannot have a maximum of $\rho$ in the range $(R-\epsilon,R)$.
This ensures that families of holomorphic sections
which are confined to $\rho^{-1}\left[0,R-\epsilon\right)$
cannot degenerate to sections which escape this region.
The choice of enclosure is important, since in general there will exist holomorphic
sections which leave the region $\rho^{-1}\left[0,R\right)$.

It will often be necessary to change enclosures other than just by isotopy
through enclosures. For this, we will need a notion of equivalence of enclosures.

\begin{definition}
\label{def:equivenc}
 We say that two enclosures $(\rho,R)$ and $(\sigma,S)$, for a given exact Lagrangian
boundary condition $Q$, are \emph{equivalent} if for some $R^\prime < R$ we have:
\begin{itemize}
 \item $\rho^{-1}[0,R^\prime)\subset\sigma^{-1}[0,S)\subset\rho^{-1}[0,R)$
 \item $\rho$ is subharmonic on $\rho^{-1}[R^\prime,R]$
\end{itemize}
or also if they can be related by a sequence of such comparisons, either way round.
For simplicity, we require that $J_0$ is fixed throughout these comparisons.
\end{definition}

In particular this makes enclosures related by isotopy of enclosures equivalent.

As defined earlier, an exact MBL--fibration carries a two form which is not necessarily
symplectic on the total space. It is important to observe that this is
more flexible than requiring $\Omega$ to be an exact symplectic structure
on the total space (within a particular enclosure),
but it is no weaker for our purposes.

\begin{lemma}
\label{lemma:nondeg}
Let $(E,\pi,\Omega,\Theta,J_0,B,j)$ be an exact MBL--fibration.
Given any enclosure $(\rho,R)$,
there is a canonical choice of isotopy class
of exact symplectic form $\Omega^\prime$ on the total space of $\rho^{-1}[0,R]$ with:
\begin{itemize}
 \item the same restriction to $TE^v$ as $\Omega$
 \item the same symplectic connection as $\Omega$
\end{itemize}
\end{lemma}

\begin{proof}
Let $\omega$ be any exact volume form on $B$ compatible with $j$ (which in
2 dimensions means only that it induces the correct orientation).
$\Omega$ restricts to horizontal vectors $T_zE^h$ for $z\in E\setminus\crit{\pi}$
as multiple of $\pi^*\omega$ by a value $f(z)$. Non-degeneracy of $\Omega$ at $z$ is
equivalent to $f(z)$ being non-zero. Furthermore,
$\Omega$ is non-degenerate at all $z$ in a neighbourhood of $\crit{\pi}$ by definition.

$\Omega$ is K\"ahler on some open neighbourhood $V$ of $\crit{\pi}$, so
$f$ is strictly positive on $V\setminus\crit{\pi}$. The region
$(E\setminus V)\cap \rho^{-1}[0,R)$ is fibrewise compact and becomes compact when one
extends to the fibres at infinity. $f$ is smooth away from $V$, so
is bounded below. $\Omega^\prime:=k\pi^*\omega+\Omega$
satisfies the necessary axioms for an exact MBL--fibration,
and for large enough $k\in\mathbb{R}$ it is non-degenerate on $\rho^{-1}[0,R)$, as required in the lemma.
\end{proof}

\begin{remark}
Lemma~\ref{lemma:nondeg} allows one to perform standard holomorphic disc counting
constructions within the enclosure $(\rho,R)$ to get invariants of exact Lagrangian
boundary conditions up to Lagrangian isotopy in a similar manner to the way one
defines Floer cohomology in exact symplectic manifolds. See Section~\ref{section:holomorphic}
for more details.
\end{remark}

\subsection{Constructing MBL--fibrations from Stein manifolds}
\label{section:flattening}

Let $E$ be a Stein manifold with plurisubharmonic function $\rho\geq 0$ and
exact K\"ahler form $\Omega:=d\Theta:=-d(d\rho\circ i)$. Suppose we have a singular
holomorphic fibration $\func{\pi}{E}{N}$ over a complex manifold $N$ and a
smooth map $\func{u}{B}{N}$, for some simply connected Riemann surface $B$ with
striplike ends or marked points on the boundary. When the singularities
of $u^*E$ are of Morse-Bott-Lefschetz type and $u$ is holomorphic near singular points,
one can view $u^*E$ naturally as an exact MBL--fibration.

This section gives a more detailed construction of exact MBL--fibrations in the manner described above.
Most of the content deals with the problems of convexity (for
defining enclosed Lagrangian boundary conditions consistently) and composition of fibrations
(by \emph{gluing} trivialisations over striplike ends).
Without these trivialisations, the gluing construction for composing the
relative invariants (cf. Section~\ref{section:holomorphic}) would be difficult.

\begin{definition}
\label{def:mblmodel}
We say a singular value $p\in N$ of $\pi$ is MBL if
there is a neighbourhood of any point in the singular locus $\crit{\pi,p}$ fitting into the following
commutative diagram

\[
\bfig
 \node a(0,500)[E]
 \node b(2000,500)[\crit{\pi,p}\times\mathbb{C}^k\times\mathbb{C}^n]
 \node c(0,0)[N]
 \node d(2000,0)[\mathbb{C}\times\mathbb{C}^n]
 \arrow[b`a;\text{local near }\crit{\pi,p}]
 \arrow[a`c;\pi]
 \arrow|r|[b`d;(\mathbf{x},\mathbf{y},\mathbf{z})\mapsto(\sum y_i^2,\mathbf{z})]
 \arrow[d`c;\text{local biholomorphism near }(0,\mathbf{0})]
\efig
\]

Here the top map is a local embedding of exact K\"ahler manifolds defined on a neighbourhood
of $\crit{\pi,p}$. However, if we only require it to be holomorphic, we can always deform the exact K\"ahler
structure on $E$ to make it an embedding of exact K\"ahler manifolds.

We shall denote the set of such critical values by $\critmbl{\pi}$. It is a submanifold of
codimension 2. The open set of regular values we denote by $N^{reg}$.
\end{definition}

This definition can equivalently be expressed in terms of smoothness of
critical loci in $E$ and $N$ and non-degeneracy of the Hessian of $\pi$
on complements to $T\crit{\pi}\cap\ker(D\pi)$ within $\ker(D\pi)$.

\begin{definition}
\label{def:admismap}
Let $(B,j)$ be a simply connected Riemann surface $(B,j)$
with striplike ends or marked points on its boundary. Let
$\func{u}{B}{N^{reg} \cup \critmbl{\pi}}$
be a smooth map with the properties that:
\begin{itemize}
\item $u$ is a constant map on each of the striplike ends
\item $u(\partial B)\subset N^{reg}$
\item $u$ is transverse to $\critmbl{\pi}$ and holomorphic near it.
\end{itemize}
We call such a map \emph{admissible}.
We shall refer to $u^{-1}(\critmbl{\pi})\subset(B,j)$ as the \emph{singular values},
since these are the singular values of the pullback fibration ${u^*E}$.
\end{definition}

Given an admissible map $u$, we consider the pullback fibration $\func{\pi}{u^*E}{B}$
equipped with $u^*\Omega$, $u^*\Theta$, $u^*\rho$. There is also a natural choice of almost complex
structure $\tilde{J}$\label{compstr} which agrees with $u^*J$ where $u$ is a holomorphic immersion.
Namely, one takes $u^*J$ on vertical tangents and $i$ on horizontal tangents with respect to
the symplectic connection induced by $u^*\Omega$ (and $u^*J$ where the
fibration is singular). These choices make
$(u^*M,\pi,u^*\Omega,u^*\Theta,\tilde{J},B,j)$ an exact MBL--fibration.

Furthermore, the isotopy class of $u$ through admissible maps defines an isotopy class of exact
MBL--fibrations.

\begin{lemma}
\label{lemma:pullback}
Given $u$, $B$ and $j$ as above up to smooth homotopy of $u$ and deformation of $j$
we have an exact MBL--fibration defined up to smooth deformation of the parameters.
\end{lemma}

The rest of this section deals with the deformations needed to construct
Lagrangian boundary conditions and then ensure holomorphic convexity of a
surrounding region (thus making an enclosed Lagrangian boundary condition).
The approach is to approximate $B$ by a tree of embedded
holomorphic discs connected at marked points and solve the same problem for
embedded holomorphic discs.

First we consider the model case, where $u$ is already a holomorphic embedding and
$B$ is the closed unit disc $\overline{\mathbb{D}}$
with finitely many (but at least one) marked points on the boundary. In this case
the exhausting, plurisubharmonic function $\func{\rho}{E}{[0,\infty)}$
pulls back to a fibrewise-exhausting, plurisubharmonic function on $u^*E$.

\begin{lemma}
\label{lemma:sptdef}
 Let $(E,\pi,\Omega,\Theta,J_0,\overline{\mathbb{D}},j)$ be an exact MBL--fibration
over the closed unit disc.
Assume further that we have $\func{\rho}{E}{\left[0,\infty\right)}$ exhausting (fibrewise), and
plurisubharmonic where $J_0$ is defined, such that $d\rho=\Theta\circ J_0$.

For any $l\in\mathbb{R}$, one can deform $\Theta$ (without changing the restriction of $\Omega$ to fibres)
inside some level set $\rho_{max}$ of $\rho$ such that symplectic parallel transport flow lines
over paths of length at most $l$ in $\partial \overline{\mathbb{D}}$ do not leave
$\rho^{-1}\left[0,\rho_{max}\right]$.

The deformation occurs only over a small open neighbourhood of
$\partial\overline{\mathbb{D}}$ and is well-defined up to isotopy through such
deformations of $\Theta$. Furthermore, the deformation may be chosen to have support
disjoint from any particular compact set.
\end{lemma}

\begin{proof}
 Let $A$ be a small annular neighbourhood of $\partial\overline{\mathbb{D}}$ not
containing any critical values of $\pi$. Let $\rho_0$ be large enough such that for all
$z\in A$ all critical values of $\rho|_A$ are less than $\rho$. This implies that
$\rho|_{\pi^{-1}(A)}^{-1}(\rho_0)$ is a smooth fibration over $A$ with compact fibre $C$
and hence carries a flat
connection, well-defined up to isotopy. Choose one. It gives a trivialisation over any
small open neighbourhood $U\subset A$ in the base of the form $\func{\proj{U}}{C\times U}{U}$.
Extending this in the positive time direction by the fibrewise Liouville flow we have
a trivialisation $\func{\proj{U}}{C\times\left[0,\infty\right)\times U}{U}$.

On any fibre $E_z$ the form $\Theta$ restricts to $C\times \{0\}\times{z}$ as
a contact form $\Theta_{0,z}$ on $C$ and restricts to the whole fibre as $e^y\Theta_{0,z}$
(here $y$ is the coordinate on $\left[0,\infty\right)$). We define
$\Theta_C$ on the trivialisation $C\times\left[0,\infty\right)\times U$ to have
this same restriction to fibres and to vanish on $TU$.
In particular $\Theta_C|_{E_z}=\Theta|_{E_z}$.

Let $R$ be the Reeb vector field on $C$ for the contact form $\Theta_{0,z}$.
We can split $TE$ in the trivialised region
as $\mathbb{R}R\oplus\ker(\Theta_{0,z}) \oplus\mathbb{R}\dd{}{y}\oplus TU$.
Given a vector $H$ in $TU$ the symplectic parallel transport vector w.r.t. the 2--form
$d(e^y\Theta_C)$ over it is of the form $w_RR+W_{con}+w_y\dd{}{y}+H$ in that splitting.
It has the defining property that for any $v_RR+V_{con}+v_y\dd{}{y}$ we have:

\begin{eqnarray*}
 0&=&d(e^y\Theta_C) \left(w_RR+W_{con}+w_y\dd{}{y}+H,v_RR+V_{con}+v_y\dd{}{y}\right)\\
  &=&\left(e^yd\Theta_C+e^ydy\wedge\Theta_C\right)\left(w_RR+W_{con}+w_y\dd{}{y}+H,v_RR+V_{con}+v_y\dd{}{y}\right) \\
  &=&e^y\left[-H(\Theta_C)(v_RR+V_{con})+d(\Theta_{0,z})(V_{con},W_{con})+v_yw_R-v_Rw_y\right]
\end{eqnarray*}

Setting $v_R=1$, $V_{con},v_y=0$ gives $w_y=-H(\Theta_C)(R)$ which by compactness has
a finite maximum over $A$. i.e.\ the velocity of this symplectic parallel transport in
the $y$--direction is bounded over compact subsets of the base. This controls the
symplectic parallel transport flow lines in large enough regions, so in particular proves the
lemma for any deformed $\Theta$ which equals
$e^y\Theta_C$ on $\rho|_{\pi^{-1}(A)}^{-1} \left[ \rho_0+1,K-1 \right)$ for large enough $K$.

Now we define a deformation $\tilde{\Theta}=g\Theta+(1-g)\Theta_C$ with a bump function
$g$ identically equal to 1 on $\rho|_{\pi^{-1}(A)}^{-1} \left[\rho_0+1,K-1\right)$
and zero on the complement of $\rho|_{\pi^{-1}(A)}^{-1} \left[\rho_0,K\right)$.
This is the required deformation of $\Theta$ to prove the lemma.

\end{proof}

\begin{remark}
\label{remark:sptdef}
 Suppose that $\crit{\pi}$ is compact. For example this is the case when the singularities are
actually of Lefschetz type. Then, by the same technique, one can contain the symplectic 
parallel transport over paths of length at most $l$ in $\overline{\mathbb{D}}$, not just in the boundary.
\end{remark}

\begin{remark}
\label{remark:pathspt}
 Suppose we are really just interested in defining symplectic parallel transport maps over a path
$\func{\gamma}{[0,1]}{N^{reg}}$, then one can run the same argument in the pullback fibration
$\gamma^*E$. This gives symplectomorphisms $E_{\gamma(0)}\to E_{\gamma(1)}$ defined on any compact
subset of $E_{\gamma(0)}$ and well-defined up to isotopy within the class of symplectic embeddings
(or also inclusion, should one enlarge the choice of compact subset).
\end{remark}

The lemma above allows us to define Lagrangian boundary conditions on such a fibration simply
by specifying a Lagrangian in a single fibre on each interval of $\partial\overline{\mathbb{D}}$
and extending to the rest of the interval by symplectic parallel transport. Call this Lagrangian
boundary condition $Q$. Furthermore, the region
of deformation is contained within a finite level set of $\rho$, so for all large enough
$R\in\mathbb{R}^{\geq 0}$ the collection $(Q,\rho,R)$ is an enclosed Lagrangian boundary condition.
If the Lagrangians in the construction are chosen to be exact, then $Q$ is also exact.

Given a surface with striplike ends mapping to $\overline{\mathbb{D}}$, with edges mapping
monotonically to $\partial\overline{\mathbb{D}}$, symplectic parallel transport respects the
pullback. Hence, we can similarly control symplectic parallel transport of compact Lagrangian boundary conditions
specified in the fibre at infinity over one end of each edge. However, it is not so easy to show
these are enclosed.

We will now construct enclosures containing these Lagrangian boundary conditions
in a model case where we have deformed the previous fibration over $\overline{\mathbb{D}}$ to
have a base with striplike ends. The enclosures constructed will be defined as
deformations of an enclosure $(\rho,R)$ performed together with the deformation which
forms the striplike ends. Furthermore, the enclosures will be compatible with the trivialisations
of $E$ over the striplike ends, so will be compatible with the construction of gluing
fibrations over striplike ends.

Label the marked points $\{z_1,\ldots,z_n\}\subset\partial\overline{\mathbb{D}}$.
A neighbourhood of each
of these marked points becomes a striplike end under the appropriate coordinate change. Namely,
we view such a neighbourhood (holomorphically) as a neighbourhood of $0$ in the upper half plane $\overline{\mathbb{H}}$
(well-defined up to rescaling of $\overline{\mathbb{H}}$) which corresponds to a strip by the map $z\mapsto \log{z}$.
Without loss of generality this model is valid on $\{z\in\overline{\mathbb{H}} : \norm{z}\leq 3\}$ and furthermore
this region contains only regular values of $\pi$.

Locally near each $z_i$ we take $\widetilde{\mathbb{D}}\to\overline{\mathbb{D}}$ to be the identity away from the
$z_i$ and near them to be given by:

\begin{eqnarray*}
 u:& \widetilde{\mathbb{H}}\to &  \overline{\mathbb{H}} \\
& re^{i\theta}\longmapsto & h(r)e^{i\theta} \\
\end{eqnarray*}

Here $\widetilde{\mathbb{H}}$ is $\overline{\mathbb{H}}\setminus{0}$ with the
standard complex structure and
$\func{h}{\mathbb{R}^{\geq 0}}{\mathbb{R}^{\geq 0}}$ is a smooth function
such that:

\begin{itemize}
 \item $h(r)=0$ for $r\leq 1$ only
 \item $h(r)=r$ for $r\geq 2$
 \item $h$ is increasing
\end{itemize}

\begin{figure}[h]
\centering\scalebox{1}{\input{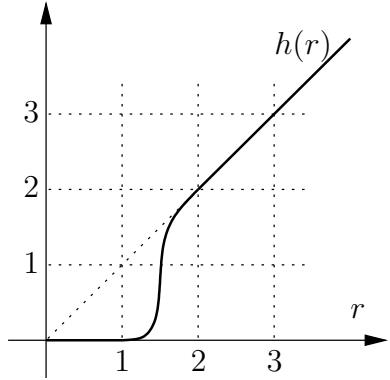}}
\caption{The function $h$.}
\label{fig:theamazingh}
\end{figure}

We now pullback $(E,\pi)$ to $(\widetilde{E},\tilde{\pi})$ which is an exact MBL--fibration over
a surface with striplike ends.

\[
\bfig
 \node a(0,500)[\widetilde{E}]
 \node b(1000,500)[E]
 \node c(0,0)[\widetilde{\mathbb{D}}]
 \node d(1000,0)[\overline{\mathbb{D}}]
 \arrow[a`b;\tilde{u}]
 \arrow[a`c;\tilde{\pi}]
 \arrow|r|[b`d;\pi]
 \arrow[c`d;u]
\efig
\]

\begin{proposition}
\label{proposition:enclosure}
Consider $(\widetilde{E},\tilde{\pi})$ as above. Suppose we have a compact connected Lagrangian
submanifold in the fibre at infinity over one end of each edge of the base. Then we can construct
some $\tilde{\rho}$ and deformation
of the exact MBL--fibration, such that:
\begin{itemize}
 \item $L$ extends by symplectic parallel transport to a Lagrangian boundary condition $Q$
 \item $(Q,\tilde{\rho})$ is a well-defined enclosed Lagrangian boundary condition.
\end{itemize}
Furthermore the resulting $(\widetilde{E},\tilde{\pi})$ and $(Q,\tilde{\rho})$ are well-defined up
to deformation through such choices.
\end{proposition}

 It should be noted here that we may, without loss of generality, use Lemma~\ref{lemma:sptdef}
to deform $E$ within some finite level set $\rho_{min}$ of $\rho$ to ensure that
$Q$ is defined on $\widetilde{E}$ and contained within $u^{-1}\left(\rho^{-1}\left[0,\rho_{min}\right)\right)$.
For the rest of the proof of the proposition we will deform only $u$ and we shall
define a $\tilde{\rho}$ with a convex level set contained in
$u^{-1}\left(\rho^{-1}\left[\rho_{min},\rho_{max}\right]\right)$ for any $\rho_{max}>\rho_{min}$.

We will see first how far $u^*\rho$ is from making $Q$ an enclosed Lagrangian boundary condition.

Let $\tilde{J}$ be the complex structure on $\widetilde{E}$ (see page~\pageref{compstr}) and 
$\tilde{u}^*J$ be the pullback of the complex structure from $E$ where this is defined.
Similarly we have the standard complex structure $\tilde{j}$ on $\widetilde{\mathbb{D}}$
and also the pullback $u^*j$ from $\overline{\mathbb{D}}$ where $u$ is an immersion.
Let $\tilde{\pi}^*\tilde{j}$, $\tilde{\pi}^*(u^*j)$ be the horizontal lifts
$D\tilde{\pi}|_{T\widetilde{E}^h}^{-1}\circ \tilde{j}\circ D\tilde{\pi}$ and
$D\tilde{\pi}|_{T\widetilde{E}^h}^{-1}\circ u^*j\circ D\tilde{\pi}$ respectively.
It should be noted here that $\tilde{J}$ and $J$ agree on $T\widetilde{E}^v$ and that
$\tilde{J}-u^*J=\tilde{\pi}^*\tilde{j} - \tilde{\pi}^*(u^*j)$.

Where $u$ is holomorphic these complex structures agree so $u^*\rho$ is plurisubharmonic. Also
where $u$ is locally constant $-d(d(u^*\rho)\circ\tilde{J})$ splits as $u^*\Omega$ on
$T\widetilde{E}^v$ and zero horizontally, so $u^*\rho$ is subharmonic. So far so good. The difficulty arises
in dealing with the region
$\frac{1}{t}\leq r\leq \frac{2}{t}$ in $\widetilde{\mathbb{H}}$. Here we have:

\begin{eqnarray*}
 -d(d(u^*\rho)\circ\tilde{J}) &= &-d(d(u^*\rho)\circ(\tilde{u}^*J+\tilde{\pi}^*(\tilde{j}-u^*j))) \\
  &= &\tilde{u}^*\Omega|_{(T\widetilde{E}^v)^{\otimes 2}}
      +\tilde{u}^*\Omega|_{(T\widetilde{E}^h)^{\otimes 2}}
      -d(d(u^*\rho)\circ(\tilde{\pi}^*(\tilde{j}-u^*j))
\end{eqnarray*}

Applied to pairs of vectors $(V,\tilde{J}V)$ the first two terms are positive semi-definite
(for the second we used that $u$ is nowhere orientation reversing). The third term may not be,
but it evaluates to zero on $(T\widetilde{E}^v)^{\otimes 2}$ and we will show how to adjust
$-d(d(u^*\rho)\circ\tilde{J})$ by adding a pull back by $\pi$ of a certain functional
$\widetilde{\mathbb{D}}\to\mathbb{R}$ to $\rho$ to
achieve positive semi-definiteness everywhere.

Let $\tilde{R},\tilde{\Theta},R,\Theta$ be horizontal lifts of the vector fields
$\dd{}{r},\dd{}{\theta}$ to $\widetilde{E},E$ respectively. Then at any point $p$ we have:
\begin{eqnarray*}
 \tilde{\pi}^*(\tilde{j}-u^*j): &r\tilde{R}\longmapsto \frac{h(r)-rh^\prime(r)}{h(r)}\tilde{\Theta} \\
    &\tilde{\Theta}\longmapsto\frac{h(r)-rh^\prime(r)}{rh^\prime(r)}r\tilde{R}
\end{eqnarray*}
and also:
\begin{eqnarray*}
 d(\tilde{u}^*\rho)|_{T_p\widetilde{E}^h} &=
        &\left(R_{\tilde{u}(p)}(\rho)(\pi^*dr)_{\tilde{u}(p)}
        +\Theta_{\tilde{u}(p)}(\rho)(\pi^*d\theta)_{\tilde{u}(p)}\right)
        \circ D\tilde{u}_p|_{T_p\widetilde{E}^h} \\
        &=
        &R_{\tilde{u}(p)}(\rho)rh^\prime(r)\left(\tilde{\pi}^*\frac{dr}{r}\right)_p
        +\Theta_{\tilde{u}(p)}(\rho)\left(\tilde{\pi}^*d\theta\right)_p
\end{eqnarray*}

Composing the functions defined above yields:
\begin{eqnarray*}
d(\tilde{u}^*\rho)\circ\tilde{\pi}^*(\tilde{j}-u^*j)_p
	&= &R_{\tilde{u}(p)}(\rho)(h(r)-rh^\prime(r))(\tilde{\pi}^*d\theta)_p \\
	&&+\Theta_{\tilde{u}(p)}(\rho)\frac{h(r)-rh^\prime(r)}{h(r)}\left(\tilde{\pi}^*\frac{dr}{r}\right)_p
\end{eqnarray*}

This gives us the following expression for the term of the expansion
of $-d(d(u^*\rho)\circ\tilde{J})$ which was potentially not positive-semidefinite (see above).
\begin{eqnarray*}
 -d(d(\tilde{u}^*\rho)\circ\tilde{\pi}^*(\tilde{j}-u^*j))_p
	&=
	&R_{\tilde{u}(p)}(\rho)r^2h^{\prime\prime}(r)\tilde{\pi}^*\left(\frac{dr}{r}\wedge d\theta\right)_p \\
	&&+R_{\tilde{u}(p)}(R(\rho))rh^\prime(r)(rh^\prime(r)-h(r))
		\tilde{\pi}^*\left(\frac{dr}{r}\wedge d\theta\right)_p \\
	&&-\Theta_{\tilde{u}(p)}(\Theta(\rho)\frac{rh^\prime(r)-h(r)}{h(r)}
			\tilde{\pi}^*\left(\frac{dr}{r}\wedge d\theta\right)_p \\
	&&+d_v(R_{\tilde{u}(p)}(\rho))\wedge(rh^\prime(r)-h(r))\tilde{\pi}^*\left(d\theta\right)_p \\
	&&+d_v(\Theta_{\tilde{u}(p)}(\rho))\wedge\frac{rh^\prime(r)-h(r)}{h(r)}\tilde{\pi}^*\left(\frac{dr}{r}\right)_p \\
\end{eqnarray*}

Here $d_v$ is the differential $d$ evaluated only in the fibre directions.

We do not yet have enough control over these summands, so we consider a one parameter family
of maps $u$ converging to the identity map on $\widetilde{H}$. These are defined by replacing
the function $h$ with the family of functions
\[h_t(r)=\frac{h(tr)}{t}\] for $t\in\left[1,\infty\right)$.

Considering the dependence on $t$ one now has:
\begin{eqnarray*}
 -d(d(\tilde{u}^*\rho)\circ\tilde{\pi}^*(\tilde{j}-u^*j))_p
	&=
	&\frac{1}{t}R_{\tilde{u}(p)}(\rho)(rt)^2h^{\prime\prime}(rt)\tilde{\pi}^*\left(\frac{dr}{r}\wedge d\theta\right)_p \\
	&&+\frac{1}{t}R_{\tilde{u}(p)}(R(\rho))(rt)h^\prime(rt)((rt)h^\prime(rt)-h(rt))
		\tilde{\pi}^*\left(\frac{dr}{r}\wedge d\theta\right)_p \\
	&&-\Theta_{\tilde{u}(p)}(\Theta(\rho)\frac{(rt)h^\prime(rt)-h(rt)}{h(rt)}
			\tilde{\pi}^*\left(\frac{dr}{r}\wedge d\theta\right)_p \\
	&&+\frac{1}{t}d_v(R_{\tilde{u}(p)}(\rho))\wedge((rt)h^\prime(rt)-h(rt))\tilde{\pi}^*\left(d\theta\right)_p \\
	&&+d_v(\Theta_{\tilde{u}(p)}(\rho))\wedge\frac{(rt)h^\prime(rt)-h(rt)}{h(rt)}\tilde{\pi}^*\left(\frac{dr}{r}\right)_p
\end{eqnarray*}

With this we can now describe how \emph{small} $-d(d(u^*\rho)\circ\tilde{\pi}^*(\tilde{j}-u^*j))$ is
in terms of $t$. To do this we define $\vecnorm{V}$ for $V\in T\widetilde{E}^v$ to be $\tilde{u}^*\Omega(V,\tilde{J}V)$,
i.e.\ the pullback of the metric on fibres of $E$.

\begin{lemma}
\label{lemma:flatenclosures}
 Given any $\rho_{max}>\rho_{min}$ there is some constant $K>0$
together with a smooth functional $\alpha\in C^\infty(\widetilde{E})$
and one-forms $\beta,\gamma$ on $\widetilde{E}$ supported over the model neighbourhood
$\{z\in \widetilde{\mathbb{H}}:\norm{z}\leq3\}$ such that:
\[-d(d(\tilde{u}^*\rho)\circ\tilde{\pi}^*(\tilde{j}-u^*j))=\alpha\tilde{\pi}^*\left(\frac{dr}{r}\wedge d\theta\right)
	+\beta\wedge\tilde{\pi}^*\left(d\theta\right)
	+\gamma\wedge\tilde{\pi}^*\left(\frac{dr}{r}\right)
\]
and
\begin{itemize}
 \item $\beta, \gamma$ evaluate to zero horizontally
 \item $\alpha,\beta,\gamma\equiv 0$ where $r\not\in [\frac{1}{t},\frac{2}{t}]$
 \item $\norm{\alpha}\leq \frac{K}{t}$ on $\rho^{-1}[0,\rho_{max})$
 \item For any vertical tangent $V\in T\widetilde{E}^v$ we have
$\norm{\beta(V)},\norm{\gamma(V)}\leq\frac{K}{2t}\vecnorm{V}$ on $\rho^{-1}[0,\rho_{max})$
\end{itemize}
\end{lemma}

\begin{proof}
 We examine the various summands of
\[-d(d(\tilde{u}^*\rho)\circ\tilde{\pi}^*(\tilde{j}-u^*j))_p\]
as described above.
The expressions in terms of $h$ and $rt$ are all smooth as functions of $rt$ and vanish
for $1\leq rt\leq2$, so must be bounded independently of $t$. By compactness $R_{\tilde{u}(p)}(\rho)$ and
$R_{\tilde{u}(p)}(R(\rho))$ have bounds independent of $t$. Similarly for $d_v(R_{\tilde{u}(p)}(\rho))$.

We are interested only in the region $r\in [\frac{1}{t},\frac{2}{t}]$ and
$u(r,\theta)=(h_t(r),\theta)$, so we restrict attention to $\tilde{u}^*\rho$ in fibres
where $r\leq\frac{2}{t}$.

By compactness $\Theta_{\tilde{u}(p)}(\Theta(\rho))$ and
$d_v(\Theta_{\tilde{u}(p)}(\rho))$ are bounded on $\rho^{-1}[0,\rho_{max})$
where $r\leq\frac{2}{t}$ and take value $0$ in the fibre over $0$,
so vanish to first order as $r\rightarrow 0$. Each summand of
$-d(d(\tilde{u}^*\rho)\circ\tilde{\pi}^*(\tilde{j}-u^*j))_p$
is a product of either of these two terms, or $\frac{1}{t}$ with bounded terms, hence the result.
\end{proof}

\begin{corollary}
\label{corollary:correction}
Let $C>\frac{K^2+K}{t}$, then over the region $\frac{1}{t}\leq r\leq\frac{2}{t}$ in $\widetilde{\mathbb{H}}$
\[\omega:=-d(d(\tilde{u}^*\rho)\circ\tilde{J})+\tilde{\pi}^*\left(C~\frac{dr}{r}\!\wedge\! d\theta\right)\]
gives a positive semi-definite quadratic form on $T\widetilde{E}$ when applied to pairs of vectors
$(X,\tilde{J}X)$.
\end{corollary}

\begin{proof}
 Split $X=V+H$ into vertical and horizontal components respectively. We will write $\vecnorm{H}$ for the
metric $\tilde{\pi}^*\left(\frac{dr}{r}\wedge d\theta\right)(\_,\tilde{J}\_)$ on $T\widetilde{E}^h$.

\begin{eqnarray*}
\omega(X,\tilde{J}X)
	&= &\Omega(V,\tilde{J}V)+\Omega(H,\tilde{J}H) \\
	&&+(C+\alpha)\tilde{\pi}^*\left(\frac{dr}{r}\wedge d\theta\right)(H,\tilde{J}H) \\
	&&+\beta\wedge\tilde{\pi}^*\left(d\theta\right)(V,\tilde{J}H)
		+\beta\wedge\tilde{\pi}^*\left(d\theta\right)(H,\tilde{J}V) \\
	&&+\gamma\wedge\tilde{\pi}^*\left(\frac{dr}{r}\right)(V,\tilde{J}H)
		+\gamma\wedge\tilde{\pi}^*\left(\frac{dr}{r}\right)(H,\tilde{J}V) \\
	&\geq &\vecnorm{V}^2 -\frac{2K}{t}\vecnorm{V}\vecnorm{H}+(C-\frac{K}{t})\vecnorm{H}^2
\end{eqnarray*}

Examination of the discriminant shows this is $\geq 0$ for all $X$ if
$C>\frac{K^2+K}{t}>\frac{K^2}{t^2}+\frac{K}{t}$.
\end{proof}

We are now ready to define $\tilde{\rho}$ for large enough $t$ and proceed with the proof of
Proposition \ref{proposition:enclosure}.

\begin{proof}[Proof of Proposition \ref{proposition:enclosure}]
We will construct a family of functions $\tilde{g}_t\in C^\infty(\widetilde{\mathbb{H}})$ and define
$\tilde{\rho}:=\tilde{g}_t\circ\tilde{\pi}+\tilde{u}^*\rho$. For large enough values
of $t$ this will have the necessary properties to achieve convexity on a level set of
$\tilde{\rho}$ contained in $\tilde{u}^{-1}\rho^{-1}[\rho_{min},\rho_{max}]$.

It is now convenient to change coordinates by the exponential map
\[\widetilde{\mathbb{H}}\longleftarrow\{z\in\mathbb{C}:\im{z}\in[0,\pi]\}\]
Using coordinates $z=x+iy$ on the strip, $r=1,2,3$ corresponds to $x=0,\log 2,\log 3$.
We define $g$ to be a scalar function on the strip, such that for some
positive $\epsilon<\frac{\log 3-\log 2}{2}$:

\begin{enumerate}[(i)]
 \item $g$ depends only on $x$ and is increasing in $x$
 \item $g(x,y)=0$ if $x\geq\log 3$
 \item $g(x,y)=\rho_{min}-\rho_{max}$ if $x\leq -\epsilon$
 \item $\exists \delta>0$ with $\ddsq{}{x}g\geq\delta$ for $x\in [0,\log 2]$
 \item $\ddsq{}{x}g=0$ for $x\in [\log 2+\epsilon, \log 3-\epsilon]$
 \item $\ddsq{}{x}g$ is non negative away from $x\in [\log 3-\epsilon,\log 3]$
\end{enumerate}
\begin{figure}[h]
\centering\scalebox{0.55}{\input{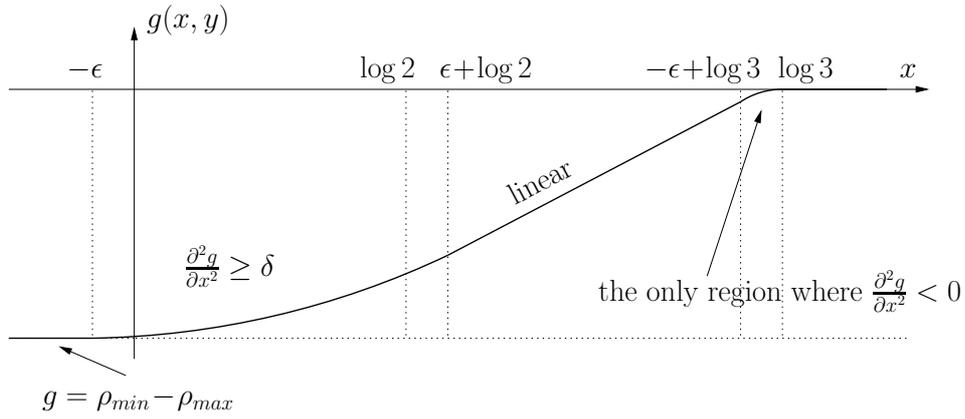}}
\caption{The function $g$ in terms of $x$.}
\label{fig:theamazingg}
\end{figure}

These conditions are illustrated in Figure~\ref{fig:theamazingg}. Such a
function is easily constructed.

Now we extend $g$ to $g_T$, a one-parameter family of functions parametrised by
$T\in [0,\infty]$.

Let $L:=\max \left(-\ddsq{}{x}g\right)$ and
\[
F:= 1+
	\frac{(g(\log 3+\epsilon)-g(\log 2+\epsilon))
	\left(\frac{T}{\log 3-\log 2-2\epsilon}+2\right)}
	{\rho_{max}-\rho_{min}}
\]

Then we can define smooth $g_T$ by:
\begin{itemize}
 \item $g_T(x,y)=\frac{g(x+T,y)+\rho_{max}-\rho_{min}}{F}-\rho_{max}+\rho_{min}$
for $x\leq \log 2 + \epsilon -T$
 \item $g_T(x,y)=\frac{g(x+T,y)}{F}$ for $x\geq \log 3 -\epsilon$
 \item $g_T$ interpolates linearly in the range $x\in [\log 2 + \epsilon -T, \log 3-\epsilon]$
\end{itemize}

For some rather messy constant $S>0$ and large enough $T$:
\begin{itemize}
 \item $\ddsq{}{x}g_T$ is non negative away from $x\in [\log 3-\epsilon,\log 3]$
 \item $\ddsq{}{x}g_T\geq
	\frac{-L}{F}\geq-\frac{SL}{T}$
 \item $\ddsq{}{x}g_T > \frac{\delta}{F}\geq \frac{S\delta}{T}$ for
	$x\in [-T,\log 2-T]$
\end{itemize}

Let $\tilde{g}_T$ be $g_T$ pushed forward to $\widetilde{\mathbb{H}}$ by the exponential map,
then we have:
\[-d(d\tilde{g}_T\circ \tilde{j})_{(r,\theta)}=
	\left(\ddsq{}{x}g_T(\log r,\theta)\right)\frac{dr}{r}\wedge\theta
\]

Setting $T=\log t$, we now consider
$\tilde{u}^*\rho+\tilde{\pi}^*\tilde{g}_{\log t}$ on the set
$\tilde{u}^{-1}\rho^{-1}[\rho_{min},\rho_{max}]$.
For large enough $t$ we have $\frac{S\delta}{\log t}>\frac{K^2+K}{t}$,
so by Corollary~\ref{corollary:correction} it is
plurisubharmonic over the region $r\in \left[\frac{1}{t},\frac{2}{t}\right]$. Also
if $t$ is small enough, then by compactness it is plurisubharmonic over the region
$r\in \left[\frac{2}{t},\frac{3}{t}\right]$. Elsewhere it is subharmonic.

The $\rho_{max}$ level set of $\tilde{u}^*\rho+\tilde{\pi}^*\tilde{g}_{\log t}$ is
contained in $\tilde{u}^{-1}\rho^{-1}[\rho_{min},\rho_{max}]$. It agrees with the 
$\rho_{min}$ level set of $\tilde{u}^*\rho$ near $r=0$ and the $\rho_{min}$ level set
away from the neighbourhood modelled by $\widetilde{\mathbb{H}}$. Hence
it makes our Lagrangian boundary condition enclosed.

A careful examination of this procedure shows that the choices involved are all
canonical up to isotopy through such choices.
\end{proof}

\begin{remark}
As $t\to\infty$ the above argument gives a deformation of MBL--fibrations from the fibration
with striplike ends to the original one over $\overline{\mathbb{D}}$ with marked points on
the boundary. For each value of $t$ we also have the same Lagrangian boundary condition.
With a little care, we also get a deformation of \emph{enclosed} Lagrangian boundary conditions as
$t\to\infty$. This will be important later as it shows the relative invariants from these
fibrations are all the same.
\end{remark}

Now we have shown how to deform our original model fibration over $\overline{\mathbb{D}}$
to have striplike ends, the trivialisations over which are specified entirely by the image in $N$
of the end and the fibre of $E\to N$ over that point. Suppose we now have a pair of holomorphic
maps $\overline{\mathbb{D}}\to N$ which agree on one point $z\in\partial \mathbb{D}$, then 
by deforming these as above to admissible maps from surfaces with striplike ends we can glue
at the ends corresponding to $z$. Furthermore, using the previous lemmas we can still define
enclosed Lagrangian boundary conditions on the resulting fibration simply by specifying some
Lagrangians in the fibres at infinity. The same works for larger composites.

In order to define an exact MBL--fibration from any admissible map $\func{u}{B}{N}$
we work with composites of our model maps (as illustrated in Figure~\ref{fig:bubblegumtree}).
Namely, any admissible $u$ is isotopic
through admissible maps to a map obtained as follows.
Take an acyclic collection of holomorphically embedded discs in $N$ joined (without any condition on tangencies)
at certain marked points
on their boundaries. By this, we mean that the graph whose vertices are given by the discs and whose edges correspond
to marked points at which the discs are joined, is acyclic. We shall call this a \emph{tree construction}. Now:
\begin{itemize}
 \item pullback the fibration $E\to N$ over each of these discs
 \item deform them all to have striplike ends instead of each marked point
 \item glue the fibrations at the striplike ends corresponding to the joined marked points
\end{itemize}

\begin{figure}[h]
\centering\scalebox{0.4}{\input{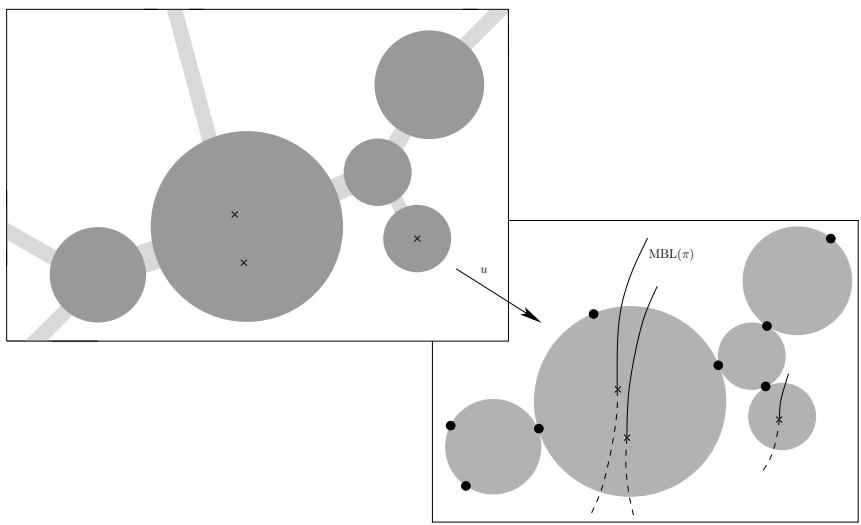}}
\caption{An example of a map of a surface $B$ into $N$ given by a tree construction.
Trivialisations of the exact MBL--fibration are over
lighter shaded regions.}
\label{fig:bubblegumtree}
\end{figure}

Let $\func{u}{B}{N}$ be such a construction. This specifies for each striplike end $e\in I$ a
regular point $z_e\in N$ to which it maps and an exact MBL--fibration $(E,\pi)$ over $B$.
The choices involved in the construction mean that $(E,\pi)$ is well-defined up to isotopy.
The same is true if we choose sufficiently many exact Lagrangian submanifolds in the fibres
over the points $z_e$ and a suitable deformation (by Lemma~\ref{lemma:sptdef})
to define an enclosed exact Lagrangian boundary condition $Q$. One can always enlarge the support
of this deformation such that $Q$ remains defined through the
relevant isotopies of exact MBL--fibrations
(and all subsequent arguments in this section).

Much of the tree structure and the positioning of the singular values of $\pi$ in $B$ by the
construction does not affect the exact MBL--fibration except by isotopy. Suppose we have
a smooth embedding of $\func{f}{B}{B}$ with following properties (e.g. Figure~\ref{fig:baserestriction}):

\begin{enumerate}[(i)]
 \item $u\circ f$ should be admissible and define the same set $\{z_e\}$ of ends.
\label{item:restriction}
In particular $f(\partial B)$ should not contain any critical value of $\pi$.
 \item Consider any of the regions of $B$ identified with parts of $\widetilde{\mathbb{H}}$
in order to construct the striplike ends (including those glued together). In each of these
regions $f(B)$ should be a union of disjoint wedges. This corresponds in the trivialisation
of the striplike ends to $f(B)$ being a union of `substrips'.
 \item $f$ should be isotopic to the identity through smooth embeddings maintaining condition
\label{item:isotopy}
(\ref{item:restriction}) above.
\end{enumerate}

\begin{figure}[h]
\centering\scalebox{0.6}{\input{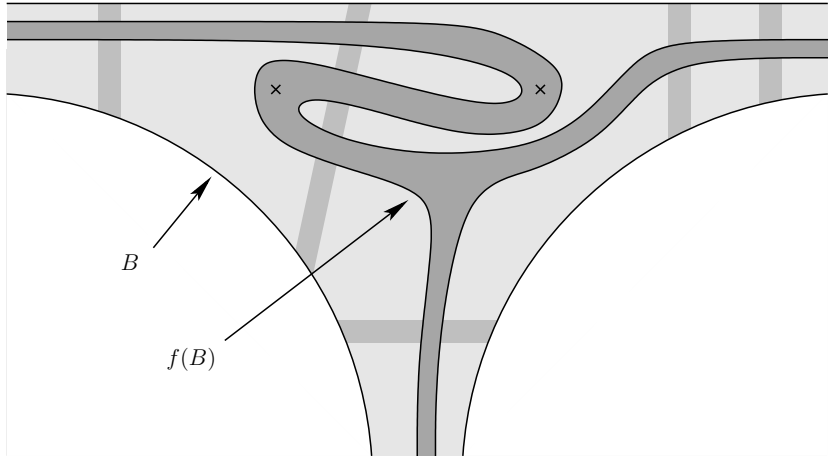}}
\caption{An example restriction of the base of a tree construction used to
change the tree structure (it redistributes the singular values). The regions
where striplike ends were glued in the construction of $B$ are
indicated by shading.}
\label{fig:baserestriction}
\end{figure}

\begin{lemma}
\label{lemma:baserestriction}
In the above construction, the isotopy of (\ref{item:isotopy}) induces an isotopy of exact MBL--fibrations
(and also of any enclosed exact Lagrangian boundary conditions, provided sufficient care is taken
with symplectic parallel transport using Lemma~\ref{lemma:sptdef}).
\end{lemma}

The resulting fibration can be viewed as another potentially very different tree construction.

\begin{proof}
 This proof of this is immediate from the definitions above. To construct the accompanying
isotopy of enclosed exact Lagrangian boundary conditions, one simply fixes an exact Lagrangian
in the fibre over one end of each edge of the base and extends by symplectic parallel transport,
at each stage of the isotopy, to define $Q$.
This works with a sufficiently strong application of Lemma~\ref{lemma:sptdef}.
\end{proof}

\begin{lemma}
Suppose $\func{\tilde{u}}{\tilde{B}}{N}$ is another such construction such that $u,\tilde{u}$ are
isotopic through admissible maps fixing the endpoints $z_e$.
Then $u,\tilde{u}$ are related by a finite sequence of the following `moves':
\begin{enumerate}[(a)]
 \item isotopy of the tree construction of embedded admissible holomorphic discs through
such constructions
 \item decomposition of any of the embedded holomorphic discs into two joined at a new marked point
(see Figure~\ref{fig:bubbledecomp})
 \item changing the tree structure by restriction of the base to a surface embedded in it as in
\label{item:treebaserestriction} 
Lemma~\ref{lemma:baserestriction}
\end{enumerate}
Suppose furthermore we have enclosed exact Lagrangian boundary conditions $Q,\tilde{Q}$
defined on the exact MBL--fibrations over $u,\tilde{u}$ by the same set Lagrangians in
the fibres of $N$ over endpoints $z_e$ which are not `internal' to the tree constructions.
Then the two exact MBL--fibrations, together with 
enclosed exact Lagrangian boundary conditions, are isotopic.
\end{lemma}

\begin{figure}[h]
\centering\scalebox{0.35}{\includegraphics{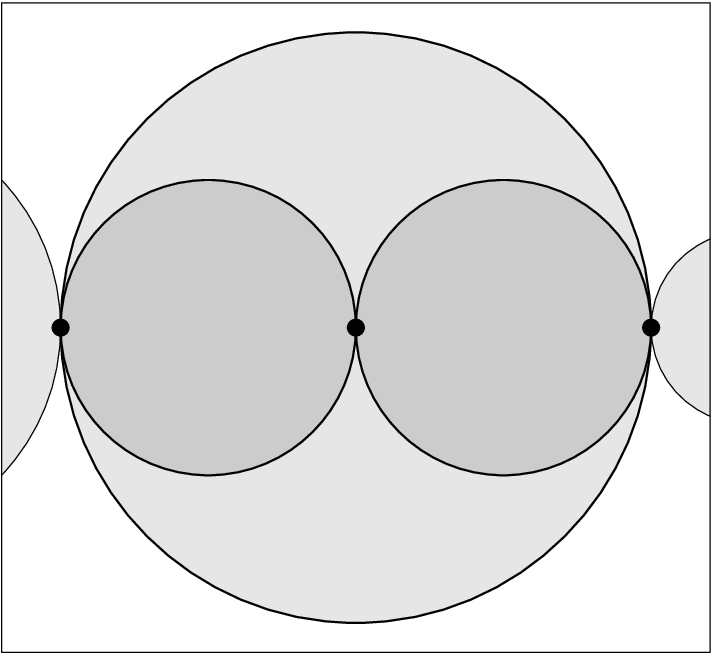}}
\caption{An example decomposition of a holomorphic disc with boundary marked points
into two joined at a new marked point.}
\label{fig:bubbledecomp}
\end{figure}

\begin{proof}[Sketch of proof]
We take the base of one of the fibrations and smoothly embed it in the base of the other
as for Lemma~\ref{lemma:baserestriction}. This can be done such that after application of move
(\ref{item:treebaserestriction}) the tree constructions are now related only by the remaining
two moves. It may help to think of decomposing and isotoping the tree structures such that
all the discs are complex linear embeddings in our favourite coordinate
neighbourhoods of $N$. 
\end{proof}

\subsection{Holomorphic sections}
\label{section:holomorphic}

In this section, we define the relative invariant associated to an exact
MBL--fibration with enclosed exact Lagrangian boundary conditions. The invariant's
definition will be largely identical to Seidel's construction of
relative invariants of exact Lefschetz fibrations (cf. \cite{seidel:les}
sections 2.1 and 2.4). We begin by defining a large class of almost
complex structures on an exact MBL--fibration $(E,\pi)$ which extend $J_0$.

\begin{definition}
An \emph{almost complex structure} on an exact MBL-fibration $(E,\pi)$
over a surface $B$ with striplike ends and boundary marked points $P\subset\partial B$
is an almost complex structure $J$ on $\pi^{-1}(B\setminus P)$ such that:
\begin{itemize}
 \item $J=J_0$ in a neighbourhood of $\crit\pi$ and on the complement of some
compact subset of $E$
 \item $D\pi\circ J=j\circ D\pi$, i.e $\pi$ is $(J,j)$--holomorphic,
 \item Over striplike ends $J$ is invariant w.r.t. translation in the $[0,\infty)$--direction,
 \item Over embedded curves $\gamma$ in $B$ ending, with non-zero derivative, at a boundary,
marked point $e$, $J$ extends smoothly to an almost complex structure $J_{e,t}$,
in the fibre over $e$ depending smoothly on the angle of approach $\pi t\in[0,\pi]$.
\end{itemize}
\end{definition}

The last two conditions ensure that, at each end, $J$ limits to some time dependent
almost complex structure $J_t$ in the fibre over the end.

\begin{definition}
An almost complex structure $J$ on $E$ is \emph{compatible relative to
$j$} if the 2--form
$\Omega(.,J.)|_{(TE^v)^{\otimes 2}}$ is everywhere symmetric and positive definite.
We will denote the set of such almost complex structures by
$\mathcal{J}(E,\pi)$.

Given a \emph{time-dependent} almost complex structure $J_{e,t}$ on the
fibre at infinity or fibre over the marked point at each end $e$ of $B$,
the set of those $J$ which agree in the limit with each $J_{e,t}$ will be denoted
$\mathcal{J}(E,\pi,\{J_{e,t}\}_{e\in I})$ or simply $\mathcal{J}(E,\pi,\{J_e\})$.
\end{definition}

For Floer cohomology, compatibility is in general an unnecessarily strong condition.
A more detailed
explanation of this is given by Salamon in $\cite{salamon:qcoh}$, in particular the discussion surrounding Remark 5.1.1.
It suffices instead for a time dependent almost complex
structure $J_t$ to tame the symplectic form.
This is often an easier condition to use, as it is preserved by small
compactly supported deformations of $J_t$.

Let $\func{u}{\mathbb{R}\times[0,1]}{M}$, where $M$ is a K\"ahler manifold. Compatibility in a neighbourhood of
the Lagrangian intersections is still convenient, though not strictly necessary. It ensures that the Cauchy-Riemann operator

\[\dd{}{s}+J_t\dd{}{t}\]

\noindent thought of as (the completion of) an operator on the
tangent space to the space of paths ($[0,1]\rightarrow M$)
with ends on the Lagrangian submanifolds, can be written as

\[\dd{}{s}+A\]

\noindent with $A$ symmetric near the ends. However, even compatibility near the intersections
is not strictly necessary as $A$ is always asymptotically symmetric at the ends.
To avoid discussing this any further, we will keep the requirement of compatibility near the
Lagrangian intersections for the purposes of this paper.

\begin{definition}
An almost complex structure $J$ on $E$ with Lagrangian boundary condition $Q$
is \emph{tame relative to
$j$} if the 2--form
$\Omega(.,J.)|_{(TE^v)^{\otimes 2}}$ is everywhere positive definite
and compatible near the Lagrangian intersections at the ends.
We will denote the set of such almost complex structures by
$\mathcal{J}_\tame(E,\pi,Q)$.

Given $J_{e,t}$ ends we define
$\mathcal{J}_\tame(E,\pi,Q,\{J_{e,t}\}_{e\in I})$ in the same way as
$\mathcal{J}(E,\pi,\{J_{e,t}\}_{e\in I})$, but with relative compatibility replaced by
relative tameness.
\end{definition}

For a relatively compatible or tame almost complex structure $J$ on an exact MBL--fibration $E$,
we denote the set of smooth sections of $E$ with boundary on $Q$, which are $(j,J)$--holomorphic
(except at boundary marked points, where $J$ is not defined) by $\mathcal{M}_J(Q)$. We refer to these
simply as \emph{holomorphic sections}. If $Q$ is enclosed we will write $\mathcal{M}_J(Q)$
for the subspace of those sections contained within the particular enclosure. It will
be clear from the context which enclosure is being used.

\begin{lemma}[cf. \cite{seidel:les}, Lemma~2.2]
\label{lemma:convexity}
Let $(E,\pi,\Omega,\Theta,J_0,B,j)$ be an exact MBL--fibration over a surface with ends
and an enclosed, exact Lagrangian boundary condition $(Q,\rho,R)$.
Thet for any almost complex structure $J$ on $(E,\pi,\Omega,\Theta,J_0,B,j)$,
the space of $(j,J)$--holomorphic sections $u$ with the given Lagrangian boundary
condition decomposes into two components:
\begin{itemize}
 \item those with image contained in the fibrewise compact set $\rho^{-1}[0,R-\epsilon]$
 \item those whose image leaves $\rho^{-1}[0,R]$
\end{itemize}
\end{lemma}

\begin{proof}
$J=J_0$ in a neighbourhood of $\rho^{-1}(R)$, so $\rho$ is subharmonic w.r.t. $J$ in that region.
This implies that $\rho\circ u$ is subharmonic where it takes values near $R$,
so cannot have a local maximum in the region $\rho^{-1}[R-\epsilon,R]$.
\end{proof}

In the construction of Floer cohomology, it is necessary to use a \emph{regular}
almost complex structure. This ensures that the moduli spaces
of sections are smooth finite dimensional manifolds and, in particular, that
the Floer differential squares to zero. Here we have the same requirement.

We say that a relatively tame almost complex structure $J$ on an exact MBL--fibration
is \emph{regular} at $u\in\mathcal{M}_J(Q)$ if $D_{u,J}$ is onto. Here $D_{u,J}$ is
the extension of the linearised $\bar\partial_J$ operator at the $J$--holomorphic section $u$
to relevant Sobolev completions. For a more detailed description of regularity
which applies in this setting, see \cite[Section 2]{seidel:les}. This
condition ensures that $\mathcal{M}_J(Q)$ is a smooth finite dimensional 
manifold of the \emph{correct} dimension (given by a Maslov index of $u$)
near $u\in\mathcal{M}_J(Q)$. $J$ is regular with respect to the enclosed Lagrangian boundary
condition $Q$ if it is regular for all $u\in\mathcal{M}_J(Q)$. We denote the space of
such almost complex structures by
\[\mathcal{J}_\tame^{reg}(E,\pi,Q,\{J_{e,t}\})\subset \mathcal{J}_\tame(E,\pi,Q,\{J_{e,t}\})\].

Corresponding to Lemma~2.20 of \cite{seidel:les} and with essentially the same proof, we
have generic regularity for complex structures in $\mathcal{J}_\tame(E,\pi,\{J_e\})$ given
any transverse enclosed exact Lagrangian boundary condition $(Q,\rho,R)$.

\begin{lemma}
 $\mathcal{J}_\tame^{reg}(E,\pi,Q,\{J_e\})$ is $C^\infty$--dense in $\mathcal{J}_\tame(E,\pi,Q,\{J_e\})$.
More precisely, given a non-empty open subset $U\subset B$ which is disjoint from the ends
and any $J\in \mathcal{J}_\tame(E,\pi,Q,\{J_e\})$, there are
$J^\prime\in \mathcal{J}_\tame^{reg}(E,\pi,Q,\{J_e\})$ arbitrarily $C^\infty$-close to $J$, such that $J=J^\prime$
outside $\pi^{-1}(U)$.
\end{lemma}

By bounding the symplectic action, as in \cite{seidel:les}
one also achieves a compactification $\overline{\mathcal{M}}_J(Q)$ of $\mathcal{M}_J(Q)$
in the \emph{Gromov-Floer topology} by adding broken sections.

Counting isolated holomorphic sections in $\mathcal{M}_J(Q)$ which limit to
given sets $\{x_e\}_{e\in I}$ of intersection points in the fibres at the ends
(denote these subsets $\mathcal{M}_J(Q,\{x_e\})$) one defines a linear map
on the level of Floer cochain complexes. In fact, by a standard argument, considering
degenerations of 1--dimensional families of sections, one shows that it is a chain map.

\[
\bfig

\node 1a(0,400)[\displaystyle\bigotimes_{e\in I^-}CF(Q_e^1,Q_e^0,J_{e,t})]
\node 2a(1900,400)[\displaystyle\bigotimes_{e\in I^+}CF(Q_e^1,Q_e^0,J_{e,t})]
\node 1b(200,0)[{\displaystyle\otimes_{e\in I^-} \langle x_e\rangle }]
\node 2b(2100,0)[{\displaystyle\sum_{\{y_e\}_{e\in I^+}}\#\mathcal{M}_J(Q,\{x_e\}\cup\{y_e\})(\displaystyle\otimes_{e\in I^+} \langle y_e\rangle )}]

\arrow[1a`2a;C\Phi_0^{rel}((E,\pi),(Q,\rho,R),J)]
\arrow/@{|->}/[1b`2b;]
\efig
\]

The above is written out in full detail, since, in particular, the differentials in the Floer cochain complexes
$CF(L,L^\prime,J_t)$ depend
on the choice of almost complex structure. When one changes this, an argument using continuation
maps gives a chain homotopy equivalence to the new cochain complex which is canonical up to
chain homotopy. Similarly, compactly supported Hamiltonian symplectomorphisms of either
$L$ or $L^\prime$ induce canonical chain homotopy classes of chain homotopy equivalences.
Lemma~\ref{lemma:isotopyoffibrations} (below) can be viewed as a generalisation of both of these
results in the setting of exact Lagrangian submanifolds of exact symplectic manifolds.

\begin{definition}
The \emph{relative invariant} $\Phi_0^{rel}((E,\pi),(Q,\rho,R),J)$ is defined to be the map induced on Floer
cohomology by $C\Phi_0^{rel}((E,\pi),(Q,\rho,R),J)$. Under the assumption that the $Q$ is spin in any fibre,
one can also orient the moduli spaces and perform the count of sections with signs (see \cite{seidel:book}).
This allows
the use of $\mathbb{Z}$ coefficients. Alternatively, one can instead use
$\frac{\mathbb{Z}}{2\mathbb{Z}}$--coefficients.
\end{definition}

\begin{remark}
The convention of signs and arrangement of $Q_e^0,Q_e^1$ differs by a $180^\circ$ rotation
from that in \cite{seidel:les}.
\end{remark}

\begin{remark}
 There is a Poincar\'e duality arising from trivial fibrations over an infinite strip with both
ends in $I^-$ (or both in $I^+$). Composition with this (see gluing below) allows us to
swap ends back and forth between $I^+$ and $I^-$ or also to view the relative invariant as
the element induced in cohomology by:
\[C\Phi_0^{rel}\in\displaystyle\bigotimes_{e\in I} CF(Q_e^1,Q_e^0)\]
\end{remark}

Composing two exact MBL--fibrations with enclosed Lagrangian boundary conditions
can be done by gluing over a single striplike end (if necessary one uses Poincar\'e
duality to move other ends out of the way). Corresponding to Proposition~2.2 of
\cite{seidel:les} we have:

\begin{lemma}
Gluing two exact MBL--fibrations with enclosed Lagrangian boundary conditions
together along an oppositely oriented, but otherwise identical striplike end
gives the composition of the $\Phi^{rel}_0((E,\pi),(Q,\rho,R),J)$ maps.

(It is important that all the data of $E,\pi,Q,J\ldots$etc agrees where the gluing occurs.)
\end{lemma}

We shall now show that a variety of changes can be made to $((E,\pi),(Q,\rho,R),J)$
without changing the relative invariant (beyond composing on either side with the canonical
isomorphisms on Floer cohomology).

\begin{lemma}
\label{lemma:switchenc}
$C\Phi_0^{rel}((E,\pi),(Q,\rho,R),J)$ is unchanged when one switches the enclosure
$(Q,\rho,R)$ for an equivalent enclosure.
\end{lemma}

\begin{proof}
Definition~\ref{def:equivenc} ensures that the moduli spaces $\mathcal{M}_J(Q,\{x_e\})$
are unaffected by this.
\end{proof}

\begin{lemma}
\label{lemma:isotopyoffibrations}
The map $\Phi_0^{rel}((E,\pi),(Q,\rho,R),J)$ depends only on the enclosed region $\rho^{-1}[0,R]$.
It is invariant (up to composition on either side with the canonical
isomorphisms on Floer cohomology) under isotopy of the combined data $((E,\pi),(Q,\rho,R),J)$
such that:
\begin{enumerate}[(a)]
 \item the data remains valid at all stages for the definition of
	$\Phi_0^{rel}((E,\pi),(Q,\rho,R),J)$,
 \item the induced isotopies at each end $E_e$ fix the symplectic form
	and vary the Lagrangian submanifolds only by compactly supported
	Hamiltonian isotopy.
\end{enumerate}
\end{lemma}

\begin{proof}
When one fixes the data over the ends, a standard argument counting sections at
all stages of (some perturbation) of this
isotopy gives a homotopy of the relative invariant at the level of chain complexes.
More generally, it gives a homotopy composed with continuation maps (yielding the
canonical isomorphisms) on
Floer cohomology groups over the ends.
\end{proof}

\begin{remark}
\label{remark:isotopyoffibrations}
In fact, the above argument works to describe general deformations of the data
$(E,\pi)$, $(Q,\rho,R)$ and $J$. Suppose for example, we do not require the symplectic
forms on fibres over the ends to remain fixed (then we most likely also have to vary $Q$
over the ends to ensure that it remains a Lagrangian boundary condition).
The relative invariant then varies by left- and right-composition
with the continuation maps from these changes. In some cases (such as deforming $\Omega$
through exact K\"ahler forms with the added condition that components of $Q$ have vanishing
first cohomology over $\mathbb{R}$) these continuation maps are still canonical isomorphisms,
but in general one cannot expect that to be the case.
\end{remark}

Suppose the geometric data defining the exact MBL--fibration and boundary conditions
splits as a some sort of product, then one can often correspondingly split the relative invariant. A simple
example of this is demonstrated in the following lemma and a more involved version comes
in Section~\ref{section:twistprod}.

\begin{lemma}
\label{lemma:fibreproduct}
Suppose:
\begin{itemize}
 \item $(E,\pi)$ splits
as a smooth fibre product of MBL--fibrations $(E_1,\pi_1)$ and
$(E_2,\pi_2)$ over the same base $(B,j)$,
 \item the exact Lagrangian boundary condition $Q$ splits as a fibre-product
of exact Lagrangian boundary conditions $Q_1,Q_2$ in the two factors,
 \item $(\rho,R),(\rho_1,R),(\rho_2,R)$ are enclosures for $Q,Q_1,Q_2$ respectively
with the property that $\rho$ is $C_0$--close to $\max\{\rho_1,\rho_2\}$ in some neighbourhood
of $\rho^{-1}(R)$.
\end{itemize}
Then the relative invariant $C\Phi^{rel}_0((E,\pi),(Q,\rho,R),J)$ splits as the product of the
relative invariants on the two factors
\[C\Phi^{rel}_0((E_1,\pi_1),(Q_1,\rho_1,R),J)\otimes C\Phi^{rel}_0((E_2,\pi_2),(Q_2,\rho_2,R),J)\]
\end{lemma}

\begin{proof}
This is a simple generalisation of the corresponding product formula for Floer cohomology
(which one has over the ends in the setting of the lemma).
$(E,\pi)$ is smooth whenever $(E_1,\pi_1)$ and $(E_2,\pi_2)$
share no singular values in the base.

Let $J_1\in\mathcal{J}_\tame^{reg}(E_1,\pi_1,Q_1)$, $J_2\in\mathcal{J}_\tame^{reg}(E_2,\pi_2,Q_2)$
be any regular almost complex structures on $E_1$, $E_2$ for
boundary conditions $Q_1$, $Q_2$ respectively.
We define the almost complex structure $J$ for the fibre product
$(E,\pi)$ by pullback from the embedding $E=E_1\times_B E_2\to E_1\times E_2$.
Relative tameness of $J_1$ and $J_2$ relative to $j$ implies that $J$
on the product restricts well to $E$ and is tame relative to $j$.

The $J$--holomorphic sections of $E$
are in bijection with the pairs of holomorphic sections of $E_1$ and
$E_2$. The regularity of $J$ for Lagrangian boundary condition $Q_1\times_{\partial B}Q_2$
is an immediate consequence of the regularity of $J_1$ and $J_2$
since the linearisation at any holomorphic section splits as that of $J_1$ and $J_2$, both of
which (when extended to the relevant Sobolev spaces) are surjective.
This bijection of moduli spaces identifies the deformation theories of the sections.
\end{proof}

\begin{remark}
In the case where $\rho_1,\rho_2$ are plurisubharmonic near their respective $R$-level sets
and the complex structures near there are integrable, there is always a plurisubharmonic
$C_0$--approximation to $\max\{\rho_1,\rho_2\}$ near its $R$-level set (see \cite[Lemma 3.8]{ciel:sgsm}).
\end{remark}

\begin{lemma}
Trivial MBL--fibrations over the infinite strip (with trivial boundary
conditions) induce the identity map on Floer cohomology.
\end{lemma}
\begin{proof}
One may take the complex structure for the fibration from that used in
calculating the Floer homology at the striplike ends. The only isolated
holomorphic sections (since we do not quotient by the $\mathbb{R}$--action
on the base)
are the constant sections. These give us the identity map.
\end{proof}

\section{A non-standard splitting of the relative invariant}
\label{section:twistprod}

In this section, we study the relative invariant on a particular family of \emph{twisted} products.
These calculations are necessary in the setting of symplectic Khovanov homology to describe maps corresponding
to presentations of trivial cobordisms
(cf. stabilisation and destabilisation maps of Section~\ref{section:stab}).

The setting will be as follows. Let $(X,J_X^{std})$ be a Stein manifold with exhausting plurisubharmonic
function $\rho_X$ and exact K\"ahler form $\Omega_X=-d(d\rho_X\circ J_X^{std})$. Take a pair of compact exact Lagrangian
submanifolds $K,K^\prime\subset \rho_X^{-1}[0,R)$ which intersect transversely. These conditions suffice
to define $HF^*(K,K^\prime)$ using an almost complex structure which is a small perturbation of $J_X^{std}$.

Now let $(E,\pi,\Omega_E,\Theta_E,J_E,D,j)$ be an exact Lefschetz fibration,
or more generally an exact MBL--fibration, together with two boundary marked points
and enclosed, exact Lagrangian boundary condition $Q$. Then $E\times X$ is an exact MBL--fibration and can be given the
Lagrangian boundary condition $Q\times K$ on one half of the boundary and $Q\times K^\prime$ on the other.
Lemma~\ref{lemma:fibreproduct} shows that, at the chain level, the Floer cohomologies at each boundary marked point,
and also the relative invariant, split as tensor products.

This aim of this section is to construct the same splitting in a case where the product $E\times X$ is replaced by
$E\times_{S^1} P$ (see below for more detail on this as a symplectic associated bundle)
for some fibre preserving Hamiltonian $S^1$--action on $E$ and some circle-bundle $P$. The main
trick is to make the regular almost complex structure and the transverse Lagrangian boundary conditions $S^1$--equivariant. This
has the consequence that the only isolated holomorphic strips/sections are those confined to the fixed locus of the action.
In general such a trick is not possible and we will indicate the important points which make it work in the specific case
considered in this section.

\subsection{The Lefschetz fibration $(E,\pi,\Omega_E,\Theta_E,J_E,D,j)$}
\label{section:E}

We start by considering the following exact Lefschetz fibration
(for some fixed $d=3x^2\in\mathbb{C}^*$):

\begin{equation*}
\xymatrix{
\mathbb{C}^3 \ar[d]^{\pi} & (a,b,c) \ar@{|->}[d]\\
\mathbb{C} & a^3-ax^2+bc}
\end{equation*}

\noindent with the standard exact K\"ahler form on $\mathbb{C}^3$.

Take a closed disc $D$ with two boundary marked points,
mapping holomorphically into a small neighbourhood of the singular value $-2x^3$ in $\mathbb{C}$, such that:
\begin{itemize}
 \item the map is an embedding of $D$ with one with the exception that the two boundary marked points map to the same point,
 \item the singular value $-2x^3$ is hit by an interior point of $D$,
\end{itemize}

\begin{figure}[h]
\centering\scalebox{0.4}{\input{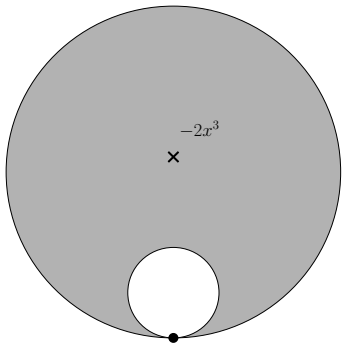}}
\caption{An illustration of the disc $D$.}
\label{fig:crescent}
\end{figure}

We now define the exact Lefschetz fibration $E$ to be the pullback by this smooth map of the above Lefschetz fibration
$\mathbb{C}^3$. This has the advantage of canonically identifying the fibres over the boundary marked points with each
other.

Define $\rho_E(a,b,c):= \norm{a}^2+\norm{b}^2+\norm{c}^2$ using coordinates on $\mathbb{C}^3$, and to start with
take the exact symplectic form $\Omega_E=d\Theta_E=-d(d\rho_E\circ J_E)$, with the standard complex structure $J_E$.
We will adjust $\Theta_E$ and $\Omega_E$, below, to give some control over symplectic parallel transport and the
Lagrangian boundary conditions.

The Hamiltonian $S^1$--action is the unit circle part of the holomorphic $\mathbb{C}^*$--action
$(a,b,c)\mapsto (a,\zeta b, \zeta^{-1}c)$. It is generated by the Hamiltonian
$\mu (a,b,c)=|b|^2-|c|^2$ and has the following important properties:
\begin{itemize}
 \item It preserves fibres of $E$ and of the projection $\proja$ to the $a$--coordinate.
 \item $\rho_E,\Theta_E,\Omega_E$ and $J_E$ are all $S^1$--equivariant.
\end{itemize}

The map $\proja$, restricts to a Lefschetz fibration on non-singular fibres $E_z$ of $E$
with three singular values. As one approaches a singular fibre of $E$, two of these singular
values collide. Understanding $\proja$ is
important as it gives some control over $S^1$--equivariant Lagrangian spheres in the fibres:

\begin{lemma}
\label{lemma:s1lagrangians}
Any $S^1$--equivariant Lagrangian
sphere in a regular fibre of $E$ (as above) can be expressed as
$\mu|_{E_z}^{-1}(0)\cap \proja|_{E_z}^{-1}\im\gamma$ for some smooth
embedded path $\gamma$ in $\mathbb{C}$ between critical values of $\proja|_{E_z}$.
Furthermore, the vanishing cycle in for any path in $B$ into the critical value
of $\pi$ is of this form. In this case $\gamma$ is a smooth path between the two critical
values of $p|_{E_z}$ which collide in the critical fibre of $\pi$.
\end{lemma}

\begin{proof}
Denote the fibre $\proja|_{E_z}^{-1}(a)$ over $a\in\mathbb{C}$ by $E_{z,a}$.
Consider a Lagrangian sphere $L\subset E_z$. By $S^1$--equivariance, $L$
intersects fibres $E_{z,a}$ in, possibly empty, unions of circles, except at
critical points of $\mu$ (which are precisely the critical points of
$\proja|_{E_z}$). At any regular $x\in L$ the intersection
$T_xL\cap (T_xE_{z,a})^{\bot_\Omega}$ is one-dimensional so we can define
a non-vanishing vector field locally on $TL$ of these vectors. Projecting
these to $\mathbb{C}$ by $D\proja$ makes them all tangent, since otherwise we
reach a contradiction with $S^1$--equivariance and the dimension of $L$.

Hence, locally $L$ is preserved by symplectic parallel transport over
certain paths in $\mathbb{C}$. The only way for components of $L$ to be
closed, embedded spheres is for these paths to end at distinct critical points.

All the critical points satisfy $\mu=0$, and $\mu$ is preserved by
symplectic parallel transport, so it suffices to observe that the
$S^1$--action is transitive on any set of the form
$\mu|_{E_{z,a}}^{-1}(0)$ (these are circles or points).
\end{proof}

\begin{remark}
\label{remark:s1trans}
A consequence of the above lemma is that two equivariant Lagrangian spheres are transverse
if and only if they correspond to paths between the critical points of $\proja$
which intersect only at critical points and which enter the critical points at
different angles. The only obstruction to being able to make equivariant Lagrangian 
spheres intersect transversely by Hamiltonian isotopy through equivariant Lagrangian spheres
is topological (from isotopy of the paths), because all Lagrangian isotopies in this case are
Hamiltonian.
\end{remark}

We now construct the transverse, exact Lagrangian boundary condition $Q$ on $E$. In doing, so
we will perform a compactly supported deformation of $\Omega_E$ without changing its restriction
to fibres. To start the construction, we take a pair of $S^1$--equivariant exact Lagrangian spheres
$L_{out}, L_{out}^\prime$ in the fibre at one boundary marked point, say the `output' end.
These correspond to embedded paths $\gamma_{out}$ and $\gamma_{out}^\prime$ under projection to the $a$--coordinate, by
Lemma~\ref{lemma:s1lagrangians}. Figure~\ref{fig:infinitycurves} shows the choice we shall use in
this section, but the same construction can be done with any $S^1$--equivariant exact Lagrangian spheres.

\begin{figure}[h]
\centering\scalebox{0.5}{\input{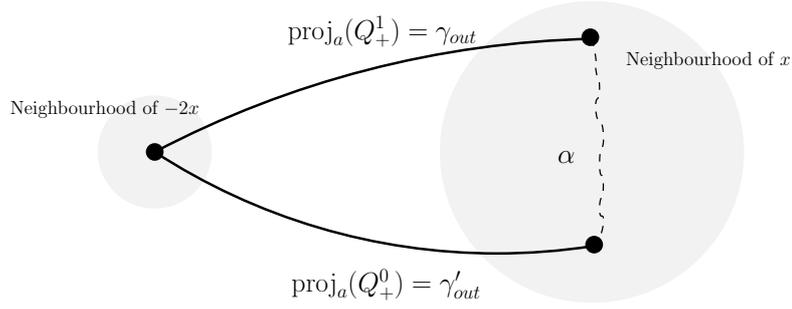}}
\caption{Projections of Lagrangians to curves $\gamma_{out}$, $\gamma_{out}^\prime$ with a single transverse
intersection at an endpoint.
This describes the Lagrangians at the output end of the fibration.}
\label{fig:infinitycurves}
\end{figure}

Extending by symplectic parallel transport over the two components of $\partial D$ defines
an exact Lagrangian boundary condition and a pair of exact $S^{-1}$--equivariant Lagrangian
spheres $L_{in},L_{in}^\prime$ at the input end. Identifying the input and output fibres by 
symplectic parallel transport over the component of $\partial D$ going anti-clockwise from the
input to the output end of the fibration, we find that $L_{in}'$ and $L_{out}^\prime$ are
Hamiltonian isotopic and that $L_{in}$ corresponds to $\gamma_{out}$ acted on by the positive half twist
$\tau_\alpha$, in the curve $\alpha$ illustrated in Figure~\ref{fig:minusinfinitycurves}. To see
that this half twist is the correct monodromy, we need only check that it is the monodromy experienced 
by the critical points of $\proja$.

\begin{figure}[h]
\centering\scalebox{0.5}{\input{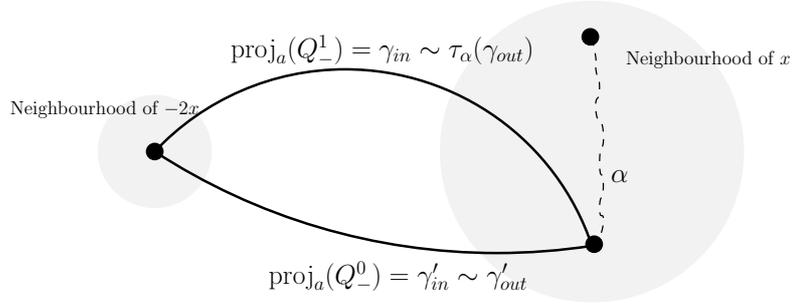}}
\caption{The curves of Figure~\ref{fig:infinitycurves} after a half twist $\tau_\alpha$ in
the curve $\alpha$ has been applied to $\gamma_{out}$. This describes the Lagrangians at the input end of the fibration.}
\label{fig:minusinfinitycurves}
\end{figure}

However,
this only identifies the paths $\gamma_{in}, \gamma_{in}^\prime$ up to isotopy, so we do not necessarily have
that $Q$ is transverse. To correct this, it suffices to perform isotopies of the
paths $\gamma_{in},\gamma_{in}^\prime$ (to get the curves in Figure~\ref{fig:minusinfinitycurves}.
These correspond to $S^1$--equivariant Lagrangian isotopies of
$L_{in},L_{in}^\prime$ and which are generated by $S^1$--equivariant, compactly supported Hamiltonian
isotopies since the fibres of $E$ are exact K\"ahler. Therefore it suffices to change $\Theta_E$ and
$\Omega_E$ on compact
neighbourhood in $E$ to ensure that $Q$ is transverse. In particular, this change is confined to fibres over
a small neighbourhood of $\partial D$, does not affect the restriction of $\Theta_E$, or $\Omega_E$ to fibres, and preserves 
$S^1$--equivariance of both.

Consider again the $S^1$--invariant points of $E$. These are precisely
the critical points of $\proja$. In every fibre $E_z$ there
are three such points, with the exception of the singular fibre, in which two of these points coincide
(in Figures~\ref{fig:infinitycurves} and \ref{fig:minusinfinitycurves}, these are the two points on the right-hand side).
The pairs of points which collide form a disc over $D$ which projects singularly to $D$,
so cannot contain any smooth sections of $E$.

Identify $\mathbb{R}\times [0,1]$ biholomorphically with $D$, punctured at both boundary marked points. This 
is canonical up to the obvious $\mathbb{R}$--action. Symplectic parallel transport maps, over the (closures of the)
paths $x\times[0,1]$ and $\mathbb{R}\times\{0\}$ are well-defined for a small enough neighbourhood of $s(D)$,
since symplectic parallel transport preserves the locus $s(D)$.
Hence, it smoothly trivialises a neighbourhood $M$ of $s(D)$, identifying it with
a neighbourhood of $\{s(z)\}\times D$ inside $E_z\times D$, where $z$ is the input boundary marked point.
Adjusting $\Theta$ in an $S^1$--equivariant way near $M$, without changing its restriction to fibres, we ensure
that symplectic parallel transport respects the trivialisation of $M$.

\begin{lemma}
\label{lemma:omegae}
$\Theta_E$ and consequently $\Omega_E$ can be deformed $S^1$--equivariantly on a compact subspace of $E$, avoiding
the singular point of $\pi$, such that:
\begin{itemize}
 \item their restrictions to fibres of $E$ are unchanged,
 \item the Lagrangian boundary condition $Q$ defined by extending $(L_{out},L_{out}^\prime)$ over $\partial D$
is transverse,
 \item symplectic parallel transport over all paths in $D$ respects a trivialisation of a neighbourhood
$M$ of the only $S^1$--equivariant section.
 \item symplectic parallel transport over the section of $\partial D$ passing anti-clockwise from the input
to the output boundary marked point defines a holomorphic map between neighbourhoods of the $S^1$--invariant points
in the input and output fibres.
\end{itemize}
Consequently $Q|_M$ is also trivial over each part of $\partial D$.
\end{lemma}

\begin{proof}
We have covered all the necessary details above. First we adjust $\Theta_E$ near the part of $\partial D$ corresponding
to $\mathbb{R}\times\{0\}$ to change the symplectic parallel transport over this path by a Hamiltonian symplectomorphism
and ensure that it identifies neighbourhoods of the $S^1$--invariant points in the fibres at either boundary marked point by
the identity map (these two fibres are already identified as the same fibre in $\mathbb{C}^3$). This is holomorphic.

Next, we perform the deformation to ensure
triviality of $M$, then (possibly shrinking $M$) we deform away from $M$ to ensure that $Q$ is transverse elsewhere.
\end{proof}

For the rest of Section~\ref{section:twistprod}, we shall
take $\Omega_E$ to have the properties described in Lemma~\ref{lemma:omegae}.

\begin{remark}
\label{remark:stabmap}

The relative invariant associated to $E$ with exact Lagrangian boundary condition $Q$ is described below.
However, the calculation is omitted here as it is covered in more generality in Section~\ref{section:s1calc}.

Identifying the input with $H^*(S^2)\simeq\mathbb{Z}[X]/(X^2)$ and the output with
$H^*(\{pt\})\simeq \mathbb{Z}$, the relative invariant for $E$ and $Q$ as given above is:

\[
\begin{array}{ccc}
	\mathbb{Z}[X]/(X^2)&\to&\mathbb{Z}\\
	X&\longmapsto&0\\
	1&\longmapsto&1\\
\end{array}
\]
\end{remark}

\subsection{The symplectic associated bundle $E\times_{S^1}P\simeq E\times_{\mathbb{C}^*}(F\setminus 0)$}
\label{section:sympassocbundle}

The construction that follows is a recap of the symplectic associated bundle construction in Section~4.3
of \cite{seidelsmith:khsymp}. 

We take $X,K,K^\prime$ as at the beginning of this section. Let $F\mapsto X$ be a holomorphic
line bundle with a Hermitian metric and compatible connection, which is a subbundle of the trivial
$\mathbb{C}^2$--bundle over $X$. Let $P\subset F$ be the unit-circle
bundle and $\alpha\in\Omega^1(P)$ be the connection one-form. The normalisation is such that, if $R$
is the rotational vector field on $P$ ($2\pi$ periodic), then $\alpha(R)\equiv 1$. Take $E$ to be 
the exact Lefschetz fibration defined in the previous subsection. The symplectic associated bundle is

$$E\times_{S^1}P\rightarrow X$$

One can check that the 2--form $\Omega_{ass}=\Omega_X+\Omega_E+d(\alpha\mu)$ on $E\times P$ descends to a well-defined
2-form on $E\times_{S^1}P\rightarrow X$. Furthermore, on a neighbourhood of $\mu^{-1}(0)$:
\begin{itemize}
 \item $\Omega_{ass}$ is symplectic where $\Omega_E$ is symplectic;
 \item $\Omega_{ass}$ is symplectic on restriction to fibres over $D$ where $\Omega_E$.
\end{itemize}

Also in a neighbourhood of $\mu^{-1}(0)$, symplectic parallel transport over $D$ fixes the projection
to $X$ and is identical to the symplectic parallel
transport in $E$ up to an $S^1$--ambiguity. Hence, $Q\times_{S^1} K$, $Q\times_{S^1} K^\prime$, over the two
parts of $\partial D$, form a transverse exact Lagrangian boundary condition $\tilde{Q}$ for
$E\times_{S^1}P\rightarrow X$.

Since the $S^1$ action on $E$ is part of a holomorphic $\mathbb{C}^*$--action, we can construct 
the holomorphic associated bundle $E\times_{\mathbb{C}^*}(F\setminus 0)$, which is
identified with $E\times_{S^1}P\rightarrow X$. This gives a natural candidate for a complex structure
$\tilde{J}^{std}$ on $E\times_{S^1}P$. Seidel and Smith \cite{seidelsmith:khsymp} show that, where it symplectic,
$\Omega_{ass}$ is K\"ahler with respect to $\tilde{J}^{std}$.

This means that a small neighbourhood of $\mu^{-1}(0)$ has the structure of an exact Lefschetz fibration with transverse,
exact Lagrangian boundary condition $\tilde{Q}$. However, enclosures for $Q$ are not yet defined. To define
them, it suffices to have an exact, $S^1$--equivariant K\"ahler form $\tilde{\Omega}$ which is equal to $\Omega_{ass}$ on 
some neighbourhood of $\tilde{Q}\subset\mu^{-1}(0)$ and which is equal to $-d(d\tilde{\rho}\circ \tilde{J}^{std})$
on the complement of some slightly larger neighbourhood. The construction of $\tilde{\Omega}$ relies upon
the triviality of $E\times_{\mathbb{C}^*}(F\setminus 0)$ as a holomorphic vector bundle over $X$ (see Remark~\ref{remark:bundletriv} below)
and is explained in detail in Section~\ref{section:stab}.

\begin{remark}[a simplified version of \cite{seidelsmith:khsymp}, Remark~42]
\label{remark:bundletriv}

$F\oplus F^{-1}$ is isomorphic to the trivial holomorphic bundle, because $X$ is Stein. As a holomorphic
bundle over $X$, the associated bundle $E\times_{\mathbb{C}^*}(F\setminus 0)$ splits by the $a,b$ and $c$ coordinates
on $E$ (from $\mathbb{C}^3$) respectively as $\mathbb{C}\oplus F\oplus F^{-1}$. This gives the following identification
of holomorphic bundles over $X$.

$$E\times_{\mathbb{C}^*}(F\setminus 0) = \mathbb{C}\oplus F\oplus F^{-1} = \mathbb{C}\oplus\mathbb{C}^2 $$

It should be noted that the projection to $D$ and $S^1$ action, do not necessarily fit nicely with this trivialisation.
\end{remark}

\subsection{Constructing a regular $S^1$--equivariant almost complex structure}
\label{section:s1equij}

In order to relate the relative invariant of $E\times_{\mathbb{C}^*}(F\setminus 0)$ to that of $E$
and the Floer cohomology of $K,K^\prime$ in $X$, we shall construct a complex structure $\tilde{J}$ on
$E\times_{\mathbb{C}^*}(F\setminus 0)$ with the following properties:
\begin{itemize}
 \item $\tilde{J}\in$ is regular (includes restricting to ends as a regular time-dependent almost complex
structure for the construction of Floer cohomology);
 \item $\tilde{J}$ is $S^1$-equivariant;
 \item $\projX$ is $(\tilde{J},J_X)$--holomorphic for some regular almost complex structure $J_X$ on $X$.
\end{itemize}

This will have the consequence that the
moduli of holomorphic sections comes with an $S^1$--action, and hence that the only
isolated sections are those which are fixed by the $S^1$ action. With the correct setup,
projection to $X$ then identifies these moduli spaces with the moduli spaces of holomorphic
strips in $X$ with boundary on $K,K^\prime$.

In fact, much of this section can be avoided by considering the $S^1$--action on the
non-regular moduli spaces (as Kuranishi spaces \cite{fooo},\cite{joyce:kcohom}) and identifying the $S^1$--fixed components
of the moduli space with the non-regular moduli space of holomorphic strips
in $X$ with boundary on $K,K^\prime$. In order not to rely on Kuranishi space technology, we
shall continue with the equivariant regularity argument.

There are two stages to the construction which proceed in a similar manner. Firstly,
we construct an almost complex structure $J_{e,t}$ at each end, which is regular for Floer
cohomology there (in fact, we may take the same one on both ends). Then we extend these
over the interior and perturb the result to get $\tilde{J}$ regular for the calculation of
the relative invariant.

\subsubsection{Regularity over the boundary marked points}
\label{section:ends}

Let $z\in D$ be a boundary marked point.
We start with $\tilde{J}^{std}$ restricted to the fibre $E_z\times_{\mathbb{C}^*}(F\setminus 0)$ and consider how close it is
to already being regular at $S^1$--invariant holomorphic strips $u$. We shall write $L, L^\prime$ for
$L_{in}, L_{in}^\prime$ or $L_{out}, L_{out}^\prime$ depending on which boundary marked point we are considering.

The set of $S^1$--fixed points of $E_z$ consists of just three points, the three critical points of $\proja|_{E_z}$.
Hence, the $S^1$--fixed points of $E_z\times_{\mathbb{C}^*}(F\setminus 0)$ form three copies of $X$, each projecting to one of the critical
values of $\proja$. Consequently, $S^1$--equivariant $u$ are specified uniquely by $\projX\circ u$
and by the critical point hit by $\proja\circ u$. Conversely, given a holomorphic strip in $X$ with boundaries on $K,K^\prime$
and a choice of critical point of $\proja|_{E_z}$, one gets an $S^1$--equivariant, holomorphic $u$ if and only if that critical point
is in $L_{in}\cap L_{in}^\prime$ or $L_{out}\cap L_{out}^\prime$ (depending in which end we are considering).
The former contains two of the critical points, the latter contains one as can be seen from Figures~\ref{fig:infinitycurves}
and \ref{fig:minusinfinitycurves}.

We would like to know if there are any deformations of $S^1$--equivariant holomorphic strips through non-$S^1$--equivariant
ones. More precisely, is this possible to first order? First some notation is needed.

Let $\mathcal{B}$ be the space of smooth maps $\func{u}{\mathbb{R}\times[0,1]}{E_z\times_{\mathbb{C}^*}(F\setminus 0)}$ satisfying the Lagrangian
boundary conditions. Let $\mathcal{E}_{u,\tilde{J}^{std}}\mapsto\mathcal{B}$ be the vector bundle of $(0,1)$--forms on $D$
with values in $u^*(TE_z,\tilde{J})$. Similarly, define $\mathcal{B}_X$ and $\mathcal{E}_{v,J_X^{std}}$ for
holomorphic strips $v$ in $X$.
The section $\func{\bar{\partial}_{\tilde{J}^{std}}}{\mathcal{B}}{\mathcal{E}_{u,\tilde{J}^{std}}}$
(extended to the relevant Sobolev and $L^p$--completions) describes the moduli space of holomorphic strips as its vanishing
locus. Its behaviour, to first order, at $S^1$--equivariant $u$ can be described by the section
$\func{\bar{\partial}_{J_X^{std}}}{\mathcal{B}_X}{\mathcal{E}_{u_X,J_X^{std}}}$ together with some `normal data'
(writing $u_X$ for $\projX\circ u$).

This normal data comes from considering deformations which fix $u_X$. These are deformations of the map
$\widehat{u}$ of the strip into $u_X^*(E_z\times_{\mathbb{C}^*}(F\setminus 0))$ induced by $u$
(note $\widehat{u}$ is not constant). Because $u_X$ is holomorphic, $u_X^*(E_z\times_{\mathbb{C}^*}(F\setminus 0))$
comes with a natural complex structure $\widehat{J}$. Let $\widehat{\mathcal{B}}$ be the space of holomorphic strips in
$u_X^*(E\times_{\mathbb{C}^*}(F\setminus 0))$ and define $\mathcal{E}_{v,\widehat{J}}\mapsto\widehat{\mathcal{B}}$ be the vector bundle
of $(0,1)$--forms on $D$ with values in $v^*(T(u_X^*(E_z\times_{\mathbb{C}^*}(F\setminus 0))),\widehat{J})$. As before,
we also get a section $\bar{\partial}_{\widehat{J}}$.

Let $D_{\widehat{u},\widehat{J}},D_{u,\tilde{J}^{std}}$ and $D_{u_X,J_X^{std}}$ be the linearisations of
$\bar{\partial}_{\widehat{J}},\bar{\partial}_{\tilde{J}^{std}}$ and $\bar{\partial}_{J_X^{std}}$ at
$\widehat{u},u$ and $u_X$, respectively.
Then we get the following map of short exact sequences:

$$\bfig
 \iiixii{15}<400>[T_{\widehat{u}}\widehat{\mathcal{B}}`
		  T_u\mathcal{B}`
		  T_{u_X}\mathcal{B}_X`
		  \mathcal{E}_{\widehat{u},\widehat{J}}`
		  \mathcal{E}_{u,\tilde{J}^{std}}`
		  \mathcal{E}_{u_X,J_X^{std}};
		  ``D_{\widehat{u},\widehat{J}}`D_{u,\tilde{J}^{std}}`D_{u_X,J_X^{std}}``]
 \efig$$

This has the consequence that $\tilde{J}^{std}$ is regular at $u$ (i.e. $D_{u,\tilde{J}^{std}}$ surjects)
whenever $\widehat{J}$ is regular at $\widehat{u}$ and $J_X^{std}$ is regular at $u_X$. Suppose now, that we
change $J_X^{std}$ by a compactly supported $C^\infty$--small perturbation avoiding $K\cap K^\prime$,
then we can replace $\tilde{J}^{std}$ in a neighbourhood of the $S^1$--fixed locus by the complex structure $\tilde{J}$
we get from $J_X$ and the Hermitian connection on $F$. With this done generically, for any $S^1$-invariant holomorphic $u$,
we have that $J_X$ is regular for $u_X$, so to prove regularity of $\tilde{J}$ at $u$, it suffices to show
that $\widehat{J}$ is always regular at $\widehat{u}$.

\begin{lemma}
 $\ker(D_{\widehat{u},\widehat{J}})=0$. Since the corresponding Maslov index is $0$, this implies regularity.
\end{lemma}

\begin{proof}
To do this we identify $u_X^*(E\times_{\mathbb{C}^*}(F\setminus 0))$ with $\mathbb{C}\oplus u_X^*(F)\oplus u_X^*(F^{-1})$ as in
Remark~\ref{remark:bundletriv}. We will take coordinates $a,b,c$ in the three summands respectively. With respect
to these coordinates, the section $\widehat{u}$ is the constant section at $(a_{crit},0,0)$. Hence,
any $v\in\ker(D_{\widehat{u},\widehat{J}})$ is a holomorphic section of $\mathbb{C}\oplus u_X^*(F)\oplus u_X^*(F^{-1})$
which is everywhere tangent to $\{a^3-ad+bc=z\}$ and over the boundary of the strip is tangent to the Lagrangian boundary conditions.
Note that $b$ and $c$ are valued in inverse line bundles, so the product $bc$ is a well-defined complex number.

At $(a_{crit},0,0)$ the tangents to $\{a^3-ad+bc=z\}$ are $u_X^*(F)\oplus u_X^*(F^{-1})$, since $\proja$ is critical
(so no tangent has a component in the $a$--direction). To first order, the Lagrangian boundary conditions
are of the form $\{(zp,e^{i\theta}\bar{z}p^{-1})|(p,p^{-1})\in u_X^*(P)\oplus u_X^*(P^{-1})\}$, where $\theta$ is fixed.
The values of $\theta$ are distinct for transverse Lagrangians.
Multiplying the $b$ and $c$ coordinates of $v$ gives a holomorphic map of the strip to $\mathbb{C}=u_X^*(F)\otimes u_X^*(F^{-1})$
with upper and lower boundary in $e^{i\theta}\mathbb{R}$ for distinct angles $\theta$. Such a map is necessarily zero.

This implies that at least one of the $b$ and $c$ coordinates of $v$ vanishes on a set large enough to apply unique continuation.
Hence, one coordinate is identically zero. However, from the expression for the tangents to the Lagrangian we see that, where
either of the $b$ and $c$ coordinates of $v$ vanish, both vanish, so $v\equiv 0$.
\end{proof}

To conclude the study of the moduli space
of $S^1$--invariant holomorphic strips, we have now proven:

\begin{lemma}
\label{lemma:fixreg}
 Let $J_X$ be any almost complex structure on $X$ which
\begin{itemize}
 \item makes the solutions to Floer's equation for $(K,K^\prime)$ regular,
 \item equals $J^{std}_X$ near $K\cap K^\prime$ and outside of a compact neighbourhood.
\end{itemize}
Define $\tilde{J}$ from $J_X$ and the Hermitian connection on $F$. Then $\tilde{J}$ is regular for all elements of the moduli space 
$\mathcal{M}_{\tilde{J}}^{S^1}(E\times_{\mathbb{C}^*}(F\setminus 0),L\times_{S^1}K,L^\prime\times_{S^1}K^\prime)$
of $S^1$--invariant holomorphic strips. Furthermore, this forms a component of the moduli space of all holomorphic strips (with
the same boundary conditions) which is canonically identified with (one or two copies of) the entire moduli space for $X,K,K^\prime$.
\end{lemma}

\begin{remark}
If we allow the use of non-regular $\tilde{J}$, we can instead conclude that the moduli of $S^1$--equivariant
holomorphic strips, as a Kuranishi space, is canonically identified with (one or two copies of) the entire moduli space for $X,K,K^\prime$.
Then a virtual perturbation preserving the $S^1$--action, shows that the only holomorphic strips which count for
Floer cohomology are those arising from the moduli space for $X,K,K^\prime$. Since we avoid going into detail about
Kuranishi spaces, and also relying on Kuranishi space technology, it is necessary to do the corresponding construction
with $S^1$--equivariant perturbations of almost complex structures.
\end{remark}

To complete the construction of a regular $\tilde{J}$ on $E_z\times_{\mathbb{C}^*}(F\setminus 0)$, we now show:

\begin{lemma}
\label{lemma:endreg}
 Let $\mathcal{J}^{S^1}_\tame(E_z\times_{\mathbb{C}^*}(F\setminus 0))$ be the space of time-dependent,
$S^1$--equivariant, almost complex structures $\tilde{J}_t$
on $E_z\times_{\mathbb{C}^*}(F\setminus 0)$, such that:
\begin{itemize}
 \item $\tilde{J}_t$ tames $\tilde{\Omega}$
 \item $\tilde{J}_t=\tilde{J}^{std}$ where $\proja$ is near a critical value
(in particular, near $L\times_{S^1}K\cap L^\prime\times_{S^1}K^\prime)$) and outside of a compact neighbourhood
 \item $\proja$ is $(\tilde{J}_t,i_t)$--holomorphic for some time dependent almost complex structure $i_t$ on $\mathbb{C}$
\end{itemize}
A generic choice of $\tilde{J}_t\in\mathcal{J}^{S^1}_\tame(E_z\times_{\mathbb{C}^*}(F\setminus 0))$ is regular for all
holomorphic strips with boundary in $L\times_{S^1}K,L^\prime\times_{S^1}K^\prime$.
\end{lemma}

\begin{proof}
By Lemma~\ref{lemma:fixreg}, $\tilde{J}_t\in \mathcal{J}^{S^1}_\tame(E_z\times_{\mathbb{C}^*}(F\setminus 0))$ is
automatically regular at any $S^1$--invariant holomorphic strip, so we now consider only non-invariant strips $u$.

The map $\proja \circ u$ is holomorphic (w.r.t. $i_t$) and maps the boundary of the strip to the curves $\gamma,\gamma^\prime$ to which
the Lagrangians $L,L^\prime$ project (see Figures~\ref{fig:infinitycurves} and \ref{fig:minusinfinitycurves}). $\proja \circ u$
must either be constant, or, for the input end, an embedding onto the region surrounded by $\gamma,\gamma^\prime$.
The singular fibres of $\proja$ are biholomorphic to the union of $F$ and $F^{-1}$ along the zero section and the Lagrangians
hit these fibres only at the zero section. This means that any holomorphic $u$ with $\proja \circ u$ constant and with
the Lagrangian boundary conditions is necessarily a constant map and so $S^1$--invariant.

We may now assume that $\proja \circ u$ embeds onto a neighbourhood in $\mathbb{C}$. Using this fact, we essentially follow
a well-known method of proof (see proof of \cite[Lemma 2.4]{seidel:les}, from which the underlying argument is copied).

Let $V\subset\mathbb{R}\times[0,1]$ be an open neighbourhood with $\bar{V}$ disjoint from the boundary of the strip,
which $\proja \circ u$ maps diffeomorphically onto its image in $\mathbb{C}$. We may assume that $V$ does not
contain any critical values of $\proja$, so $S^1$--acts freely on a neighbourhood of $u(V)$.

Let $\mathcal{T}$ be the subset of $T\mathcal{J}^{S^1}_\tame(E_z\times_{\mathbb{C}^*}(F\setminus 0))$ containing
those $Y$ which vanish except on the $S^1$--orbits of points in a neighbourhood of $u(V)$. Now consider the operator:

$$D_{u,\tilde{J}}^{univ}:\mathcal{W}^1_u\times\mathcal{T}\mapsto\mathcal{W}^0_{u,J}.$$

Here $\mathcal{W}^1_u$ is the $W^{1,p}$ completion of $T_u\mathcal{B}$ and $\mathcal{W}^0_{u,J}$ is the $L^p$ completion
of $\mathcal{E}_{u,J}$. $D_{u,\tilde{J}}^{univ}$ is the completion of the linearisation of the map
$\bar{\partial}^{univ}_{u,J}\co (u,J)\mapsto \bar{\partial}_J(u)$. Surjectivity of the operator $D_{u,\tilde{J}}^{univ}$
at all non-invariant, holomorphic $u$ implies the desired result.

Suppose that $D_{u,\tilde{J}}^{univ}$ does not surject. Then there exists $\eta\neq 0$, an $L^p$--section of the bundle dual to
$\mathcal{E}_{u,J}$, which is orthogonal to the image of $D_{u,\tilde{J}}^{univ}$. It satisfies

\[D^*_{u,J}\eta=0
\qquad \text{and}\qquad
\int_{\mathbb{R}\times[0,1]}\langle\eta,Y\circ Du\circ j\rangle=0
\quad \text{for}\quad
Y\in\mathcal{T}.\]

The first equation implies that $\eta$ is smooth away from the boundary. Suppose that $\eta\equiv 0$
on $V$, then unique continuation proves $\eta\equiv 0$ everywhere. Therefore, there must be some $x\in V$
with $\eta_x\neq0$. To derive a contradiction, and conclude the proof, we now construct $Y\in\mathcal{T}$
for which $Y\circ Du$ at $x$ is an arbitrary $(j,\tilde{J})$--antilinear map. Multiplying $Y$ by a bump
function supported near $x$, this then contradicts
$\int_{\mathbb{R}\times[0,1]}\langle\eta,Y\circ Du\circ j\rangle=0$.

\begin{remark}
\label{remark:complexpart}
For the following calculation, a quick remark on $\mathbb{R}$--linear maps between complex
vector spaces is helpful. Suppose $A$ is such a linear map, and $J_0, J_1$ are the linear
complex structures on the domain and range. Then $A^+=\frac{1}{2}(A-J_1AJ_0)$ is complex linear
and $A^-=\frac{1}{2}(A+J_1AJ_0)$ is complex anti-linear, with $A=A^++A^-$. We refer to these as the
complex and linear and anti-linear parts of $A$, respectively. This generalises straightforwardly to
the case where $A$ is an $\mathbb{R}$--linear map between different complex vector spaces with different
linear complex structures.
\end{remark}

Now, to construct $Y$, as above, it suffices to specify it only at the point $u(x)$. Splitting $TE_z$ as the sum
$\ker(D\proja)\oplus\ker(D\proja)^{\perp_{\tilde{\Omega}}}$, we can describe allowable values for $\tilde{J},Y,Du$ as follows:

\[\begin{array}{clll}
\tilde{J}=\left(\begin{array}{cc}
A&B\\
&i_t
\end{array}\right) && \text{where}\quad A^2=-I &\text{and}\quad B^+=0,\\
Y=\left(\begin{array}{cc}
R&S\\
&T
\end{array}\right) && \text{where}\quad S^+=\frac{1}{2}A(BT+RB) &\text{and}\quad R^+=T^+=0,\\
Du=\left(\begin{array}{c}
F\\
G
\end{array}\right) && \text{where}\quad F^-=\frac{1}{2}ABG &\text{and}\quad G^-=0.
\end{array}\]

Similarly a general complex anti-linear map $T(\mathbb{R}\times[0,1])\rightarrow T_{u(x)}(E_z\times_{\mathbb{C}^*}(F\setminus 0))$
is described by:

\[\begin{array}{clll}
\left(\begin{array}{c}
W\\
Z
\end{array}\right) && \text{where}\quad W^+=\frac{1}{2}ABZ &\text{and}\quad Z^+=0.\\
\end{array}\]

$G$ is an isomorphism, because $\proja\circ u$ is locally a diffeomorphism. This allows us to choose $Y$ as follows.
Take $R=0$, $T=ZG^{-1}$ and $S=\frac{1}{2}ABZG^{-1}=\frac{1}{2}ABT$, then

\[Y\circ Du = \left(\begin{array}{c}
W\\
Z
\end{array}\right)\]
Hence, as claimed, $Y\circ Du$ can be chosen to be any complex anti-linear map.
\end{proof}

\begin{corollary}
\[
\begin{array}{ccc}
HF^*(L_{out}\times_{S^1}P|_K,L_{out}^\prime\times_{S^1}P|_{K^\prime})&\simeq &\mathbb{Z}\otimes HF^*(K,K^\prime)\\
HF^*(L_{in}\times_{S^1}P|_K,L_{in}^\prime\times_{S^1}P|_{K\prime})&\simeq &(\mathbb{Z}\oplus\mathbb{Z})\otimes HF^*(K,K^\prime)
\end{array}
\]
\end{corollary}

\begin{proof}
Using the previous lemma, we choose a regular, $S^1$--equivariant, almost complex structure. The only isolated holomorphic strips
are those fixed by the $S^1$--action. However, these are in bijection with the isolated holomorphic strips used to calculate
$HF^*(K,K^\prime)$, hence the result. In fact, the same equations hold at the chain level.
\end{proof}

\subsubsection{Regularity over $D$}

We start with a choice of $\tilde{J}_t^{end}$ on the fibres over the boundary marked points, which is regular (for
the construction of Floer cohomology at either end). We may in fact take the same almost complex structure at both ends.
To produce a regular almost complex structure $\tilde{J}$ on the entire fibration $E\times_{\mathbb{C}^*}(F\setminus 0)$,
the construction is similar to that for the ends. First, we deal with regularity for $S^1$--invariant sections, then
we perturb the almost complex structure away from a neighbourhood of the $S^1$--invariant points to ensure regularity
for the remaining sections.

Recall that the Lagrangians $K,K^\prime$ lie in $\rho_X^{-1}[0,R)$, where $\rho_X$ is exhausting and plurisubharmonic.
We may assume that the almost complex structure $J_X$, chosen in the previous section, is standard (equal to $J_X^{std}$)
outside of $\rho_X^{-1}[0,R)$, so all holomorphic strips with boundary on $K,K^\prime$ are contained within
$\rho_X^{-1}[0,R)$.

\begin{definition}
\label{def:jontotalspace}
Let $j$ the standard complex structure on $D$.
We define the family $\mathcal{J}^{S^1}_\tame(E\times_{\mathbb{C}^*}(F\setminus 0))$ of almost complex structures
on $E\times_{\mathbb{C}^*}(F\setminus 0)$, within which we seek a regular almost complex structure as follows.
Again we introduce a time--coordinate $t$ on $D$ with boundary marked points removed, by identifying it with
$\mathbb{R}\times[0,1]$.
An almost complex structure $\tilde{J}$ is in this family if:
\begin{itemize}
 \item $\tilde{J}$ induces $\tilde{J}_t^{end}$ on the fibres at boundary marked points,
 \item $\tilde{J}$ is tame relative to $j$,
 \item $\tilde{J}=\tilde{J}^{std}$ outside of some compact set and in a neighbourhood of those Lagrangian intersection
points not in $M\times_{\mathbb{C}^*}(F\setminus 0)$,
 \item $\tilde{J}$ makes the $S^1$ invariant loci complex submanifolds,
 \item $\tilde{J}$ restricts to the part of the $S^1$--invariant locus which may contain sections (a copy of $D\times X$)
as the product of $j$ and $J_X$,
 \item $\tilde{J}$ is given by $\tilde{J}_t^{end}$ and $j$ on an open neighbourhood of set of $S^1$--invariant
points in $M\times_{\mathbb{C}^*}(F\setminus 0)|_{\rho_X^{-1}[0,R)}$
with respect to the trivialisation generated by the trivialisation of $M$.
\end{itemize}
\end{definition}

The penultimate point is important as it ensures that the moduli space of $S^1$--invariant holomorphic sections is
identical to the moduli space of holomorphic strips in $(X,J_X)$ with boundary on $(K,K^\prime)$. The final point
then ensures that this part of the moduli space of sections is also regular.
It remains to choose $\tilde{J}$ within this family, which is also regular for all non-invariant holomorphic sections
in order to show that these do not contribute to the relative invariant.

We follow the same argument as in the previous section, but it is easier here as the complex structure has more freedom
to vary.

\begin{lemma}
\label{lemma:totalreg}
Generic $\tilde{J}\in\mathcal{J}^{S^1}_\tame(E\times_{\mathbb{C}^*}(F\setminus 0))$ is regular.
\end{lemma}

\begin{proof}
Pick any $\tilde{J}$. Consider the moduli space $\mathcal{M}_\nonreg$ of holomorphic sections $u$ at which $\tilde{J}$ is not regular.
Such $u$ are necessarily non-invariant sections and form a closed subspace of the moduli space of all holomorphic sections. Using the
Gromov-Floer topology, and compactifying the moduli space by adding broken sections (i.e. a section plus holomorphic
strips at one or both ends), the closure $\overline{\mathcal{M}}_\nonreg$ of $\mathcal{M}_\nonreg$ is compact.
$\overline{\mathcal{M}}_\nonreg$ also does not contain any $S^1$--invariant holomorphic sections, even as broken sections.
This is because the
almost complex structure on the ends is regular and $\tilde{J}$ is regular at all $S^1$--invariant holomorphic
sections, so the gluing theorem shows that any $\tilde{J}$ is regular at any holomorphic section sufficiently
close (in the Gromov-Floer topology) to a broken section involving an $S^1$--invariant section.

By compactness of $\overline{\mathcal{M}}_\nonreg$, it is possible to choose
an open neighbourhood $W$ of the $S^1$--invariant points, such that for
any non-regular, holomorphic section $u$, there is an open neighbourhood $V_u$ in $D$ with
$\bar{V}_u\subset D\setminus\partial D$ and $u(V_u)$ disjoint from $W$.

The action of $S^1$ is free at all points which it does not fix and non-invariant sections $u$ are transverse
to $S^1$ orbits purely by merit of being sections. This means we can construct
first order deformations $Y$ of $\tilde{J}$ at a single point in $u(V_u)$, as in the proof of
Lemma~\ref{lemma:endreg}, and extend these to elements of
$T\mathcal{J}^{S^1}_\tame(E\times_{\mathbb{C}^*}(F\setminus 0))$ supported nearby.

At this point we split $TE\times_{\mathbb{C}^*}(F\setminus 0)$ as the sum of vertical vectors
$\ker(D\pi)$ and horizontal vectors $\ker(D\pi)^{\perp_{\tilde{\Omega}}}$, then:

\[\begin{array}{clll}
\tilde{J}=\left(\begin{array}{cc}
A&B\\
&j
\end{array}\right) && \text{where}\quad A^2=-I &\text{and}\quad B^+=0,\\
Y=\left(\begin{array}{cc}
R&S\\
&0
\end{array}\right) && \text{where}\quad S^+=\frac{1}{2}ARB &\text{and}\quad R^+=0,\\
Du=\left(\begin{array}{c}
F\\
I
\end{array}\right) && \text{where}\quad F^-=\frac{1}{2}AB.&
\end{array}\]

A generic complex anti-linear map $TD\rightarrow u^*(TE\times_{\mathbb{C}^*}(F\setminus 0))$
with purely vertical image is:

\[\left(\begin{array}{c}
Z\\
0
\end{array}\right) \quad \text{where}\quad Z^+=0,\]

\noindent and

\[Y\circ Du=\left(\begin{array}{c}
RF+S\\
0
\end{array}\right).\]

\noindent It suffices to choose $S^-=Z-R(F^+)$, because $S^+$ cancels with $R(F^-)=\frac{1}{2}RAB=-\frac{1}{2}ARB$.
\end{proof}

\subsection{Two calculations of the relative invariant}
\label{section:s1calc}

In this Section we calculate up to sign the relative invariant in
two cases using the fibration $E\times_{\mathbb{C}^*}(F\setminus 0)$.
These are the two cases needed in
Section~\ref{section:stab} to understand the stabilisation and destabilisation maps.

For the stabilisation map we use the Lagrangian boundary condition $\tilde{Q}$
(see Section~\ref{section:sympassocbundle}). For the destabilisation map, we modify the
Lagrangian boundary condition $Q$ on $E$ in Section~\ref{section:E} and proceed with the
same constructions as in Sections~\ref{section:sympassocbundle} and \ref{section:s1equij}.
The curves $\gamma$, $\gamma^\prime$ describing the Lagrangian boundary conditions for the
stabilisation and destabilisation maps are illustrated in Figure~\ref{fig:destabstab}.

\begin{figure}[h]
\centering\scalebox{0.55}{\input{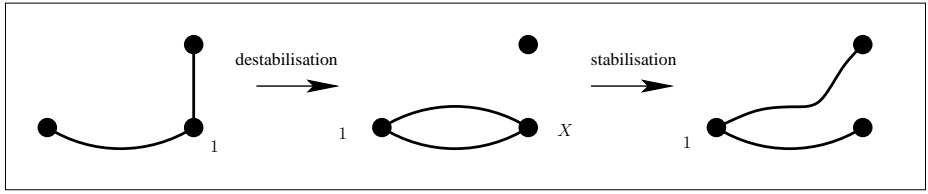}}
\caption{An illustration of the image under $\proja$ of the
$S^1$--equivariant Lagrangians for the two particular relative invariants. These
correspond, as labelled, to parts of the stabilisation and destabilisation maps
for symplectic Khovanov homology (see Section~\ref{section:stab}).
Generators of the chain complexes for Floer cohomology are labelled.
It should be noted, that the setting for the destabilisation map has been rotated, so
that the monodromy clockwise around the base of the corresponding fibration gives
a positive half twist in a line joining the top and left-hand critical points.}
\label{fig:destabstab}
\end{figure}

The Floer cochain complex at input and output ends of $E\times_{\mathbb{C}^*}(F\setminus 0)$
splits as either of

\[\mathbb{Z}\otimes CF^*(K,K^\prime),\]
\[(\mathbb{Z}\oplus\mathbb{Z})\otimes CF^*(K,K^\prime),\]

\noindent depending on whether the curves $\gamma,\gamma^\prime$ describing the Lagrangian boundary condition
intersect at one or both ends. To see this, we simply observe that the Lagrangian intersections form
one or two copies of $K\cap K^\prime$, and that isolated holomorphic strips are $S^1$--invariant and these
correspond to (one or two copies of) those counted for $CF^*(K,K^\prime)$.

To fit in with the notation used in the comparison to Khovanov homology, we prefer to write
$\mathbb{Z}[X]/(X^2)$ in place of $\mathbb{Z}\oplus\mathbb{Z}$, where $X$ has degree $2$. Then:

\begin{proposition}
\label{prop:stabmap}

The relative invariants for the stabilisation and destabilisation map split
at the chain-level as the identity map on $CF^*(K,K^\prime)$
and the following maps on the other factor:

\[
\begin{array}{ccccccc}
	&\text{destabilisation}&&&&\text{stabilisation}&\\
	\mathbb{Z}&\to&\mathbb{Z}[X]/(X^2)&\qquad\qquad\qquad&\mathbb{Z}[X]/(X^2)&\to&\mathbb{Z}\\
	1&\longmapsto&X&&X&\longmapsto&0\\
	&&&&1&\longmapsto&1\\
\end{array}
\]
\end{proposition}

\begin{proof}
Using an $S^1$--equivariant regular almost complex structure, we observe that the isolated holomorphic
sections are precisely those which are $S^1$--invariant and have constant projection to $X$. The
remaining holomorphic sections come in non-trivial $S^1$--orbits or $\mathbb{R}$--orbits. In $E$
there is only one smooth, $S^1$--equivariant section and this is holomorphic.
Correspondingly there is precisely one isolated holomorphic section of $E\times_{\mathbb{C}^*}(F\setminus 0)$
with boundary on $\tilde{Q}$ (in the case of stabilisation or destabilisation) for each
point in $K\cap K^\prime$.

It remains only to check the relative Maslov index of the Lagrangian intersections of $Q$ at
the ends of $E$ in order to work out which copy of $\mathbb{Z}$ maps to which. This can be done
using $\proja$, and agrees with the labelling in Figure~\ref{fig:destabstab}.
\end{proof}

\section{Symplectic Khovanov homology}
\label{section:khsymp}

\subsection{A short summary of symplectic Khovanov homology}
\label{section:introkhsymp}

This section summarises the construction of $\skh $ with
emphasis on those parts of most relevance later.
The construction is based on the observation that braids on $n$ strands can
be viewed as loops in the configuration space $\operatorname{Conf}_n(\mathbb{C})$
of $n$ points in the plane.

Suppose we have a symplectic fibration
$\chi:S\to\operatorname{Conf}_n(\mathbb{C})$ on which symplectic parallel transport
maps are well-defined. Then symplectic
parallel transport defines a map
\[\pi_1(\operatorname{Conf}_n(\mathbb{C}),P)\to\pi_0(\sympc(\chi^{-1}(P)))\] for any base point
$P\in\operatorname{Conf}_n(\mathbb{C})$. This is a representation of the braid
group. Ideally one would also require the symplectic parallel transport maps
to be compactly supported.

This sort of symplectic geometry is not new. Khovanov and Seidel
\cite{khovanovseidel} show that similar symplectic geometry can be used
to construct Khovanov's categorification of the Burau representation of
the braid group.

We use instead a fibration with subtly different properties.
Let $\conf{n}$ be the space of configurations with coordinates summing
to 0 and let $\confbar{n}\cong\mathbb{C}^{n-1}$ be its closure
in $\operatorname{Sym}_n(\mathbb{C})$.
Seidel and Smith \cite{seidelsmith:khsymp} define a singular holomorphic fibration of a
Stein manifold $S_n$ over $\confbar{2n}$.
$S_n$ has the property that it pulls back to an exact MBL--fibration over any holomorphic
disc in $\confbar{2n}$ passing transversely through the locus of
$\confbar{2n}\setminus \conf{2n}$ where precisely two coordinates meet
(cf.\ Lemma~27 of \cite{seidelsmith:khsymp}).

Let $\overline{\mathbb{D}}$ be such a disc with a single singular fibre and let
$C$ be a bounded subset of the singular locus of the pullback fibration over
$\overline{\mathbb{D}}$.
Then, for simple enough paths $\gamma$ (short and linear in $\overline{\mathbb{D}}$ suffices),
the vanishing locus $V_C$ to $C$, in nearby regular fibres over $\gamma$,
is well-defined (see \cite[Section 4.2]{seidelsmith:khsymp}).
We refer to $V_C$ as the relative vanishing cycle to $C$.
If $C$ is open, then $V_C$ is a bounded coisotropic, fibring over $C$ with fibres all
isotropic spheres.

The monodromy of the elementary
braid performing a single positive twist in the two chosen coordinates
is realised by a loop around the singular value of $\overline{\mathbb{D}}$.
If one is careful to make the loop small, then the monodromy map is
well-defined over an open neighbourhood of $V_C$.
By a deformation of the exact symplectic structure, the monodromy can be taken to 
restrict to a neighbourhood of $V_C$ as a fibred Dehn twist in the case that
$C$ is open. This is a
consequence of work of Perutz \cite{perutz:matchinv}. Non-compactness of the
singular locus and of $V_C$ are a problem here, so we shall not 
make direct use of this description, however, it is useful heuristically.

The path into the singular value of $\overline{\mathbb{D}}$ corresponds naturally
to a `singular braid' of the form of the $(n-2,n)$--tangle in
Figure~\ref{fig:captangle} (right hand side). Heuristically, we think of $V_C$, for large enough
$C$, as corresponding to this tangle. The loop around the singular value corresponds
to the elementary braid (as illustrated on the left hand side). Extending $V_C$ over this
loop by symplectic parallel transport, one sees already an important invariance result. Namely, $V_C$
is unchanged, up to an appropriate isotopy. This corresponds to the $(n-2,n)$--tangle
being unchanged, up to isotopy, by composition with the braid.

\begin{figure}[h]
\begin{center}
\scalebox{0.8}{\includegraphics{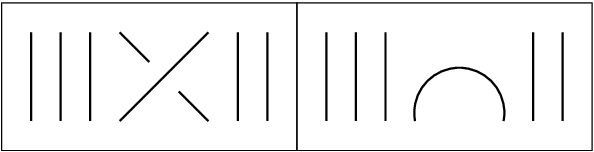}}
\end{center}
\caption{An elementary braid (left) and a basic ``cap'' tangle (right)
which is unchanged by the action of that braid.}
\label{fig:captangle}
\end{figure}

\begin{definition}
\label{def:polynomial}
Let $S_n$ be the space of degree $n$ monic polynomials with coefficients in the ring
of 2 by 2 matrices in $\mathbb{C}\left[X\right]$ such that
the coefficient of $X^{n-1}$ is trace free. The determinant map $\chi$
restricted to $S_n$ always gives monic polynomials of degree $2n$ with roots
summing to zero, so can be
thought of as a map to $\confbar{2n}$
(by identifying monic polynomials with the configuration of their roots counting
multiplicities). This is an example of the type of fibration just
discussed.
\end{definition}

The original definition of $S_n$ was as a nilpotent slice of the Lie
algebra $\mathfrak{sl}_{2n}(\mathbb{C})$, specifically a local transverse
slice to the orbit of the Adjoint action of $\operatorname{SL}_{2n}(\mathbb{C})$
at a nilpotent matrix with precisely 2 Jordan blocks of size $n$. This setting
is described more explicitly as follows.

\begin{definition}
\label{def:nilpslices}
One defines
$S_n$ to consist of the following $2n$ by $2n$ matrices with complex coefficients:

\[
 \left(
\begin{array}{cccc}
A_1&I&&\\
\vdots&&\ddots&\\
\vdots&&&I\\
A_n&&&0\\
\end{array}
\right)
\]

\noindent where each $A_i$ is a $2$ by $2$ matrix (i.e.\ in $\mathfrak{gl}_{2}(\mathbb{C})$)
and $A_1$ is trace free (in $\mathfrak{sl}_{2}(\mathbb{C})$). Then the map $\chi$ is
defined to give the characteristic polynomial of the matrix.
\end{definition}

The two definitions are related by the isomorphism:
\[
  \left(
\begin{array}{cccc}
A_1&I&&\\
\vdots&&\ddots&\\
\vdots&&&I\\
A_n&&&0\\
\end{array}
\right)
\longmapsto
X^n-\sum_{i=1}^{n} X^{n-i}A_i
\]
which commutes with the map $\chi$ on either side. In either setting we shall
denote the fibre of $S_n$ over a configuration $P$ by $\sfibre{n}{P}$.

Alternatively, it has been shown by Manolescu
\cite{manolescu:hilb} that, for $P\in\conf{2n}$, the fibre $\sfibre{n}{P}$ injects
holomorphically
into the Hilbert scheme $\operatorname{Hilb}^n(M_n(P))$. Here, $M_n(P)$ is the 2--dimensional
(over $\mathbb{C}$) Milnor fibre associated to the $A_{2n}$ singularity. It is
explicitly described as a smooth affine surface
by the equation $u^2+v^2+P(z)=0$ in $\mathbb{C}^3$.
The symplectic parallel transport maps (defined appropriately on compact subsets containing
the relevant Lagrangian submanifolds) are, in an appropriate sense,
lifts of Dehn twists on $M_n(P)$.
This reveals a close connection between the
symplectic geometry of $S_n$ with the braid group $Br_{2n}$ which acts naturally
by Dehn twists on the $A_{2n}$ Milnor fibre.

With an appropriate choice of K\"ahler metric \cite{seidelsmith:khsymp} uses
``rescaled'' symplectic parallel transport maps. For this paper however, it is more
convenient to use actual symplectic parallel transport. This has the
consequence that we often need to perform some deformation in order to ensure that the 
symplectic parallel transport is well-defined on relevant compact subsets. Given any
smooth path in $\confbar{2n}$ we define the symplectic parallel transport maps
on sufficiently large compact subsets of a given fibre by the deformation described in
Lemma~\ref{lemma:sptdef}. Specifically Remark~\ref{remark:pathspt} covers this case.

The singular locus of $S_n$ corresponding to two coordinates in
$\confbar{2n}$ meeting at zero is precisely $\critmbl{\chi}$
(see Section~\ref{section:flattening}).

Furthermore, it is shown that $S_{n-1}$ is canonically isomorphic to a singular
locus in $S_n$ over the points
$(0,0,\mu_3,\ldots,\mu_{2n})$  in a manner compatible with the
fibrations. Local neighbourhoods have a particularly nice form.

\begin{lemma}[cf.\ Lemma 27 of \cite{seidelsmith:khsymp}]
\label{lemma:a1}
Let $D$ be a disc in
$\overline{\operatorname{Conf}}^0_{2n}(\mathbb{C})$
given by the monocoordinates $(-\sqrt{\epsilon},\sqrt{\epsilon}, \mu_3, \ldots,\mu_{2n})$
with $\epsilon$ small. In terms of monomials, the elements of $D$ have the form
\[(X^2-\epsilon)\Pi_{k=3}^{2n}(X-\mu_k).\]

Then there is a neighbourhood of
$S_{n-1}\subset \chi^{-1}(D)\cap S_n$ on which the fibration $\chi$
has the local model
given by a neighbourhood of
$\sfibre{n-1}{(\mu_3,\ldots,\mu_{2n})}\times{0}$ below:

\[
\bfig
 \Square[
	\chi^{-1}(D)`
	\sfibre{n-1}{(\mu_3,\ldots,\mu_{2n})}\times\mathbb{C}^3`D`\mathbb{C};
	\text{local isomorphism near } S_{n-1}\subset\chi^{-1}(D)\cap S_n`
	\chi`a^2+b^2+c^2`
	(-\sqrt{\epsilon},\sqrt{\epsilon}, \mu_3, \ldots,\mu_{2n})
	\mapsto\epsilon
	]
 \efig
\]
\end{lemma}

In fact, one can drop the requirement that the two
coordinates meet at zero (rather than an arbitrary point in $\mathbb{C}$)
in the case that $n\ge 2$.

Let $L$ be a simply connected, compact, exact Lagrangian submanifold of $\sfibre{n-1}{(\mu_3,\ldots,\mu_{2n})}$.
Then the relative vanishing locus to $L$ (over simple enough paths into the singular
value of the model fibration in the lemma) gives
an exact Lagrangian submanifold of a nearby regular fibre $\sfibre{n}{u(z)}$ of
$S_n$ (cf. \cite[Section 4.2]{seidelsmith:khsymp}). It is diffeomorphic to $L\times S^2$
and well-defined up to Lagrangian isotopy, which, since $\pi_1(L\times S^2)=0$, implies it is
well-defined up to compactly supported Hamiltonian isotopy.
Rezazadegan \cite{reza} has an approach to the non-compact case, but it
is not necessary here.

By starting with $L=\{0\}=\sfibre{0}{0}$ and repeating this construction
we can generate compactly supported Hamiltonian isotopy classes of Lagrangian submanifolds
corresponding to isotopy classes of certain $(0,2n)$--tangles.

\begin{definition}
Let $\gamma_i$ for $i=1,2,\ldots,n$ be a sequence of vanishing paths in
${\operatorname{Conf}}^0_{2i}(\mathbb{C})$ from non-singular values $z_i$
to singular values $w_i$. Suppose also that, for each $i=1,2,\ldots,n-1$
we have $X^2z_i=w_{i+1}$. Then we can consider the composition of the paths
$X^{2(n-i)}\gamma_i$ as a piecewise smooth path in
$\confbar{2n}$. We call such a path an
\emph{iterated vanishing path}. Viewed in terms of configurations of points in $\mathbb{C}$,
such a path describes a way of bringing together $2n$ points into pairs (fixed at 0).

Repeating the relative vanishing cycle construction along a short enough iterated
vanishing path, starting with a single point inside $\{0\}\cong \sfibre{0}{0}$,
allows one to construct a Lagrangian \emph{iterated vanishing cycle}. Given a
longer iterated vanishing path, this construction also works, though it becomes
necessary to perform a deformation (described in Lemma~\ref{lemma:sptdef}
and Remark~\ref{remark:pathspt}) to ensure that
symplectic parallel transport of the vanishing loci is always well-defined.
\end{definition}

It is shown in \cite{seidelsmith:khsymp} that iterated vanishing cycles
(up to isotopy) are independent of isotopy of the
iterated vanishing paths defining them.

\begin{definition}
A \emph{crossingless matching} of $P\in\operatorname{Conf}_{2n}(\mathbb{C})$
is an embedded set of $n$ curves in $\mathbb{C}$, such that each coordinate of
$P$ is an endpoint of precisely one curve. We denote the set of crossingless
matchings on $P$ by $\match{n}{P}$.
\end{definition}

\begin{figure}[h]
\begin{center}
\scalebox{0.7}{\includegraphics{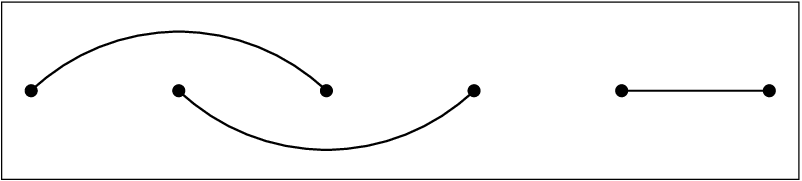}}
\end{center}
\caption{An example crossingless matching for $P\in\operatorname{Conf}_{6}(\mathbb{C})$.}
\label{fig:crossinglessmatching}
\end{figure}

It suffices to specify a crossingless
matching $A\in\match{n}{P}$ together with an ordering on the curves
in order to determine an iterated vanishing
cycle up to isotopy. To construct the iterated vanishing cycle,
one takes any
\(\func
	{\gamma}
	{\left[0,1\right]}
	{\overline{\operatorname{Conf}}^0_{2n}(\mathbb{C})}
\)
whose coordinates in $\confbar{2n}$ sweep out
the curves of $A$ in $\mathbb{C}$, bringing together endpoints of the curves
in the chosen order. Choices of such $\gamma$ are all isotopic
to each other and these isotopies induce exact isotopies of the iterated
vanishing cycles. In fact, changing the ordering also changes
the iterated vanishing cycle \cite{seidelsmith:khsymp} only by exact isotopy.

By \cite{alexander:braidpos} any link $\mathfrak{L}$ can be put into
\emph{braid position} by isotopy, i.e.\ a position of the form shown in
Figure~\ref{fig:braidpos} for some braid $\beta$ on $n$ strands. This position
can be thought of as splitting into two identical crossingless matchings
(on some $P\in\conf{2n}$)
and a braid on $2n$ strands consisting of $\beta$ on the left-hand $n$ strands
and the trivial braid on the others as in the picture.

\begin{figure}[h]
\begin{center}
\scalebox{0.4}{\input{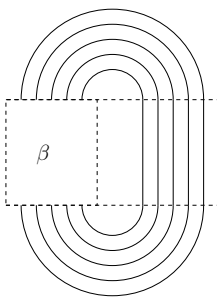}}
\end{center}
\caption{Braid position of a link, split as two identical crossingless
matchings and a braid, trivial on the strands on one side.}
\label{fig:braidpos}
\end{figure}

Let $L\subset\sfibre{n}{P}$ be the Lagrangian corresponding to the
crossingless matching and $L^\prime$ be its image under the parallel
transport corresponding to $\beta$. Then the \emph{symplectic Khovanov
homology} is defined as $\skh (\mathfrak{L}):=HF(L,L^\prime)$.

It follows from general results on Floer cohomology under symplectomorphisms
and Hamiltonian isotopy, that $\skh $ is independent of braid isotopy.
To prove that $\skh $ did not depend on the choice of braid representing
a given link, Seidel and Smith show that it is invariant under the two
\emph{Markov moves}. These are changes to the braid which suffice to pass
between any two braid positions of the same link.

Invariance under conjugation of the braid (\emph{Markov I}) works
straightforwardly by a trick involving isotopy of crossingless matchings.

The \emph{Markov II} move (Figure~\ref{fig:markov2original}) involves adding
an extra strand to the braid $\beta$ and ``twisting it in'' by either a positive
or a negative half twist. The proof of invariance under this move is more
involved and makes use of a local model for the fibration of $S_n$ over a
neighbourhood of a point in
$\confbar{2n}$ with three coordinates
coinciding.

\begin{figure}[h]
\begin{center}
\scalebox{0.4}{\input{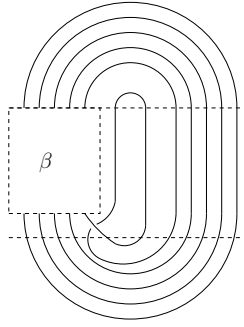}}
\end{center}
\caption{The braid position of Figure~\ref{fig:braidpos} after application
of the Markov II move.}
\label{fig:markov2original}
\end{figure}

Another way in which parts of $S_{m-1}$ are
locally nested in $S_m$ is as follows. Let $P\in\conf{2(m-1)}$ be a polynomial with
pairwise distinct roots, one of which is $0$.
Then it turns out that the singular locus of $S_m$ over $X^2P$ is precisely
$X^2\sfibre{m-1}{P}$.
There is also a simple holomorphic model for neighbourhoods of this locus inside $S_m$.
Namely one can describe it as a holomorphic associated bundle over $\sfibre{m-1}{P}$
in the following way.

\begin{lemma}[cf. Lemma 29 of \cite{seidelsmith:khsymp}]
\label{lemma:a2}
Let $XP\in\conf{2(m-1)}$ be a polynomial with pairwise distinct roots
one of which is $0$.
Let $\mathbb{D}^2$ be a small holomorphic bidisc in
$\confbar{2m}$ parametrised by
\[(d,z)\mapsto (X^3-Xd+z)P\]

\noindent There is a holomorphic line bundle $\mathcal{F}$ over
$\sfibre{m-1}{XP}$ such that the following is a local model for
$\chi^{-1}(\mathbb{D}^2)$.

\[
\bfig
 \Square[
	\chi^{-1}(\mathbb{D}^2)`
	(\mathcal{F}\setminus 0)\times_{\mathbb{C}^*}\mathbb{C}^4`
	\mathbb{D}^2`\mathbb{C}^2;
	\text{local isomorphism near } X^2\sfibre{m-1}{XP}\cap \chi^{-1}(0,0)`
	\chi`
	f`
	]
 \efig
\]

\noindent The map f is induced by the map $\mathbb{C}^4\to\mathbb{C}^2$
taking $(a,b,c,d)\mapsto(d,a^3-ad+bc)$. The $\mathbb{C}^*$--action
on $\mathbb{C}^4$ is given by $(a,b,c,d)\mapsto(a,\zeta^{-2}b,\zeta^2c,d)$ for $\zeta\in\mathbb{C}^*$.

\end{lemma}

We have stated a simple version of the lemma.
In fact Lemma~29 of \cite{seidelsmith:khsymp} is more general in that it
gives a local model near the fibre over a polynomial $P^\prime\in\confbar{2m}$
with root $\mu$ of multiplicity 3 and the remaining roots pairwise distinct.
The difference is that $\mu$ is not required to be zero. This generality is
necessary (though only in some cases).

The use of this local model is as follows. One considers the fibration
over discs defined by a fixed value of $d\neq 0$.
$\mathbb{C}^3\times\{d\}\to \mathbb{C}$ is an exact Lefschetz fibration
with two singular values. These each correspond to moving two of the three
roots of $(X-\mu)^3-(X-\mu)d+z$
together in different ways. This makes $(\mathcal{F}\setminus 0)\times_{\mathbb{C}^*}\mathbb{C}^3$
an exact Morse Bott Lefschetz fibration over $\mathbb{C}$.

Let $K,K^\prime$ in $\sfibre{m-1}{XP}$ be the iterated vanishing cycles
constructed for some braid position of a link.
One performs, with each of $K,K^\prime$, the
relative vanishing cycle construction in
$(\mathcal{F}\setminus 0)\times_{\mathbb{C}^*}\mathbb{C}^3$
to create $L,L^\prime$. This
corresponds to adding a new unlinked component to the link. Then
$L$ is carried around the other singularity by parallel transport
to $\tau L$, corresponding to twisting the new component into the
braid as in Figure~\ref{fig:markov2original}. When this is done compatibly with
splitting of the fibration from
the lemma one finds a bijection of the moduli spaces used to
calculate $HF(K,K^\prime)$ and $HF(\tau L,L^\prime)$.

With an appropriate grading shift, this takes care of invariance under
the Markov $II$ move. The important part to notice is
that, by Lemma~\ref{lemma:a2}, the fibration splits in such a way as to
separate the local behaviour of the three coordinates involved from the
rest of the link. This is studied in more detail later.

\subsection{Symplectic Khovanov homology for links in bridge position}

Here we generalise the definition of $\skh $
such that it is defined directly using any bridge position of a link,
not necessarily just braid positions. This will be useful in defining
the maps on $\skh $
corresponding to smooth closed cobordisms in $\mathbb{R}^4$ between
links. These maps will be covered in the next section.

Any link in $\mathbb{R}^3=\{(x,y,z)\}$ is isotopic to one in which the height
function (mapping points on the link to the value of their
$z$--coordinate) has only non-degenerate critical points. Furthermore
we may require all local maxima of $z$ to occur where $z>0$ and
minima where $z<0$. We will call such a position of a link \emph{admissible}.
Each half of a link in an admissible position (i.e.\ the parts above and below
height $z=0$) specifies a path for the iterated vanishing cycle construction
and hence a Lagrangian submanifold of the same fibre $\sfibre{n}{P}$. One should
note that these paths need not be embedded or disjoint from each other.

An isotopy of a link through admissible positions fixing the point
$P\in\operatorname{Conf}^0_{2n}(\mathbb{C})$ causes Hamiltonian isotopies
of this pair of Lagrangians.
On Floer cohomology these give the identity up to canonical isomorphisms
(from continuation maps). Although the Floer cohomology
may arise from a different chain complex, before and after the isotopies,
it is well-defined up to canonical isomorphism.

Now we extend this definition to bridge diagrams of links and use it to
prove that the Floer cohomology one gets for any two bridge diagrams of the
same link is always isomorphic (though not with a canonical isomorphism unless
one is comparing bridge positions coming from isotopic admissible positions).

\begin{figure}[h]
\begin{center}
\scalebox{0.5}{\includegraphics{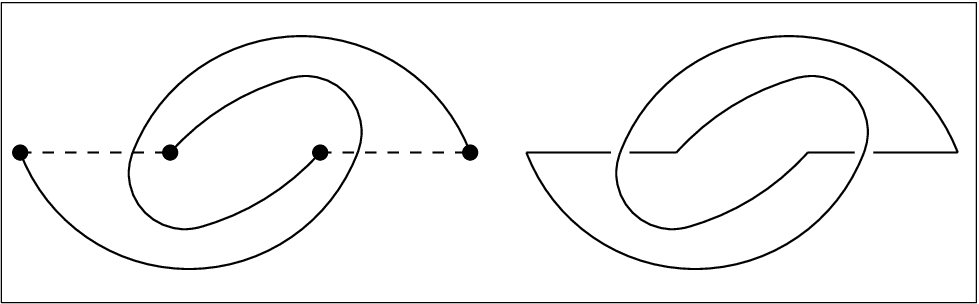}}
\end{center}
\caption{The Hopf link as a bridge diagram and a projection of the admissible
position the diagram represents ($\alpha$--curves are drawn smooth,
$\beta$--curves dashed).}
\label{fig:hopflink}
\end{figure}

\begin{definition}
An \emph{$n$--bridge diagram} $(P,A,B)$ of a link consists of a set
$P$ of $2n$ points in the plane together with a pair of
crossingless matchings $A,B$ (sets of $n$ disjoint curves in the plane
which join the points of $P$ in pairs) which intersect
transversely away from $P$. We refer to the
curves in these matchings as \emph{$\alpha$-- and $\beta$--curves}
respectively and to the points of $P$ as the \emph{vertices} of the diagram.
\end{definition}

Any bridge diagram is easily turned into an admissible link by pulling the
$\alpha$--curves up out of the plane in which the diagram is drawn and
pushing the $\beta$--curves down. Such an admissible position will project
back onto the bridge diagram by the orthogonal projection onto the plane.
In fact any two such admissible positions are isotopic by an isotopy
of admissible positions fixing $P$.
Therefore we have a canonical choice of isomorphisms between the
Floer homologies they induce.

\begin{definition}
Given an admissible position $\mathfrak{L}$ for a link or bridge diagram $(P,A,B)$,
we define its \emph{symplectic Khovanov homology} $\skh (\mathfrak{L})$ or
$\skh (P,A,B)$ to be the Lagrangian Floer cohomology $HF(L_A,L_B)$. Here $L_A$, $L_B$ are constructed
by the iterated vanishing cycle construction for a path specified by
$\mathfrak{L}$, $(P,A,B)$ respectively.
\end{definition}

From the above discussion, we see that this gives a well-defined invariant of
admissible positions, up to isotopy fixing the point
$P\in\operatorname{Conf}^0_{2n}(\mathbb{C})$, and also of bridge diagrams up
to isotopy fixing $P$.

A \emph{regular isotopy} of bridge diagrams is an isotopy $P(t)$ of $P$
together with isotopies of $A$ and of $B$ through crossingless matchings
on $P(t)$. The intermediate matchings need not intersect transversely.

A \emph{passing move} on a crossingless matching is where we take a smooth
loop $\gamma$ disjoint from the curves of the matching and enclosing precisely
one of the curves and replace another of the curves with its connect sum
with $\gamma$.

A \emph{stabilisation move} on an $n$--bridge diagram is as follows. We mark a
closed interval on an $\alpha$--curve $\gamma$ which is disjoint from
the $\beta$--curves and from $P$. We then add its two endpoints to $P$,
add the interval to $B$, replace $\gamma$ in $A$ with the two curves
we get by removing the interior of the interval from $\gamma$.
The resulting $(P,A,B)$ is an $(n+1)$--bridge diagram for the same link.

\begin{figure}[h]
\begin{center}
\scalebox{0.5}{\includegraphics{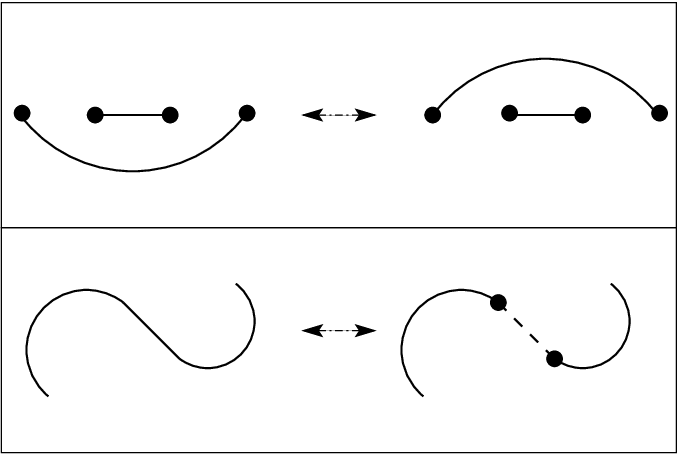}}
\end{center}
\caption{The \emph{passing move}
(above) and the \emph{stabilisation move} (below).}
\label{fig:bridgemoves}
\end{figure}

Any bridge diagram $D$ naturally yields a crossing diagram
$\operatorname{proj}(D)$ by projecting
the construction in $\mathbb{R}^2\times [-1,1]$ to $\mathbb{R}^2$. There
are, however, many different bridge diagrams that yield the same
crossing diagram by this method. All such diagrams are related by a
finite sequence of stabilisation moves (and destabilisation moves)
and regular isotopy fixing the projection.

\begin{lemma}
Finite sequences of the three moves described above applied to any
bridge diagram $D$ suffice to perform all Reidemeister moves on
$\operatorname{proj}(D)$.
\end{lemma}

\begin{figure}[h]
\begin{center}
\scalebox{0.45}{\includegraphics{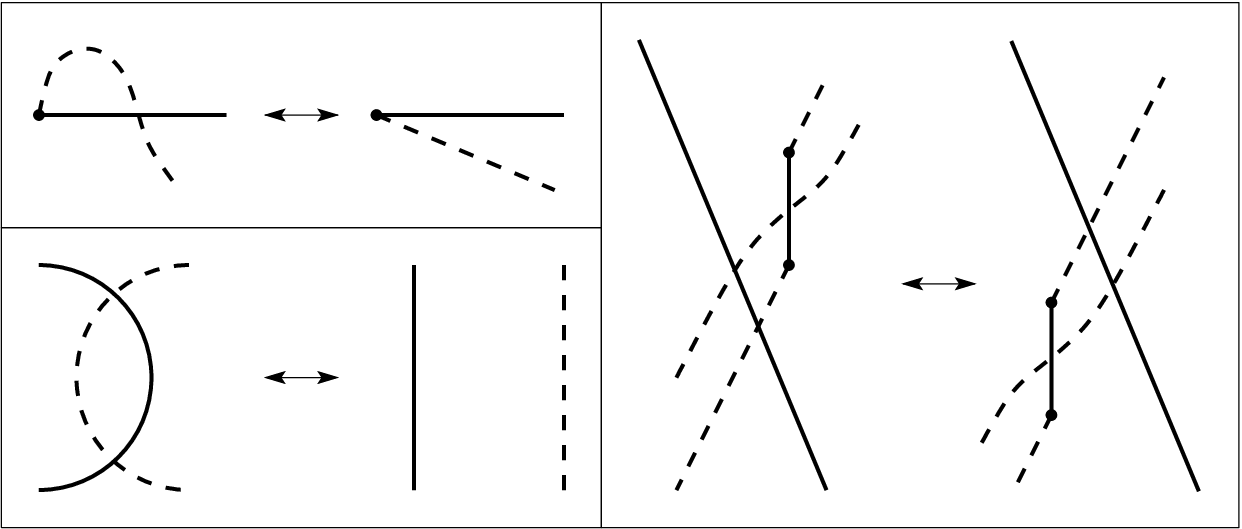}}
\end{center}
\caption{The Reidemeister moves (one way up) after some destabilisation
moves are the projections of the above moves.}
\label{fig:reidemeister}
\end{figure}

\begin{proof}
By repeated destabilisation, the neighbourhood in $\mathbb{R}^2$ in
which one performs a Reidemeister move can be made to be one of those
in Figure~\ref{fig:reidemeister} (up to a reflection and/or swapping the
$\alpha$-- and $\beta$--curves). The illustration makes it clear that
isotopy suffices to perform the Reidemeister I and II moves and the
passing move to perform the Reidemeister III move.
\end{proof}

\begin{lemma}
Finite sequences of these three moves suffice to go between any two bridge
diagrams of the same link.
\end{lemma}

\begin{proof}
Given $D_1$, $D_2$, bridge diagrams of the same link, there is a sequence
of Reidemeister moves from $\operatorname{proj}(D_1)$ to
$\operatorname{proj}(D_2)$. By the previous lemma, there is a finite sequence of
moves taking $D_1$ to some $D_3$ with
$\operatorname{proj}(D_3)=\operatorname{proj}(D_2)$. Now a finite
sequence of stabilisation and destabilisation moves relates $D_3$ and
$D_2$.
\end{proof}

Isotopy of $P$ to $Q$ along a path in $\operatorname{Conf}^0_{2n}(\mathbb{C})$
gives an exact symplectomorphism from (any compact subset of) $\sfibre{n}{P}$ to $\sfibre{n}{Q}$
by symplectic parallel transport.
Given an isotopy of admissible links which induces this same isotopy $P$
to $Q$ on its intersection with the plane $z=0$,
this symplectomorphism carries the Lagrangians corresponding to the first
admissible link to those for the second (up to Hamiltonian isotopy). Hence,
it gives an isomorphism of symplectic Khovanov homology.
Using this, one sees that regular isotopy of bridge diagrams and also the
passing move do not change the isomorphism class of symplectic Khovanov
homology.

\begin{lemma}
Stabilisation of a bridge diagram gives a bridge diagram with isomorphic
symplectic Khovanov homology.
\end{lemma}

Modified by an isotopy of admissible links,
the Markov $II^+$ move is just the stabilisation move
(see Figure~\ref{fig:m2stabilisation}).
The proof of invariance under the Markov $II^+$ move in
\cite{seidelsmith:khsymp} can be carried out in exactly the same manner
in this setting since there is no part of it that requires the admissible
links defining the Lagrangians to be of the particular form used in
that paper.

The only difference is technical. The localisation to coordinates of the form given
in Lemma~\ref{lemma:a2} must be carried out differently if we insist on using
the construction of Lemma~\ref{lemma:sptdef} instead of rescaled symplectic parallel
transport in the iterated vanishing cycle construction.
A strictly stronger version of the above result, compatible with
the methods of this paper, is proven in Section~\ref{section:stab},
so we will not go into detail here.

\begin{figure}[h]
\begin{center}
\scalebox{0.5}{\includegraphics{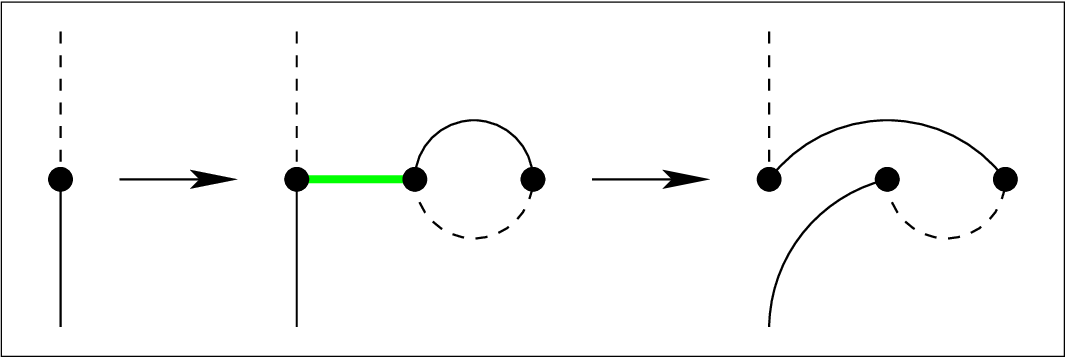}}
\end{center}
\caption{An illustration of the Markov $II^+$ move as seen in a bridge
diagram. In the first step a new link component is created,
then in the second step `twisted in' by a positive twist along the thicker line.}
\label{fig:m2stabilisation}
\end{figure}

In conclusion it has been shown in this section that:

\begin{theorem}
\label{theorem:linkinv}
The isomorphism class of $\skh $ of a bridge diagram as a relatively
graded group is an invariant of link isotopy.
\end{theorem}

\begin{remark}
It is worth noting here that the Markov $II^-$ move can equally well be
used in the above. Namely the crossingless matchings one gets by replacing
the positive twist in Figure~\ref{fig:m2stabilisation} by a negative twist
differ by an isotopy. In the setting of \cite{seidelsmith:khsymp}
$HF(L_A,L_B)$ carries an absolute grading which the Markov $II^\pm$ moves
change by different, but fixed, amounts. An overall grading shift from
the writhe of the link counteracts this to give an absolute grading on
$\skh $. To define the same absolute grading for links in bridge
position will take a little more care as these observations show that
the absolute grading on $HF$ should not be preserved by all isotopies
of admissible links. In fact a correction needs to be made for the passing
moves to get the absolute grading on $\skh $ in this setting.
\end{remark}

\begin{remark}
By stretching all the local maxima in the upper half of an
admissible link upwards one can decompose it as a trivial crossingless
matching and a braid. Now considering the braid group action by
diffeomorphisms on the plane containing this trivial matching, one
finds that the upper tangle is isotopic to a tangle whose projection
to the $(x,y)$--plane has no crossings. Doing the same with the lower
half, the link can be taken by isotopy of the two halves
to one which projects to a bridge diagram.
Hence Theorem~\ref{theorem:linkinv} also holds for $\skh $ of any 
admissible link.
\end{remark}

\section{Maps on symplectic Khovanov homology from smooth cobordisms}
\label{section:cobordisms}

In the setting of Khovanov homology one has maps between the chain complexes
corresponding to elementary saddle cobordisms and creation/annihilation (`cap'/`cup')
cobordisms as portrayed in a given crossing diagram. In this section we construct analogous
maps on $\skh $ and prove that, up to an overall sign ambiguity, they
compose to give maps depending only on the isotopy class of the composite
cobordism. Unless otherwise indicated, isotopies of a cobordism must fix the
links at either end of the cobordism, since we generally want the symplectic
Khovanov homology of the domain and range to be well-defined up to canonical
isomorphism.

In the case of cobordisms which are themselves isotopies of links through
admissible positions there is little difficulty.
Namely, corresponding to isotopies of admissible links which fix the
intersection with the plane $z=0$ we have continuation maps between Floer
cohomology groups. Suppose we have an isotopy
of admissible positions which moves this intersection along a path
$\func{\gamma}{\mathbb{R}}{\conf{2n}}$ which is constant near $\pm\infty$.
Then we can construct an admissible map

\[u_\gamma:\mathbb{R}\times[0,1]\to\mathbb{R}\to\conf{2n}\]

This defines an exact MBL--fibration (cf. Section~\ref{section:flattening})
with two ends. Extending Lagrangians along the boundaries of this fibration
allows us to calculate a relative invariant mapping between the symplectic
Khovanov homologies of the admissible positions at either end of the isotopy.

In fact this map on $\skh $ is the same as the map induced by the symplectic parallel
transport along $\gamma$, but it is handy to have it
given in terms of a symplectic fibration.

\begin{remark}
Deforming the admissible map $u_\gamma$ gives, by Lemma~\ref{lemma:isotopyoffibrations},
the following result.
Isotopies which are isotopic through isotopies of admissible positions of links
(fixing the links at either end) give the same map on symplectic Khovanov
homology. One should note however, that there exist
isotopies from a link to itself which do not give the identity map on Khovanov
homology \cite[Theorem 1]{jacobsson:cob}.
\end{remark}

\subsection{Saddle cobordisms}
\label{subsection:saddles}

In this section, we construct maps which correspond to certain saddle cobordisms.
We begin by outlining
how saddle cobordisms naturally arise from admissible maps into $\confbar{2n}$.

Let $\mathfrak{L}$ be a link in an admissible position which intersects the $z=0$--plane in
the configuration $P$. Up to some choice, this defines (and is defined by)
iterated vanishing cycles $\gamma_A,\gamma_B:[0,1]\to\confbar{2n}$. We shall
assume both paths end in a segment of constant path at $P$.

For the moment, we shall restrict
attention to link cobordisms supported close to $\{z=0\}$.
The parts of the trivial cobordism from $\mathfrak{L}$ to itself away from $\{z=0\}$
can be realised as the following sets:
\[
\{(t,z,\lambda)\in[-1,1]\times[1,2]\times\mathbb{C}~:~\text{$\lambda$ is a root of $\gamma_A(z-1)$}\},
\]
\[
\{(t,z,\lambda)\in[-1,1]\times[-2,-1]\times\mathbb{C}~:~\text{$\lambda$ is a root of $\gamma_B(1-z)$}\}.
\]
We now exhibit some link cobordisms by joining these surfaces up
with a surface in the missing region $[-1,1]\times[-1,1]\times\mathbb{C}$.

Let $u:\overline{\mathbb{D}}\to\confbar{2n}$ be an admissible map (see
Definition~\ref{def:admismap})
with two boundary marked points $\pm 1$, such that $u(-1)=P$. To make the
cobordism we construct be smooth, we also require $u$ to be a constant
map near each of the boundary marked points. Now choose your
favourite smooth map $[-1,1]\times[-1,1]\to\overline{\mathbb{D}}$ which is a diffeomorphism
away from the boundaries and maps $\{\pm 1\}\times[-1,1]$ to $\pm 1$ respectively.
We shall call the composition of $u$ with this map $\tilde u$.

The following set defines a braid cobordism from the trivial braid
to some other braid fixing the configuration $P$ at the ends:
\[
\{(t,z,\lambda)\in[-1,1]\times[-1,1]\times\mathbb{C}~:~\text{$\lambda$ is a root of $\tilde u(t,z)$}\}.
\]
Hence, it can be inserted in the above construction to give a link
cobordism, starting at $\mathfrak{L}$, which is supported near $\{z=0\}$.

Now we describe how to construct an admissible map corresponding to certain
elementary saddle cobordisms starting at the link $\mathfrak{L}$. As input
to the construction, take an embedded path $\delta:[0,1]\to\mathbb{C}$ ending at
roots of $P$ and otherwise disjoint from them. The link cobordism
will be the elementary braid cobordism which adds a single negative half-twist
along the curve $\delta$. For reasons of orientation of holomorphic embeddings
in $\confbar{2n}$, the
sign of the half-twist one adds along $\delta$ will be necessarily negative. However, this
problem is actually an artefact of the viewpoint of braid cobordisms. There
is no sign associated to elementary saddle cobordisms of links.

By fattening the curve $\delta$ one can define a smooth map
$v:\overline{\mathbb{D}}\to\mathbb{C}$
which:
\begin{itemize}
 \item maps $\pm 1$ to the ends of $\delta$,
 \item otherwise misses all the roots of $P$,
 \item and is a holomorphic embedding near $0\in\overline{\mathbb{D}}$
\end{itemize}
The space of choices of such a smooth map is contractible.

Let $\tilde P$ be the product of $X-r$ over all roots $r$ of $P$ which are not ends of $\delta$.
Then we define the admissible map $u_\delta$ to be the following composite:
\[
\bfig

\node 1a(0,300)[\overline{\mathbb{D}}]
\node 2a(1200,300)[\overline{\operatorname{Conf}}^0_2(\overline{\mathbb{D}})]
\node 3a(2400,300)[\confbar{2n}]
\node 1b(0,150)[z]
\node 2b(1200,150)[X^2+z]
\node 3b(1200,0)[(X-a)(X-b)]
\node 4b(2400,0)[(X-v(a))(X-v(b))\tilde P]

\arrow[1a`2a;]
\arrow[2a`3a;]
\arrow/@{|->}/[1b`2b;]
\arrow/@{|->}/[3b`4b;]
\efig
\]

One can check that this is admissible. We fix an input boundary marked point
at $-1$ and output boundary marked point at $1$.
In particular $u(-1)$ is the configuration $P$. In order to have this admissible
map well-defined up to isotopy fixing the boundary marked points, we shall
generally require that the $v$ is constant on the segments of $\partial\overline{\mathbb{D}}$
between 1 and $i$ and between $-1$ and $-i$. This causes the output boundary marked point
to also be $P$ and the path in $\confbar{2n}$ given by $u$ applied to the
lower half of $\partial\overline{\mathbb{D}}$ to be contractible through $\conf{2n}$.
Hence, the lower half (where $z<0$) of the link one has at the output end of the cobordism
is the same as the lower half at the input. This is a convenient simplification
when one wishes to express these cobordisms in terms of bridge diagrams.

\begin{figure}[h]
\begin{center}
\scalebox{1}{\includegraphics{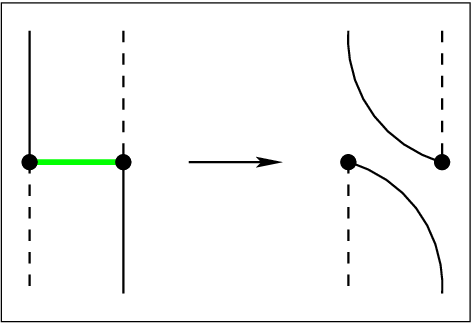}}
\end{center}
\caption{An elementary saddle cobordism specified by a curve in a bridge diagram.
When projected to a crossing diagram, this example is also the usual elementary
saddle cobordism between resolutions of a crossing.}
\label{fig:saddlecob}
\end{figure}

In a bridge diagram $(P,A,B)$ for $\mathfrak{L}$ we draw the curve $\delta$
in grey to indicate a cobordism about to be performed. Figure~\ref{fig:saddlecob} gives an example
of this with an arrow pointing to the bridge diagram for the output link.
The simplification mentioned above allows us to fix $P$ and apply the entire half-twist
to the curves from $A$, without changing $B$.
The cobordism illustrated in Figure~\ref{fig:saddlecob} is an elementary saddle cobordism
of the more familiar kind under projection of the bridge diagram to a crossing diagram.
From now on we shall view cobordisms ``from above'' in this manner.

It is now a simple matter to define the map $f_\delta$ that a cobordism specified by a curve
$\delta$ induces on symplectic Khovanov homology.

\begin{definition}
Given a bridge diagram $(P,A,B)$ together with a curve $\delta$, as above, specifying a cobordism
to the bridge diagram $(P^\prime,A^\prime,B^\prime)$ we define the map
\[f_\delta:\skh (P,A,B)\to \skh (P,A,B)\]
to be the relative invariant induced by the admissible map $u_\delta$ and 
the pair of Lagrangians used to define $\skh (P,A,B)$. This relative invariant
is well-defined by Lemma~\ref{lemma:isotopyoffibrations}.
\end{definition}

\begin{remark}
The relative invariant of this definition is well-defined up to composition
with the canonical isomorphisms of $\skh $ of the domain and range. To show
this one simply applies Lemma~\ref{lemma:isotopyoffibrations}.
It is independent of isotopy of $\delta$ and of changing of $(P,A,B)$ to any
other bridge diagram representing the same admissible link position (i.e.
isotopy fixing $P$ and the passing move).
\end{remark}

Composing two admissible maps given by elementary saddle cobordisms
(by gluing striplike ends of the admissible maps)
gives the composite
of the maps on symplectic Khovanov homology. The two singular values for
the composite admissible map can be moved past each other by an isotopy of
admissible maps (see Figure~\ref{fig:baserestriction}.
Decomposing the result into two admissible maps shows that the map
on symplectic Khovanov homology can be expressed in different ways as
a composite. This is a simple version of Lemma~\ref{lemma:baserestriction}.
It will be vital later to know which composites
of saddle cobordisms give the same maps in this way.

\begin{figure}[h]
\begin{center}
\scalebox{0.7}{\input{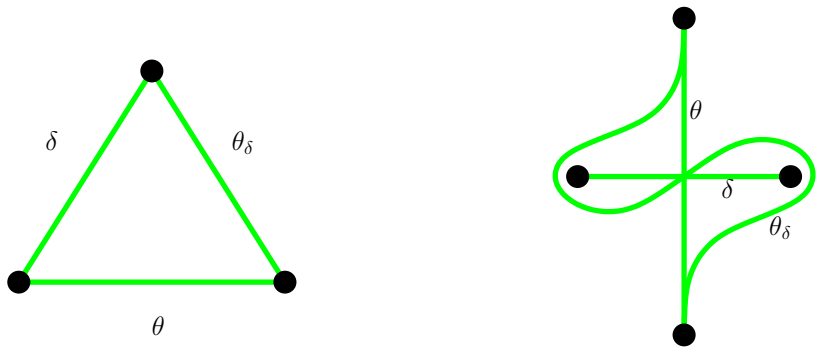}}
\end{center}
\caption{Two simple examples of curves $\delta$, $\theta$ and $\theta_\delta$ in the plane satisfying the
conditions of Lemma~\ref{lemma:commuting}.}
\label{fig:commuting2}
\end{figure}

\begin{lemma}
\label{lemma:commuting}
Let $\delta$, $\theta$ be embedded curves in
$\mathbb{C}\setminus{P}$, each joining two points of
$P\in\operatorname{Conf}^0_{2n}(\mathbb{C})$. Write $\theta_\delta$ for
the result of a positive half twist along $\delta$ applied to $\theta$. Then
the maps these curves specify on any $\skh (P,A,B)$ satisfy the following relation:
\[f_\theta\circ f_\delta = f_\delta \circ f_{\theta_\delta}\]
\end{lemma}

\begin{proof}
One glues the pairs of admissible maps for each composite together and observes that
the resulting admissible maps are isotopic by an isotopy which moves the singular
values past each other.
\end{proof}

This is a minor abuse of notation. For example, the maps $f_\delta$ have different domain
and range on either side of the equation. However, the composites map between the same
(up to canonical isomorphism) cohomology groups.

\begin{corollary}
Suppose $\delta$ and $\theta$ do not intersect in $\mathbb{C}$. Then the cobordisms
specified by $\delta$ and $\theta$
can be performed in either order with the same resulting map on $\skh (P,A,B)$.
\end{corollary}

\subsection{A semi-canonical splitting of symplectic Khovanov homology}
\label{section:splitting}

In this section, we generalise the splitting of Lemma~\ref{lemma:a1} in
a way which allows one to split symplectic Khovanov homology of a link
consisting of multiple unlinked components as a product. This is done
in the case of 2 components, where one is an unknot, in Lemma~41 of
\cite{seidelsmith:khsymp}. We then strengthen the result by showing the splitting
is well-defined up to canonical isomorphism on one factor and
furthermore that one can split certain maps $f_\gamma$ induced by saddle
cobordisms in a similar manner.

In the following we will use both descriptions of the fibrations $S_k$ given
in Definitions~\ref{def:polynomial} and \ref{def:nilpslices} and freely switch
between them. It should be clear from context which is meant.

We start the construction by observing that there is a natural choice of
copy of $S_k$ in $S_{k+l}$, namely $X^lS_k$. Similarly
$X^{2l}\confbar{2k}\subset\confbar{2(k+l)}$ and these two inclusions are compatible
with the fibrations. That is, for $y$ in $S_k$ we have $X^{2l}\chi (y)=\chi (X^ly)$.

Choose any polynomial $P=\prod_i(X-\mu_i)\in\operatorname{Conf}^0_{2k}(\mathbb{C^*})$.
Let $\mathfrak{\xi}$ be the linear subspace of $\mathbb{C}^{2k+1}$ given by
the equation $2lz_0+\sum_{i=1}^{2k}z_i=0$.
There is a neighbourhood of $X^lP\in\confbar{2(k+l)}$ which is
biholomorphic to a neighbourhood of $(X^{2l},0,P)$ in
$\confbar{2l}\times\mathbb{C}\times\confbar{2k}$
and of $(X^{2l},(0,\mu_1,\ldots,\mu_k))$ in $\confbar{2l}\times\mathfrak{\xi}$.
Explicitly, the biholomorphisms are described (in the opposite order) by
\begin{eqnarray*}
	\left(\prod_{i=1}^{2l}(X-a_i),(z_0,z_1,\ldots,z_{2k})\right)
	&\longmapsto&
	\left(\prod_{i=1}^{2l}(X-a_i),z_0,\prod_{i=1}^{2k}(X-z_i-\frac{lz_0}{k})\right)\\
	&\longmapsto&
	\prod_{i=1}^{2l}(X-a_i-z_0)
	\prod_{i=1}^{2k}(X-z_i). \\
\end{eqnarray*}

We will now show that, for $y\in \sfibre{k}{P}$ we can locally model any
transverse slice to $X^ly\in S_{k+l}$ on $S_l\times\mathfrak{\xi}$
compatibly with the description of a neighbourhood of $X^lP$
as $\confbar{2l}\times\mathfrak{\xi}$. In fact these models can be chosen
in a consistent manner for all $y\in \sfibre{k}{P}$ simultaneously:

\begin{proposition}[cf.\ Lemma~\ref{lemma:a1}]
\label{prop:semisplit}
There is a tubular neighbourhood of $X^l\sfibre{k}{P}$ in $S_{k+l}$ which
is biholomorphic to a neighbourhood of $\sfibre{k}{P}$
in $\sfibre{k}{P}\times S_l\times\mathfrak{\xi}$. Furthermore the following
diagram of holomorphic maps commutes:

\[
\bfig

\node 1a(0,500)[S_{k+l}]
\node 2a(2000,500)[\sfibre{k}{P}\times S_l\times\mathfrak{\xi}]
\node 1b(0,0)[\confbar{2(k+l)}]
\node 2b(2000,0)[\confbar{2l}\times\mathfrak{\xi}]

\arrow[1a`1b;\chi]
\arrow[2a`2b;\chi\times\id{\mathfrak{\xi}}]
\arrow[2a`1a;\text{local $\cong$ defined near }X^l\sfibre{k}{P}]
\arrow[2b`1b;\text{local $\cong$ defined near }X^{2l}P]
\efig
\]
\end{proposition}

We shall in general actually use the following simpler version (restricting to $0\in\mathfrak{\xi}$)
which suffices to describe the behaviour of links with multiple unlinked components.

\begin{corollary}
\label{corollary:semican}
Given any $P\in\operatorname{Conf}^0_{2k}(\mathbb{C}^*)$ we have the following local holomorphic
model for $S_{k+l}$ defined:
\begin{itemize}
 \item on a neighbourhood of $X^l\sfibre{k}{P}$ in $\chi^{-1}(P\confbar{2l})\subset S_{k+l}$
 \item over a neighbourhood of $PX^{2l}$ in $P\confbar{2l}\subset\confbar{2(k+l)}$
\end{itemize}

\[
\bfig

\node 1a(2000,1000)[\sfibre{k}{P}\times S_l]
\node 2a(0,1000)[\chi^{-1}(P\confbar{2l})]
\node 1b(2000,500)[\confbar{2l}]
\node 2b(0,500)[P\confbar{2l}]

\arrow[1a`1b;\chi]
\arrow[2a`2b;\chi]
\arrow[1a`2a;\text{local $\cong$ defined near } X^l\sfibre{k}{P}]
\arrow[1b`2b;\text{local $\cong$ defined near } PX^2l]

\efig
\]

The map of bases takes $Q$ to $PQ$. The map of total spaces restricts to
$\sfibre{k}{P}\times\{X^l\}$ as multiplication $(y,X^l)\mapsto yX^l$.
\end{corollary}

The proof of Proposition~\ref{prop:semisplit} comes from generalising
Lemmas~24 to 27 of \cite{seidelsmith:khsymp} and will be presented in
an analogous manner for ease of comparison. The slices $S_n$ are thought of here as
spaces of $2n\times 2n$ matrices.

\begin{lemma}
\label{lemma:kernel}
For $y\in S_{k+l}$, projection to the first $2l$ coordinates of
$\mathbb{C}^{2(k+l)}$ restricts to an injection
\[\ker (y^l) \to \mathbb{C}^{2l}\]
\end{lemma}

\begin{proof}
For $y\in S_{k+l}$ we observe that $y^l$ is of the form

\[
\left(
\begin{array}{cc}
A & I \\
B & 0
\end{array}
\right)
\]

\noindent where $A$ is a $2l\times 2k$ matrix and
$B$ is a $2l\times 2l$ matrix. Hence, if
$(v_1,\ldots,v_{2(k+l)})\in\ker (y^l)$, then we have:

\[
-A\left(
\begin{array}{c}
v_1\\
\vdots\\
v_{2l}
\end{array}
\right)
=
\left(
\begin{array}{c}
 v_{2l+1}\\
\vdots\\
v_{2(k+l)}
\end{array}
\right)
\]

\noindent Therefore the first $2l$ coordinates determine the rest.
\end{proof}

\begin{lemma}
The subspace of $y\in S_{k+l}$ such that $\ker (y^l)$ is $2l$--dimensional
can be canonically identified with $S_k$. The identification is compatible
with the adjoint quotient map and the inclusion
$\confbar{2k}\subset\confbar{2(k+l)}$ previously mentioned.
\end{lemma}

\begin{remark}
This lemma describes the inclusion, written earlier as $X^lS_k\subset S_{k+l}$.
In particular it shows that for $P\in\operatorname{Conf}^0_{2k}(\mathbb{C}^*)$
the set $X^lS_k$ can be described as the transverse intersection of
$S_{k+l}\subset\mathfrak{sl}_{k+l}(\mathbb{C})$ with an orbit of the adjoint
action of $SL_{k+l}(\mathbb{C})$. This orbit is furthermore given by the property
that its elements have 2 Jordan blocks, each of size $l$, for the eigenvalue
0. Hence transverse slices to elements of the orbit should contain \emph{$(l,l)$--type
nilpotent slices}, i.e.\ copies of $S_l$. The proof of Proposition~\ref{prop:semisplit}
will be based upon this observation.
\end{remark}

\begin{proof}
Let $y\in S_{k+l}$ have kernel of dimension $2l$ over $\mathbb{C}$.
We consider the Jordan blocks for $y^l$ for eigenvalue 0.
Lemma~\ref{lemma:kernel} in the case $l=1$ says that there are at most two.
Hence there must be precisely 2 and they must both have size at least $l$.
This means the injection of Lemma~\ref{lemma:kernel} is an isomorphism for
such $y$.

We will write vectors $v\in(\mathbb{C}^2)^{(k+l)}$ with coordinates
$v_1,\ldots,v_{(k+l)}\in\mathbb{C}^2$. Now, observe that $y^i(v)_1=M+v_{i+1}$
where $M$ depends only on $v_1,\ldots,v_i$. We prove, by induction on $i$, that
$A_{k+1-i}=0$.

In the case $i=1$, we have $y(v)_k=A_kv_1$. In order for there to be a $v\in\ker(y)$
for any choice of $v_1$, we must have $A_k=0$. For $i>1$ (by applying
$A_{k-i+2},\ldots,A_k=0$), we have:
\begin{eqnarray*}
y^i(v)_{k+1-i}&=&y\left(y^{i-1}(v)\right)_{k+1-i}\\
	&=&A_{k+1-i}(y^{i-1}(v)_1)+y^{i-1}(v)_{k+2-i}\\
	&=&A_{k+1-i}(M+v_i)+A_{k+2-i}(y^{i-2}(v)_1)+y^{i-2}(v)_{k+3-i}\\
	&\vdots&\\
	&=&A_{k+1-i}(M+v_i)
\end{eqnarray*}

Here again $M$ depends only on $v_1,\ldots,v_{i-1}$, so in order for there to
be a $v\in\ker(y^i)$ for any choice of $v_1,\ldots,v_i$ we must have $A_{k+1-i}=0$.

We have just shown that the space of $y\in S_{k+l}$ with $\dim\ker(y^l)=2l$ is
contained in the space of those $y$ with $A_{k+1},\ldots,A_{k+l}=0$. In fact one
can check they are the same. This space is canonically isomorphic to $S_k$ by the
map

\begin{equation}
\label{eqn:nestedfibres}
f:
\left(
\begin{array}{cccc}
A_1&I&&\\
\vdots&&\ddots&\\
\vdots&&&I\\
A_k&&&\\
\end{array}
\right)
\longmapsto
\left(
\begin{array}{cccccc}
A_1&I&&&&\\
\vdots&&\ddots&&&\\
A_k&&&\ddots&&\\
0&&&&\ddots&\\
\vdots&&&&&I\\
0&&&&&0\\
\end{array}
\right)_.
\end{equation}

\noindent This has the property that $X^{2l}\chi(y)=\chi(f(y))$, so it commutes with
the inclusion \[X^l\confbar{2k}\subset\confbar{2(k+l)}\]
Also note that this map
takes $y$ to $X^ly$ in the polynomial presentation of $S_{k+l}$.
\end{proof}

With these results we now prove Proposition~\ref{prop:semisplit}.

\begin{proof}
First we choose for every $y\in X^l\sfibre{k}{P}$ a local affine linear transverse
slice varying holomorphically with $y$.
It suffices (Lemma~5 (i) of \cite{seidelsmith:khsymp}) to pick
these slices transverse to the adjoint orbit of $y$ within $S_{k+l}$, i.e.\ from a complement
to $T_yX^l\sfibre{k}{P}$ in $T_yS_{k+l}$.
For example one can take the following.
First choose a complement $W_y$ to $T_y\sfibre{k}{P}$ in $T_yS_k$
varying holomorphically with y. This splitting problem has a positive solution
because $\sfibre{k}{P}$ is an affine subvariety of $S_k$. Then take the direct sum of
$V\oplus X^lW$ with the orthogonal complement to $X^lS_k$ as a vector subspace
of $S_{k+l}$.

Another way to choose a transverse slice to $X^ly$ is (by Lemma~8 of
\cite{seidelsmith:khsymp}) to decompose $\mathfrak{sl}_{2(k+l)}(\mathbb{C})$ into
eigenspaces of the semisimple part $(X^ly)_s$ of $X^ly$. All the eigenspaces
are 1--dimensional with the exception of the $0$--eigenspace which is canonically
identified with $\mathbb{C}^{2l}$ by Lemma~\ref{lemma:kernel}.
This gives a holomorphically varying local transverse slice.

\[
\bfig

\node 1a(0,1000)[X^ly_s+S_l+\mathfrak{z}]
\node 1b(0,500)[\confbar{2l}\times\mathfrak{\xi}]

\arrow[1a`1b;\chi\times f]

\efig
\]

\noindent where $f$ is a local biholomorphism
$\mathfrak{z}\to\mathbb{C}\times \operatorname{Conf}^0_{2k}(\mathbb{C}^*)$.

Any two local transverse slices at a point are locally isomorphic by an isomorphism that
moves points only inside their adjoint orbits (Lemma~5 (iii) of \cite{seidelsmith:khsymp})
and hence does not affect $\chi$. This
gives us the required result so long as this isomorphism can be chosen to vary holomorphically
with $y$. In fact the isomorphism depends only on a choice of local submanifold
$K_y\subset SL_{k+l}(\mathbb{C})$ near the identity element $e$ with tangent space at $e$
complementary to that of the adjoint orbit of $y_s$.
The choice $K_y=\exp [\mathfrak{sl}_{k+l}(\mathbb{C}),y_s]$ suffices
(cf. proof of Lemma~27 of \cite{seidelsmith:khsymp}).
\end{proof}

We now demonstrate how one can split symplectic Khovanov homology using appropriate
\emph{localisation} arguments. We begin with a useful, if obvious,
lemma related to shrinking components of a link.

\begin{lemma}
\label{lemma:shrinklink}
Let $D$ be a disc centred at $0\in\mathbb{C}$ and let
\[P=P_1P_0\in\overline{\operatorname{Conf}}^0_{2k}(\mathbb{C}\setminus D)\times\operatorname{Conf^0_{2l}}(D)\subset\confbar{2(k+l)}\]
Suppose we have a bridge diagram $(P,A,B)$ which splits as the union of
bridge diagrams $(P_0,A_0,B_0)$ supported on $D$ and $(P_1,A_1,B_1)$ supported
outside $D$. See Figure~\ref{fig:splitbridge} for an example.
Then there is a canonical isotopy of bridge diagrams starting at this bridge diagram and
ending at one which splits as the union of $(P_1,A_1,B_1)$ with $(P_0,A_0,B_0)$,
the latter scaled in $\mathbb{C}$ by any $\lambda\in(0,1)$. Hence, rescaling
defines canonical isomorphisms on $\skh $ (or canonical chain homotopy classes
of chain maps at the cochain level).

One can scale down any admissible map of a surface with boundary marked points into
\[P_1\operatorname{Conf^0_{2l}}(D)\subset\confbar{2(k+l)}\] to obtain another admissible map,
such that the relative invariant of the two surfaces are related by composing on either side
by the canonical isomorphisms mentioned above.
\end{lemma}

\begin{figure}[h]
\begin{center}
\scalebox{0.45}{\input{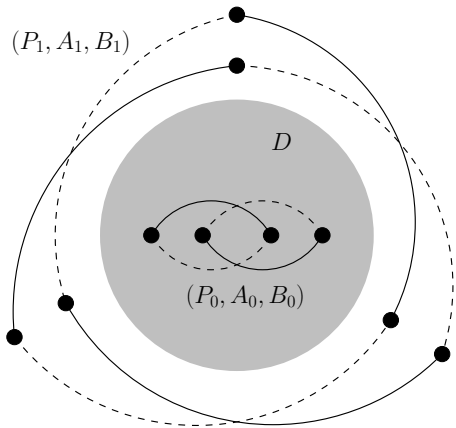}}
\end{center}
\caption{A diagram of a link with multiple unlinked parts arranged as specified in
Lemma~\ref{lemma:shrinklink} and Theorem~\ref{theorem:splitskh}.}
\label{fig:splitbridge}
\end{figure}

\begin{proof}
The first part is obvious. For the second part one just needs to check that
the $\mathbb{C}^*$--action on $\confbar{2n}$, rescaling all roots by a complex number,
preserves the property of a surface being admissible. The result comes from the isotopy
of admissible curves one gets by continuously performing the scaling and
composing on either end by the isotopies from the first part.
\end{proof}

\begin{theorem}
\label{theorem:splitskh}
Let $D$ be a disc centred at $0\in\mathbb{C}$.
Suppose we have a bridge diagram $(P,A,B)$ which splits as the union of
bridge diagrams $(P_0,A_0,B_0)$ supported on $D$ and $(P_1,A_1,B_1)$ supported
outside $D$.

Then $\skh (P,A,B)$ is isomorphic to the K\"unneth product of $\skh (P_0,A_0,B_0)$ and
$\skh (P_1,A_1,B_1)$. Furthermore, this isomorphism is canonical
up to composition with automorphisms of the $\skh (P_0,A_0,B_0)$--factor.
\end{theorem}

\begin{remark}
We only apply this theorem in cases where one factor has no torsion (e.g. it is an
unlinked union of unknots), so the K\"unneth product on cohomology
is an honest tensor product.
\end{remark}

We shall denote by $\Psi$ the local isomorphism given in Corollary~\ref{corollary:semican}
\[\sfibre{k}{P_1}\times S_l\to^\Psi \chi^{-1}(P_1\confbar{2l})\]

The idea is now to \emph{localise} the calculation of Floer cohomology
to the model region $\sfibre{k}{P}\times S_l$ with Lagrangians and exact
symplectic structures which respect the splitting. In very specific cases (where
$(P_0,A_0,B_0)$ is the simplest bridge diagram for a single unknot, or also a union
of such bridge diagrams), one can use
the same argument as for Lemma~41 of \cite{seidelsmith:khsymp}.

The proof of Theorem~\ref{theorem:splitskh} below necessarily uses a more
complicated argument to control the position of the Lagrangian iterated vanishing cycles
to ensure they are contained in the region where $\Psi$ is defined.
Otherwise it follows essentially the same argument.
We begin by outlining the steps of the proof.

We start with the following information:
\begin{itemize}
 \item Configurations $P_1\in\confbar{2k}$ and $P_0\in\confbar{2l}$,
 \item Iterated vanishing paths $\gamma_{A_1},\gamma_{B_1}$ in $\confbar{2k}$
	representing the crossingless matching $(P_1,A_1,B_1)$,
 \item Iterated vanishing paths $\gamma_{A_0},\gamma_{B_0}$ in $\confbar{2l}$
	representing the crossingless matching $(P_0,A_0,B_0)$,
 \item Any choice of exhausting plurisubharmonic functions $\rho_l$, $\rho_k$
and $\rho_{k+l}$ on $S_l$, $S_k$, $S_{k+l}$ respectively, such that $\rho_k$ is
the restriction of $\rho_{k+l}$ to $X^lS_k\subset S_{k+l}$,
 \item Exact K\"ahler structures
	\( \Omega_k=d\Theta_k,\Omega_l=d\Theta_l,\Omega_{k+l}=d\Theta_{k+l} \)
	defined by the above plurisubharmonic functions,
 \item Our favourite choices of Lagrangian iterated vanishing cycles
	$K_A,K_B\subset\sfibre{k}{P_1}$ over paths $\gamma_{A_1},\gamma_{B_1}$ respectively.
\end{itemize}
Abusing notation, we refer to the restrictions of $\rho_k,\Omega_k,\Theta_k$ to
$\sfibre{k}{P_1}$ by the same names where it is clear from context
which is meant. Similarly, we refer to the restrictions of $\rho_{k+l},\Omega_{k+l},\Theta_{k+l}$
to $\chi^{-1}(P_1\confbar{2l})$ by the same names.

The vanishing paths $\gamma_{A_1}$ and $P_1\gamma_{A_0}$ compose to give an iterated
vanishing path $\gamma_A$ in $\confbar{2(k+l)}$. We refer to this as the \emph{composite}.
One defines $\gamma_B$ similarly.
The symplectic Khovanov homology cochain complex $\cskh (P,A,B)$ may be chosen
to be the Floer cochain complex for Lagrangian submanifolds $L_A,L_B\subset\sfibre{k}{P}$
defined as iterated vanishing cycles along the vanishing paths $\gamma_A,\gamma_B$ respectively.

It will be necessary to explicitly specify the deformations performed to
control symplectic parallel transport over the paths $\gamma_{A_0}$ and $\gamma_{B_0}$.
Smoothly varying these choices will then give continuation maps between Floer cohomology groups.
We may suppose (by isotopy of vanishing paths) that both $\gamma_{A_0}$ and $\gamma_{B_0}$ are the composite
of smooth paths in $\conf{2l}$ with the vanishing path constructed as a piecewise smooth composite
of linear paths $X^{2(n-1)}\prod_{i=n}^l(X^2-\epsilon_i)$ to $X^{2n}\prod_{i=n+1}^l(X^2-\epsilon_i)$
for some collection of real numbers $0<\epsilon_1\ll\epsilon_2\ll\ldots\ll\epsilon_l\ll 1$.
The iterated vanishing cycle construction
(cf. the relative vanishing cycle construction, Lemma~32 of \cite{seidelsmith:khsymp})
is well-defined using honest (neither deformed, nor rescaled) symplectic parallel
transport along this composite of linear paths when each $\epsilon_i$ is chosen sufficiently small
in terms of $\epsilon_{i-1}$ and the K\"ahler metric defined near $X^l\in S_l$.

To continue the construction along the rest of $\gamma_{A_0}$ and $\gamma_{B_0}$
it is in general necessary to deform symplectic parallel
transport (see Remark~\ref{remark:pathspt}).
We now deform the pullback $\gamma_{A_0}^*\Theta_l$
on the fibration $\gamma_{A_0}^*S_l$.
This deformation, performed on fibres away from a small neighbourhood of $0\in[0,1]$, suffices to
give a well-defined iterated vanishing cycle
in $\sfibre{l}{P_0}$.
The same works for $\gamma_{B_0}$.

We can view $K_{A_1},K_{B_1}$ as an intermediate step in the iterated vanishing cycle
construction for $\gamma_A,\gamma_B$ respectively, since these paths split with first sections
$\gamma_{A_1},\gamma_{B_1}$. One can continue the construction with deformed one forms
$\Theta_{A_0},\Theta_{B_0}$
to control symplectic parallel transport over the remainder of the paths
(i.e.\ over $P_1\gamma_{A_0},P_1\gamma_{B_0}$).
We denote the resulting Lagrangians in the fibre $\sfibre{k+l}{P}$ by 
\[L_A(\Theta_{k+l},\Theta_{A_0}),L_B(\Theta_{k+l},\Theta_{B_0}).\]

In general, for any choice of $\Theta$ and necessary deformations of it, the Lagrangians
$L_A(\Theta,\Theta_{A_0})$, $L_B(\Theta,\Theta_{B_0})$ depend only on
a neighbourhood containing the Lagrangians at each stage of the construction.
Suppose for some reason that all stages of the construction, starting with $K_A$ or $K_B$,
are guaranteed to remain entirely within the neighbourhood of $X^l\sfibre{k}{P_1}$ on which $\Psi$ is defined,
then $\Theta$ and the deformations need only be defined in that neighbourhood of $X^l\sfibre{k}{P_1}$.
Of particular interest are those $\Theta$ which respect the
splitting $\sfibre{k}{P_1}\times S_l$ together with deformations which change only the $S_l$ part of $\Theta$.

To prove Theorem~\ref{theorem:splitskh} one adjusts the choices of K\"ahler forms
(this gives canonical isomorphisms on cohomology
and corresponding homotopy equivalences on the cochain complexes).
Then the proof proceeds in two stages. Firstly one deforms $\Theta_{k+l}$ and
$\Theta_{A_0},\Theta_{B_0}$ until the Lagrangian iterated vanishing cycles are contained
inside a holomorphically convex neighbourhood on which $\Psi$ is defined. This
induces a `continuation map' on the Floer cochain complex. Then one
performs a deformation of exact symplectic forms (and the various 1--forms),
defined only in the convex region, ending with the pushforward under $\Psi$ of
$\Theta_k +\Theta_l$ (with deformations only of the $\Theta_l$ part). This gives
a second continuation map. The composite with the first continuation map is the cochain
homotopy equivalence given in the theorem. Some care is needed to show
the manner in which it is `canonical'.

\begin{proof}[Proof of Theorem~\ref{theorem:splitskh}]

The slice $S_{k+l}$ can be identified with $\mathbb{C}^{4(k+l)-1}$ (see either definition),
with the origin corresponding to $X^{k+l}$ and such that $X^lS_k$ is a linear subspace.
We decompose $S_{k+l}\cong X^lS_k\oplus V$, where $V$ is a linear complement
to $X^lS_k$. Now we choose the exhausting plurisubharmonic function $\rho_k$ on $S_k$ and
any exhausting plurisubharmonic function $\rho_V$ on $V$.

Consider the family of plurisubharmonic functions $\rho_k+s\rho_V$ for $s>0$.
This restricts as $\rho_k$ to $X^lS_k$ so, taking $\rho_{k+l}=\rho_k+s\rho_V$ for some such $s$,
the Lagrangian submanifolds $K_{A_1},K_{B_1}\subset\sfibre{k}{P_1}$
are well-defined independently of $s$.
Furthermore, we can choose some $R$ (independent of $s$) such that $X^lK_{A_1},X^lK_{B_1}$
lie strictly inside the $\rho_{k+l}=R$ level-set.

Choosing $s$ large enough, we can ensure that $\rho_{k+l}^{-1}[0,4R]$
is confined to any particular open neighbourhood of $S_k$ in $S_{k+l}$.
Now consider the local submanifolds $\Psi(\{y\}\times S_l)$ for $y$ in $\sfibre{k}{P_1}$.
These cannot be
tangent at $X^ly$ to $X^lS^k$ (they are not tangent to $X^l\sfibre{k}{P_1}$ and other tangencies
to $X^lS_k$ would imply that their projection (Proposition~\ref{prop:semisplit})
to $\confbar{2l}\times\mathfrak{\xi}$ varies in the $\mathfrak{\xi}$ direction).
Hence, by compactness of $\rho_{k+l}^{-1}[0,4R]$, we can also ensure that
$\rho_{k+l}^{-1}[0,4R]\cap\chi^{-1}(P_1\confbar{2l})$ is contained within
an arbitrarily small
neighbourhood of $X^l\sfibre{k}{P_1}$. Let $s$ be any number large enough such that
this small neighbourhood lies within the range of $\Psi$.

We now choose a candidate $\tilde\rho_l$ for the function $\rho_l$ on $S_l$ which is
$\vecnorm{\mathbf{z}}^2$ w.r.t. the natural linear coordinates one has from
entries of the matrices
$A_1\in \mathfrak{sl}_n(\mathbb{C})$ and $A_2,\ldots,A_n\in\mathfrak{gl}_n(\mathbb{C})$
defining $S_l$.
This gives an exhausting plurisubharmonic function
$\tilde\rho_{split}:=\rho_k+\tilde\rho_l$ on $\sfibre{k}{P_1}\times S_l$.
Scaling up $\tilde\rho_l$ by a large enough constant factor,
we can ensure that the $R$ level-set of $\tilde\rho_{split}$
in $\sfibre{k}{P_1}\times S_l$ maps by $\Psi$ to within the $2R$ level-set
of $\rho_{k+l}$.
This level-set necessarily contains $K_{A_1}\times \{X^l\}$ and $K_{B_1}\times \{X^l\}$.

The enclosures used in the localisation argument will be $(\rho_{k+l},4R)$ and
$(\tilde\rho_{split},R)$. In the appropriate setting they will be shown to be equivalent.

Now we consider in detail the process by which the Lagrangian iterated vanishing cycles
are defined, starting from $K_{A_1},K_{B_1}$. By comparison to a region of
$\sfibre{k}{P_1}\times S_l$ and choice of $\rho_{split}=\rho_l+\rho_k$, we show
how to deform the exact symplectic structure on the range of $\Psi$ and how to choose
the deformations controlling symplectic parallel transport, such that the iterated
vanishing cycles are contained in the image under $\Psi$ of the
$R$ level-set of $\tilde\rho_{split}$.

We begin by considering the paths $\gamma_{A_0},\gamma_{B_0}$ in $\confbar{2l}$ and
the iterated vanishing cycle construction in $S_l$ over them. We start with the exact K\"ahler form
\[\tilde\Omega_l:=d\tilde\Theta_l:=d(-d\tilde\rho_l\circ i)\] on $S_l$, and with
deformations $\tilde\Theta_{l,A_0},\tilde\Theta_{l,B_0}$ of 1--forms over the vanishing paths necessary for
the iterated vanishing cycle construction.
In general, the Lagrangian iterated vanishing cycles in $S_l$ one
defines in this way are not confined to the region where $\Psi$ is defined. We \emph{fix} this by
rescaling, using a holomorphic $\mathbb{C}^*$--action compatible with the fibration $S_l$.
For $\lambda\in\mathbb{C}^*$, this acts on $S^l$ as:
\[\func{\lambda}{X^nI-\sum_{i=1}^{n} X^{n-i}A_i}{X^nI-\sum_{i=1}^{n} X^{n-i}\lambda^iA_i}\]
and on $\confbar{2l}$ by multiplying all roots by $\lambda$. We will only be interested in
sufficiently small positive real values of $\lambda$.

We define a new exhausting plurisubharmonic function $\tilde\rho^\lambda_l$
as the pushforward by the action of $\lambda$. Similarly, we get $\tilde\Theta_l^\lambda$ and
deformations $\tilde\Theta_{l,\lambda(A_0)},\tilde\Theta_{l,\lambda(B_0)}$.
These deformations are defined over new iterated vanishing paths
\[\gamma_{\lambda(A_0)}:=\lambda(\gamma_{A_0}),\qquad\qquad\gamma_{\lambda(B_0)}:=\lambda(\gamma_{B_0}),\]
representing the scaled crossingless matchings $\lambda(A_0),\lambda(B_0)$ respectively.
The Lagrangian iterated vanishing cycles also respect this pullback, namely:

\[
L_{\lambda(A_0)}(\tilde\Theta_l^\lambda,\tilde\Theta_{l,\lambda(A_0)}) =
	\lambda(L_{A_0}(\tilde\Theta_l,\tilde\Theta_{l,A_0}))
\]
\[
L_{\lambda(B_0)}(\tilde\Theta_l^\lambda,\tilde\Theta_{l,\lambda(B_0)}) =
	\lambda(L_{B_0}(\tilde\Theta_l,\tilde\Theta_{l,B_0}))
\]

Choose $\lambda$ small enough that the new iterated vanishing cycles are contained
within so small a level set that the products
$K_{A_1}\times L_{\lambda(A_0)}(\tilde\Theta_l^\lambda,\tilde\Theta_{l,\lambda(A_0)})$
and
$K_{B_1}\times L_{\lambda(B_0)}(\tilde\Theta_l^\lambda,\tilde\Theta_{l,\lambda(B_0)})$
lie within the enclosure $(\tilde\rho_{split},R)$ in $\sfibre{k}{P_1}\times S_l$.
Now we define $\rho_l:=\rho_l^\lambda$ and $\rho_{split}:=\rho_l+\rho_k$.
Importantly, these are plurisubharmonic for the same
complex structure as $\tilde\rho_l,\tilde\rho_{split}$ respectively.
This also specifies $\Theta_l$ and $\Theta_{split}=\Theta_l+\Theta_k$. Similarly,
$\Theta_{split,A_0},\Theta_{split,B_0}$ are given by adding $\Theta_k$ to the
deformations $\Theta_{l,A_0},\Theta_{l,B_0}$. These define relative vanishing cycles to
$K_{A_1},K_{B_1}$ respectively, which split the products written above.
By careful isotopy of $\Theta_{k+l}$ to make it equal to $\Psi_*\Theta_{split}$ within
the image under $\Psi$ of the enclosure $(\tilde\rho_{split},R)$, we shall control
the position of the iterated Lagrangian vanishing cycles in $S_{k+l}$ in the required manner.

It is helpful to define one more exact 2--form $\Omega_{flat}=d\Theta_{flat}$
before defining an isotopy of exact symplectic forms relating $\Omega_{k+l}$ and $\Psi_*\Omega_{split}$.
For a suitable choice of smooth function
\footnote{
It suffices to choose $h$ such that $h(x)=0$ for $x\leq 2R$ and $h^\prime(x)>0,h^{\prime\prime}(x)\geq 0$ for $x>2R$.
Then for any vector $V\neq 0$ we have:
\(-d\left(d(h\circ\rho)\circ i\right)(V,iV)=\frac{h^{\prime\prime}(\rho)}{2}\left(d\rho(V)^2
  +d\rho(iV)^2\right)+h^\prime(\rho)\left(-d(d\rho\circ i)(V,iV)\right)>0\)
}
$\func{h}{\mathbb{R}}{\mathbb{R}}$, the function
$\rho_{flat}:=h(\rho_{k+l})$ vanishes where $\rho_{k+l}\leq 2R$, and is plurisubharmonic
elsewhere. We define $\Theta_{flat}=-d\rho_{flat}\circ i$.

Let $\func{g}{\mathbb{R}}{[0,1]}$ be any smooth function such that $g(t)=1$ for $t\leq 2R$ and
$g(t)=0$ for $t\geq 3R$. Then $g\Psi_*\Theta_{split}$ extends to a 1--form on all of
$\chi^{-1}(P_1\confbar{2l})$. Consider the compact family
\[d(r\Theta_{k+l}+s\Theta_{flat}+t\Psi_*\Theta_{split}+u\epsilon g\Psi_*\Theta_{split})\]
\noindent of 2--forms where $r,s,t,u\in[0,1]$ and $\max\{r,s,t\}\geq \frac{1}{2}$.
The space of all K\"ahler forms is convex, so these are K\"ahler for $\rho_{k+l}\leq 2R$
(and for trivial reasons for $\rho_{k+l}\geq 3R$). Non-degeneracy of
2--forms is an open condition, so there exists some $\epsilon>0$ small enough
such that the 2--forms are all symplectic.

We are now ready to define the isotopy of exact symplectic forms from
$\Omega_{k+l}$ to $\Psi_*\Omega_{split}$.
Namely, we take the linear isotopies between the following 1--forms (in the given order):

\begin{enumerate}[(1)]
 \item $\Theta_{k+l},$
 \item $\Theta_{k+l}+\epsilon g\Psi_*\Theta_{split},$
 \item $\Theta_{flat}+\epsilon g\Psi_*\Theta_{split},$
 \item $\Psi_*\Theta_{split}+\epsilon g\Psi_*\Theta_{split},$
 \item $\Psi_*\Theta_{split}.$
\end{enumerate}

From stage 1 to stage 3, this is an isotopy of exact symplectic forms defined on
all of $\chi^{-1}(P_1\confbar{2l})$ which are equal to $\Omega_{k+l}$ where $\rho_{k+l}\geq 3R$.
Then it continues to stage 5 as an isotopy of forms defined only where $\rho_{k+l}< 4R$.

We choose, along with the 1-forms from stages 1 to 3, smoothly varying families
of deformations over the paths
$P_1\lambda(\gamma_{A_0}),P_1\lambda(\gamma_{B_0})$ to control symplectic parallel transport.
These deformations are chosen to be supported where $\rho_{k+l}>4R$.
From stage 3 onwards we can use the deformations $\Theta_{split,A_0},\Theta_{split,B_0}$
of whichever constant multiple of $\epsilon g\Psi_*\Theta_{split}$
we have where $\rho_{k+l}<2R$.

The transition to these deformations at stage 3 is no problem.
Namely, the deformations given by $\Theta_{split,A_0},\Theta_{split,B_0}$ force
all stages of the construction
of the relative vanishing cycles to $K_{A_1},K_{B_1}$ to be contained in $\rho_{k+l}^{-1}[0,2R]$.
Hence, we can remove the previous deformations (by isotopy) without changing the Lagrangians. In fact, the
relative vanishing cycles are now the images under $\Psi$ of the product Lagrangians
$K_{A_1}\times L_{\lambda(A_0)}(\tilde\Theta_l^\lambda,\tilde\Theta_{l,\lambda(A_0)})$
and
$K_{B_1}\times L_{\lambda(B_0)}(\tilde\Theta_l^\lambda,\tilde\Theta_{l,\lambda(B_0)})$
described earlier.

$\Theta_{flat}+\epsilon g\Psi_*\Theta_{split}$ is K\"ahler for the complex structure we
have on $\chi^{-1}(P_1\confbar{2l})$ outside of the $\rho_{k+l}=3R$ level-set, so none
of the holomorphic strips defining the differential in the Floer cochain complex leaves
the $\rho_{k+l}\leq 4R$ locus. For the remaining stages, the forms remain K\"ahler
where $3R\leq\rho{k+l}\leq 4R$, so we can get away with calculating Floer cohomology
only within this neighbourhood. Also from stage 3 onwards we always have a constant
multiple of $\Psi_*\Theta_{split}$ where $\rho_{k+l}<2R$, so we may continue
to use the same deformations to control symplectic parallel transport. This means that
the Lagrangian relative vanishing cycles do not change.

Finally, once we reach the end of this isotopy of exact symplectic structures, we
can compare the Floer cochain complex to
\[
CF^*(
	\Psi^{-1}(K_{A_1})\times 
	L_{\lambda(A_0)}(\tilde\Theta_l^\lambda,\tilde\Theta_{l,\lambda(A_0)}),
	\Psi^{-1}(K_{B_1})\times 
	L_{\lambda(B_0)}(\tilde\Theta_l^\lambda,\tilde\Theta_{l,\lambda(B_0)}))
\]

\noindent within $\sfibre{k}{P_1}\times S_l$. We find that nothing has changed,
since all holomorphic strips used to define the differential here lie within
the $R$ level-set of $\rho_{split}$ and this is contained in the pre-image under $\Psi$
of the $\rho_{k+l}\leq 4R$ region of $\chi^{-1}(P_1\confbar{2l})$. All the data
defining this splits, so the cochain complex splits as a tensor product:
\[
CF^*(\Psi^{-1}(K_{A_1}),\Psi^{-1}(K_{B_1}))
\otimes
CF^*(L_{\lambda(A_0)}(\tilde\Theta_l^\lambda,\tilde\Theta_{l,\lambda(A_0)}), 
	L_{\lambda(B_0)}(\tilde\Theta_l^\lambda,\tilde\Theta_{l,\lambda(B_0)}))
\]

Using the continuation maps associated to the isotopies of exact symplectic structures
and simultaneous isotopy of the Lagrangians, we find this is isomorphic to the following:
\[
\cskh (P_1,A_1,B_1)\otimes \cskh (\lambda(P_0),\lambda(A_0),\lambda(B_0))
\]
which, by the continuation map induced by varying the rescaling parameter from $1$ to
$\lambda$ as in Lemma~\ref{lemma:shrinklink}, is isomorphic to
\[
\cskh (P_1,A_1,B_1)\otimes \cskh (P_0,A_0,B_0)
\]

It remains only to describe how these isomorphisms relate to the canonical isomorphisms
on symplectic Khovanov homology (and to give the cochain level version of this).
The cochain complex $CF^*(\Psi^{-1}(K_{A_1}),\Psi^{-1}(K_{B_1}))$
is canonically isomorphic to $\cskh (P_1,A_1,B_1)$.
However, for $(P_0,A_0,B_0)$ the corresponding statement is not
quite true. The $S_l$ factor in $\sfibre{k}{P_1}\times S_l$ is only identified
up to some automorphism of $S_l$.
Hence,
\[CF^*(L_{\lambda(A_0)}(\tilde\Theta_l^\lambda,\tilde\Theta_{l,\lambda(A_0)}), 
L_{\lambda(B_0)}(\tilde\Theta_l^\lambda,\tilde\Theta_{l,\lambda(B_0)}))\] is
identified with $\cskh (P_0,A_0,B_0)$ up to an automorphism of cochain complexes.

At the other end of the construction we have some choice of
$\cskh (P,A,B)$. Any two choices are related by
a chain homotopy equivalence inducing the canonical isomorphisms
on cohomology. The isotopy of exact symplectic structures given above
induces a well-defined continuation map
\[\cskh (P,A,B)\to \cskh (P_1,A_1,B_1)\otimes \cskh (P_0,A_0,B_0)\]
since the space of choices involved was connected.
\end{proof}

\begin{lemma}
\label{lemma:splitmap}
Let $D$ be a disc centred at $0\in\mathbb{C}$.
Suppose we have a bridge diagram $(P,A,B)$ which splits as in
Theorem~\ref{theorem:splitskh} as the union of bridge diagrams
$(P_1,A_1,B_1)$ supported away from $0$ and $(P_0,A_0,B_0)$ supported inside $D$.

Let $u$ be an admissible map from a disc $B$ into $\confbar{2(k+l)}$
with image contained in
\[P_1\operatorname{Conf^0_{2l}}(D)\subset\confbar{2(k+l)}\]
Suppose also that $u$ has a single input marked point
mapping to $P_1P_0$ and output marked point mapping to some $P_1P_0^\prime$.
We write $(P^\prime,A^\prime,B^\prime)$ for the bridge diagram given by the
output end of the saddle cobordisms induced by $u$. It is the union of
$(P_1,A_1,B_1)$ with some bridge diagram $(P_0^\prime,A_0\prime,B_0\prime)$
supported inside $D$.

Then the relative invariant factors up to homotopy through the
splittings (as K\"unneth products) of $\skh $ at each end given by Theorem~\ref{theorem:splitskh}

\[
\bfig

	\node a(0,800)[\skh (P,A,B)]
	\node b(2000,800)[\skh (P^\prime,A^\prime,B^\prime)]
	\node c(0,0)[\begin{array}{c}
		\skh (P_1,A_1,B_1)\\
		\otimes \\
		\skh (P_0,A_0,B_0)
	\end{array}]
	\node d(2000,0)[\begin{array}{c}
		\skh (P_1,A_1,B_1)\\
		\otimes \\
		\skh (P_0^\prime,A_0^\prime,B_0^\prime)
	\end{array}]

	\arrow[a`b;f_u]
	\arrow[a`c;]
	\arrow[b`d;]
	\arrow[c`d;f_1\otimes f_0]
 \efig
\]

Furthermore, the bottom map splits w.r.t. this product as $f_1\otimes f_0$
where $f_1$ is the identity on $\skh (P_1,A_1,B_1)$ and $f_0$ the map on
$\skh (P_0,A_0,B_0)$ induced by considering
$u$ as a map to $\operatorname{Conf^0_{2l}}(D)$.
\end{lemma}

\begin{remark}
We use the same splitting to define, up to an overall sign ambiguity, chain maps induced by
creation and annihilation cobordisms (see Section~\ref{section:crean}).
Therefore, up to this sign ambiguity, the splitting of maps described above
in fact holds for general cobordisms supported near $0\in\mathbb{C}$.
\end{remark}

\begin{proof}
This is an adaptation of the proof of Theorem~\ref{theorem:splitskh}. One begins with a
tree construction (as in Section~\ref{section:flattening}) representing $u$.

Throughout the argument one must then consider the deformations of one forms over
iterated vanishing paths leading to $P=P_1P_0$ and also all discs in the tree construction.
The immediate problem is that the pullback of $d(\Theta_{flat}+\epsilon g\Psi_*\Theta_{split})$ is not
an allowed choice of 2--form over any of these discs containing a singular value, since
it is not in general K\"ahler near the singular locus. The same problem occurs for other of the
forms used in the argument.

Singular values are isolated in the interior of the discs, and we only need
control of symplectic parallel transport around the boundaries, so we can
perform the following correction.

First, choose the deformation controlling symplectic parallel transport to be supported
within distance $\epsilon/2$ of the boundary of the disc. Then between distance $\epsilon/2$
and $\epsilon$ of the boundary one linearly interpolates (w.r.t. the radial coordinate
on the disc) between the offending 1--form $\Theta_{flat}+\epsilon g\Psi_*\Theta_{split}$
and the 1--form $\Theta_{k+l}$.

The same trick can be performed with $\Theta_{k+l}$ replaced by $\Theta_{split}$ or any
convex combination of the two. Hence the entire localisation argument of
Theorem~\ref{theorem:splitskh} works for tree constructions as well.
\end{proof}

Suppose $P=P(X)\in\confbar{2n}$. We can translate all the roots of $P$ by $\mu\in\mathbb{C}$
by a change of variables replacing $P(X)$ by $P(X-\mu)$. Using this trick, we can map
$\operatorname{Conf}^0_{2k}(\mathbb{C}\setminus\{(1+\frac{k}{l})\mu\})$ into
$\confbar{2(k+l)}$ by
\[P(X)\mapsto (X-\mu)^{2l}P(X+\frac{l}{k}\mu)\]
such that the image is precisely the subset of $\confbar{2(k+l)}$ of polynomials
with root $\mu$ of multiplicity $2l$.

There is a map of fibrations $S_k\to S_{k+l}$ over this map which has a similar description.
Suppose $y(X)\in S_{k}$, then $y(X)$ is a polynomial of degree $k$ with matrix coefficients.
We define the map of fibrations as
\[y(X)\mapsto (X-\mu)^ly(X+\frac{k}{l}\mu)\]

This is an affine linear map (in terms of the entries of all the matrix coefficients of $y(X)$).
Restricting to fibres of $S_k$ over $\operatorname{Conf}^0_{2k}(\mathbb{C}\setminus\{(1+\frac{k}{l})\mu\})$,
this map has image in the singular locus of the fibres of $S_{k+l}$ which it hits.
In fact all the calculations of this section generalise straightforwardly to this setting.

\begin{remark}
\label{remark:centreshift}
This has the consequence that Theorem~\ref{theorem:splitskh} (and Lemma~\ref{lemma:splitmap})
hold when applied to unlinked components of a bridge diagram (or cobordism of bridge diagrams)
supported on a small disc not necessarily centred at $0\in\mathbb{C}$. The only extra
complication is that, in decomposing a bridge diagram into two, one has to translate
the resulting bridge diagrams in $\mathbb{C}$ to ensure they are of the form $(P,A,B)$ with $P\in\conf{2n}$.
\end{remark}

\subsection{Creation/annihilation cobordisms}
\label{section:crean}

Suppose we have two bridge diagrams $D$ and $D^\prime$ which are everywhere
identical except that $D^\prime$ contains, in some small neighbourhood,
an extra unlinked component formed from a single alpha and beta curve.
By Theorem~\ref{theorem:splitskh} there is an isomorphism
\[\skh (D^\prime)\to \skh (D)\otimes H^*(S^2)\]
which is canonical on the first factor.

We now define the creation and annihilation maps explicitly in terms
of this splitting. Namely the creation map, corresponding to the elementary
cobordism $D$ to $D^\prime$, shall be:

\begin{eqnarray*}
\skh (D)&\to&\skh (D)\otimes H^*(S^2) \\
r&\longmapsto&r\otimes 1 \\
\end{eqnarray*}

Here we view $H^*(S^2)$ as the ring $\mathbb{Z}[X]/(X^2)$.
The annihilation map corresponding to the elementary
cobordism $D^\prime$ to $D$, shall be:

\begin{eqnarray*}
\skh (D)\otimes H^*(S^2)&\to&\skh (D) \\
r\otimes 1~~ &\longmapsto& 0 \\
r\otimes X &\longmapsto& r \\
\end{eqnarray*}

In both cases, composing with a (grading preserving) automorphism of the
$H^*(S^2)$ factor can only change the map by a sign. Hence they are well
defined up to sign.

The motivation behind these two definitions is twofold. Firstly, they are
simply defined to copy the corresponding maps between Khovanov homology groups.
Secondly it is necessary in order to make the stabilisation and destabilisation
maps of Section~\ref{section:stab} independent of the vertex of a bridge diagram
at which one performs stabilisation.

\subsection{Stabilisation and destabilisation maps}
\label{section:stab}

As shown earlier, stabilisation of a bridge diagram yields a new bridge
diagram and an isomorphism between the old and new symplectic Khovanov homologies.
In this section,
we shall show that this isomorphism can be realised as the composite
of a creation map and a single saddle cobordism.
A similar construction will also be made for destabilisation.

\begin{figure}[h]
\begin{center}
\scalebox{0.7}{\input{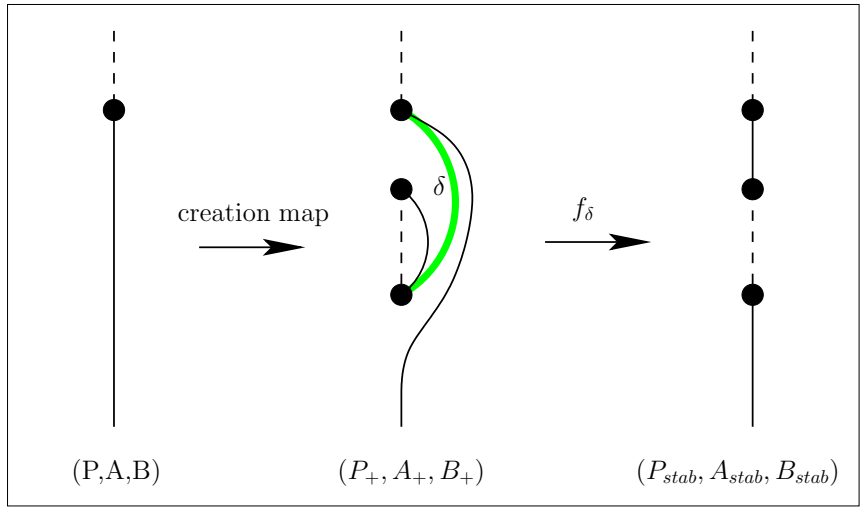}}
\end{center}
\caption{An illustration of the stabilisation map locally near a vertex
of the bridge diagram.}
\label{fig:stabmap}
\end{figure}

\begin{definition}
Let $(P,A,B)$ be any bridge diagram. Choose a vertex $v$ and define a new bridge
diagram $(P_+,A_+,B_+)$ by adding an unlinked unknot component consisting
of a single $\alpha$--~and $\beta$--curve both supported near $v$.
Let $\delta$ be any curve supported near $v$ and joining $v$ to a
vertex of the unknot component. Let $(P_{stab},A_{stab},B_{stab})$ be the bridge diagram obtained
by performing the saddle cobordism which $\delta$ specifies. Figure~\ref{fig:stabmap}
illustrates these constructions locally near $v$.

We define the \emph{stabilisation map} to be the composite
of two maps: the creation map
\[\skh (P,A,B)\to \skh (P_+,A_+,B_+)\]
and the map induced by $\delta$
\[\func{f_\delta}{\skh (P_+,A_+,B_+)}{\skh (P_{stab},A_{stab},B_{stab})}\]

We define the \emph{destabilisation map} as a similar composite. First one performs an elementary cobordism
$\skh (P_{stab},A_{stab},B_{stab})\to\skh (P_+,A_+,B_+)$, then the annihilation map.
\end{definition}

We begin by describing the iterated vanishing cycle construction in coordinates
which fit with the the model neighbourhood given
by Lemma~\ref{lemma:a2}.

Let $(XP,A,B)$ be a bridge diagram with $2m-2$ vertices, precisely one of which lies at $0\in\mathbb{C}$.
Let $\gamma_A,\gamma_B$ be iterated vanishing paths in $\confbar{2m-2}$
corresponding to the crossingless matchings $A,B$ respectively. With the appropriate
choices made, these give iterated vanishing cycles $K_A,K_B\subset\sfibre{m-1}{XP(X)}$.
These data are the start point of the constructions of this section.

The base $\mathbb{C}^2$ of the model neighbourhood corresponds locally near
$(0,0)$ to a subset of $\confbar{2m}$ by the map
\[(d,z)\mapsto (X^3-Xd+z)P(X)\]
In particular, the critical values are the points of the form $(3x^2,2x^3)$
which correspond to configurations $(X-x)^2(X+2x)P(X)$ with a root of multiplicity 2
at $x$.

Pick some small $x$ and let $d=3x^2$. We will discuss how small this $x$ has to be later.
We now extend $\gamma_A,\gamma_B$ by composing with the path
$t\mapsto (X+2tx-\frac{tx}{m-1})P(X-\frac{tx}{m-1})$ for $t$ ranging from 0 to 1.
This is the same as composing
$X^2\gamma_A,X^2\gamma_B$ with the path $(X-tx)^2(X+2tx)P(X)$, staying in the singular locus.
Denote the composite paths
by $\gamma_{A,x},\gamma_{B,x}$. To finish the iterated vanishing cycle construction we
choose vanishing paths $\psi_A,\psi_B$ in $\mathbb{C}$ for the Lefschetz fibration
\begin{equation*}
\xymatrix{
\mathbb{C}^3 \ar[d]^{\pi} & (a,b,c) \ar@{|->}[d]\\
\mathbb{C} & a^3-ad+bc}
\end{equation*}
We then compose the paths $(X^3-Xd+\psi_A)P(X),(X^3-Xd+\psi_B)P(X)$ with
$\gamma_{A,x},\gamma_{B,x}$ (considered as paths in $\confbar{2m}$) respectively
to give iterated vanishing paths $\gamma_A^\prime,\gamma_B^\prime$ in $\confbar{2m}$
ending at some regular value $P^\prime:=(X^3-Xd+z)P(X)$ of $\chi$.

Suppose we conclude that $x$ needs to be smaller, then $\psi_A,\psi_B$
can be scaled down continuously by $(d,z)\mapsto (\lambda^2 d,\lambda^3 z)$ for
small real $\lambda$. This replaces $x$ by $\lambda x$.

Let $S$ be a disc centred at $0\in\mathbb{C}$ which no roots of $XP$, except 0,
and only one section of a curve in each of $A,B$, each ending at 0.
The iterated vanishing paths $\gamma_A^\prime,\gamma_B^\prime$ come from
a bridge diagram $(P^\prime,A^\prime,B^\prime)$ with the following properties:
\begin{itemize}
\item $P^\prime=(X^3-Xd+z)P(X)$ has three roots in $S$.
\item $(P^\prime,A^\prime,B^\prime)$ is identical to $(XP,A,B)$ outside of $S$.
\item $A^\prime,B^\prime$ are each intersect $S$ in the end section of one curve and the whole of another.
\end{itemize}

In fact, by considering the effect of moving $\psi_A,\psi_B$ around the two
singularities $\pm2x^3$ of the Lefschetz fibration, one sees that all
$(P^\prime,A^\prime,B^\prime)$ with the above properties (for a fixed $(d,z)$)
are realised in this manner. Also, any two versions of the construction for
the same $(P^\prime,A^\prime,B^\prime)$, up to isotopy
supported on $S$, are essentially the same. i.e.\ the pairs
$\psi_A,\psi_B$ are isotopic.

Corresponding to a construction as above, we define $L_A,L_B$ to be the
vanishing cycles (in a regular fibre of $\mathbb{C}^3\rightarrow\mathbb{C}$),
for $\psi_A,\psi_B$.

\begin{lemma}
\label{lemma:vcycleconstraint}
Let $\psi$ be a vanishing path for the Lefschetz fibration $\mathbb{C}^3\longmapsto\mathbb{C}$,
as above, with some fixed value of $x\in\mathbb{C}^*$.

Then, for sufficiently small $\lambda>0$, the vanishing cycle construction for the vanishing path
$\lambda^3\psi$ (with $x$ replaced by $\lambda x$) is well-defined. Furthermore, it may be confined
to an arbitrarily small neighbourhood of $0\in\mathbb{C}^3$.
\end{lemma}

\begin{proof}
For fixed $d\neq 0$, at a point $(a,b,c)\neq(\pm x,0,0)$, the symplectic parallel
transport vector over $\dot{z})$ is:

\[\frac{\dot{z}}{9\norm{a^2-x^2}^2+\norm{b}^2+\norm{c}^2}\left(\begin{array}{c}
3\overline{(a^2-x^2)}\\
\bar{c}\\
\bar{b}
\end{array}\right)\]

\noindent so the locus $\{\norm{b}^2=\norm{c}^2\}$ is preserved by symplectic parallel
transport. The closure of this locus contains the singular points of the fibration,
so vanishing cycles must also lie in this locus. Here, the flow on the $a$--coordinate
is given by:

\begin{equation*}
\frac{\partial}{\partial t} a =
	\frac{3(\overline{a^2-x^2})}{9\norm{a^2-x^2}^2+2\norm{z+3ax^2-a^3}}
	\dot{z}(t)
\end{equation*}

In particular, one has
\[
\norm{\dd{}{t}a}\leq\frac{\norm{\dot{z}}}{3\norm{a^2-x^2}}
\]
\noindent and, as one scales down $x\mapsto\lambda x$, the path $\psi$
is replaced by $\lambda^3\psi$, so this bound is sufficient to contain the vanishing
cycles near $a=x$. This also contains $b,c$ near $0$, since $\norm{b}^2=\norm{c}^2=\norm{bc}=\norm{z-a^3+ad}$,
which vanishes at $a=x$.
\end{proof}

The same argument can be applied to constructing an exact Lagrangian boundary condition on
the restriction of $\mathbb{C}^3$ to a disk $D$ with two boundary marked points.
Given $\psi_A,\psi_B$ ending at the input boundary marked point, we get $L_A, L_B$ in the fibre over
that point and, extending by symplectic parallel transport around, $D$, we get
an exact Lagrangian boundary condition $Q$. Scaling down the whole construction, as above, suffices
to contain $Q$ near $0\in\mathbb{C}^3$.

In this, section we are interested in two particular constructions of $\psi_A,\psi_B$ and $D$.
For the stabilisation map ($f_\delta$ in Figure~\ref{fig:stabmap}), we take $\psi_A=\psi_B$
to be vanishing paths from a regular point near $z=-2x^3$ to the singular point $z=2x^3$, and
$D$ to embed (holomorphically) onto a small neighbourhood of the singular point $z=-2x^3$.
For the destabilisation map, we move one of $\psi_A,\psi_B$ around $z=-2x^3$ first.

\begin{lemma}
\label{lemma:a2splitmap}
Let $K_A,K_B$ be iterated vanishing cycles in a regular fibre $\sfibre{m-1}{XP}$
as above, representing a bridge diagram $(XP,A,B)$. Choose $\psi_A,\psi_B,D$ describing the stabilisation or destabilisation move locally near 0.

Then, the symplectic Khovanov homology chain complexes split canonically as the tensor
product of $\cskh (XP,A,B)=HF^*(K_A,K_B)$ with $CF^*(L_A, L_B)$ of the Lagrangian vanishing cycles
in $\mathbb{C}^3$ constructed over the same paths. $CF^*(L_A, L_B)$ is either $\mathbb{Z}$
or $\frac{\mathbb{Z}[X]}{(X^2)}\simeq\mathbb{Z}^2$, with zero differential.

The relative invariant splits in the same manner, as
the identity map on $\cskh (XP,A,B)$ and the following maps
on the remaining factor:

\[
\begin{array}{ccccccc}
	&\text{destabilisation}&&&&\text{stabilisation}&\\
	\mathbb{Z}&\to&\mathbb{Z}[X]/(X^2)&\qquad\qquad\qquad&\mathbb{Z}[X]/(X^2)&\to&\mathbb{Z}\\
	1&\longmapsto&X&&X&\longmapsto&0\\
	&&&&1&\longmapsto&1\\
\end{array}
\]

Furthermore, when $CF^*(L_A, L_B)=\frac{\mathbb{Z}[X]}{(X^2)}$
(a valid choice of the cochain complex for the simplest unknot diagram), the splitting
agrees with that given in Theorem~\ref{theorem:splitskh}.
\end{lemma}

\begin{proof}
The proof uses a localisation argument similar to that in the proof of Theorem~\ref{theorem:splitskh}, but
somewhat more involved.

We shall be using the model neighbourhood of Lemma~\ref{lemma:a2}, so
here it is again as a reminder.
\[
\bfig
 \Square[
	\chi^{-1}(\mathbb{D}^2)`
	\mathbb{C}^4\times_{\mathbb{C}^*}(F\setminus 0)`
	\mathbb{D}^2`\mathbb{C}^2;
	\text{local isomorphism $\Psi$ near } X^2\sfibre{m-1}{XP} ~\text{in}~ \chi^{-1}(0,0)`
	\chi`
	f`
	]
 \efig
\]

The term $F$, used on the right-hand side, is a holomorphic line bundle over $\sfibre{m-1}{XP}$,
which is a subbundle of the trivial $\mathbb{C}^2$--bundle. $\mathbb{C}^*$ acts on $\mathbb{C}^4$
by $\lambda:(a,b,c,d)\mapsto (a,\zeta b,\zeta^{-1} c,d)$, so $\mathbb{C}^4\times_{\mathbb{C}^*}(F\setminus 0)$
is isomorphic as a holomorphic line bundle to $\mathbb{C}\oplus F\oplus F^{-1} \oplus\mathbb{C}$.
Because $\sfibre{m-1}{XP}$ is Stein, $F\oplus F^{-1}$ is isomorphic to the trivial $\mathbb{C}^2$ bundle.

On $\sfibre{m-1}{XP}$, we have the exhausting, plurisubharmonic function on $S_{m-1}$, such that
$K_A, K_B\subset\rho^{-1}[0,R)$ for some $R\in\mathbb{R}$.

We construct an exact K\"ahler form $\tilde{\Omega}=d\tilde{\Theta}$
on $\mathbb{C}^4\times_{\mathbb{C}^*}(F\setminus 0)$ with the following properties:
\begin{itemize}
 \item $\tilde{\Theta}$ is $S^1$--equivariant,
 \item on an open neighbourhood $U$ of $0\times_{\mathbb{C}^*}(F\setminus 0)|_{\rho_{m-1}^{-1}[0,R)}$,
symplectic parallel
transport fixes the projection to $\sfibre{m-1}{XP}$ and is given, up to an $S^1$--ambiguity by
the symplectic parallel transport in $\mathbb{C}^4\longmapsto\mathbb{C}^2$ with the standard symplectic form,
\end{itemize}

There is a natural choice of exact $(1,1)$--form on $\mathbb{C}^4\times_{\mathbb{C}^*}(F\setminus 0)$
which we shall call $\Omega_{ass}=d\Theta_{ass}$ (see Section~\ref{section:sympassocbundle} and also the description
of symplectic associated bundles in \cite[Section~4.3]{seidelsmith:khsymp}). This form
is non-degenerate, so K\"ahler in a neighbourhood $0\times_{\mathbb{C}^*}(F\setminus 0)$
and on that neighbourhood satisfies the condition on symplectic parallel transport.
We shall make it K\"ahler everywhere by adding a \emph{sufficiently positive} form.

Identifying $\mathbb{C}^4\times_{\mathbb{C}^*}(F\setminus 0)$ biholomorphically
with the trivial $\mathbb{C}^4$--bundle over $\sfibre{m-1}{XP}$, we define the
exhausting plurisubharmonic function $\rho_{split}:=\rho_m+\lambda\vecnorm{z}^2$.
Averaging $\rho_{split}$ over the $S^1$--action, gives another 
exhausting plurisubharmonic function $\rho_{split}^{S^1}$.
For large enough $\lambda\in\mathbb{R}$, the neighbourhood $(\rho^{S^1}_{split})^{-1}[0,2R)$
is small enough in the $\mathbb{C}^4$--directions that $\Omega_{ass}$ is non-degenerate
there.

Let $\func{h}{[0,\infty)}{\mathbb{R}}$ be a smooth function, identically zero on $[0,2R)$
and with $h^\prime,h^{\prime\prime}$ strictly positive on $(R,\infty)$.
By choosing $h^\prime(t)$ large enough for each $t\in[2R,\infty]$, we ensure that:
\[-d(d(h\circ\rho_{split}^{S^1})\circ i)(V,JV)>\norm{\Omega_{ass}(V,iV)}\]
Therefore $\tilde{\Theta}:=-d(h\circ\rho_{split}^{S^1})\circ i+\Omega_{ass}$ is
as required.

By scaling down any choice of $x$ to $\lambda x$, symplectic parallel transport
applied to $K_A,K_B$, over the path to $(3x^2,2x^3)$ (described earlier)
and then along $\psi_A,\psi_B$ and a section of $\partial D$, respectively, can be contained within
the neighbourhood $U$ (see Lemma~\ref{lemma:vcycleconstraint}).
Furthermore, no matter how small we choose $U$,
we may still choose $x$ small enough to contain the symplectic parallel transport
in this way.

Having chosen an exact K\"ahler form on $\mathbb{C}^4\times_{\mathbb{C}^*}(F\setminus 0)$
which we would like to use, we now have to perform a localisation argument to show that we
may use it. We start by carefully choosing the exact K\"ahler form on $\chi^{-1}(\mathbb{D}^2)$.

$S_m$ is biholomorphic to
$S_{m-1}\times\mathbb{C}^4$ where $S_{m-1}$ is identified with $S_{m-1}\times \{0\}$.
This works because $S_{m-1}$ is a (complex) affine subspace of $S_m$.
Define $\rho_m$ on $S_m$ to be $\rho_{m-1}+\lambda\vecnorm{z}^2$ with respect to this splitting.
Then, $\rho_m|_{S_{m-1}}=\rho_{m-1}$ and, for large enough $\lambda\in\mathbb{R}$,
the neighbourhood $\rho_m^{-1}|_{\chi^{-1}(\mathbb{D}^2}[0,4R)$, is contained within the 
neighbourhood on which the local isomorphism with $\mathbb{C}^4\times_{\mathbb{C}^*}(F\setminus 0)$
is defined.

By using $h$, as above, we can also define a 2--form $-d((dh\circ\rho_m)\circ i)$, which
is identically 0 on $\rho_m^{-1}[0,2R)$ and K\"ahler elsewhere. We will also need a
smooth function $\func{g}{[0,\infty)}{\mathbb{R}}$ which takes value 1 on $[0,2R]$
and 0 on $[3R,\infty)$.

We now choose $U$, small enough that it corresponds to a subset of $\rho_m^{-1}[0,2R)$ under
the local isomorphism. Then, as described above we choose $x$ small, dependent on $U$.
We are now ready to define an isotopy of exact symplectic forms from a composite of linear
between the following 1--forms:

\begin{enumerate}[(1)]
 \item $\Theta_{m}$
 \item $\Theta_{m}+\epsilon g(\rho_m)\Psi^*\tilde{\Theta}$
 \item $-d(h\circ\rho_m)\circ i+\epsilon g(\rho_m)\Psi^*\tilde{\Theta}$
 \item $\Psi_*\tilde{\Theta}+\epsilon g(\rho_m)\Psi^*\tilde{\Theta}$
 \item $\Psi^*\tilde{\Theta}$
\end{enumerate}

\begin{remark}
The following is a summary of the significance of the stages of the isotopy. The aim is to get from the exact
K\"ahler form $d\Theta_m$ on $S_m$, to the $S^1$--equivariant form $d\Psi^*\tilde{\Theta}$ on the model neighbourhood.
The first half of the isotopy, ending at stage 3, allows the calculation of Floer homology and the relative invariant to
be performed locally on the model neighbourhood. The second half of the isotopy occurs only on that model neighbourhood
and ends with the particular symplectic form, required in order to apply the arguments of Section~\ref{section:twistprod}.
\end{remark}

Until stage 3, the 1--forms are defined on all of $\chi^{-1}(\mathbb{D}^2)$. From stage 3 onwards,
they are defined only on the holomorphically convex region $\rho^{-1}[0,4R]$.
So long as $\epsilon>0$ is chosen sufficiently small, the isotopy generates an isotopy of 
exact symplectic forms, where defined.

Corresponding to the isotopy we also get isotopies of iterated the vanishing cycles,
we get from finishing the iterated vanishing cycle constructions with $K_A,K_B$
over the path $(0,0)$ to $(3x^2,2x^3)$ and then along the paths $\psi_A,\psi_B$,
respectively. At stage 3 and later, this iterated vanishing cycle construction and the
extensions around $\partial D$ do not leave
the neighbourhood $U$. In fact it gives,
$L_A\times_{S^1}F_1|_{K_A}$ and
$L_B\times_{S^1}F_1|_{K_B}$, where $F_1$ is the unit circle subbundle of $F$ with respect
the hermitian metric used to define $\Omega_{ass}$.

By convexity of $\rho^{-1}[0,4R]$, holomorphic strips with boundary on these
Lagrangians, must remain within this region. Hence, to study Floer cohomology and the relative invariant,
the isotopy of exact symplectic forms need only be defined on this region. At stage 5, 
a similar convexity argument would show that we can work with
the total space of $\mathbb{C}^4\times_{\mathbb{C}^*}(F\setminus 0)$, rather
than just $\Psi(\rho^{-1}[0,4R])$. However, it is not necessary to do this.

Proposition~\ref{prop:stabmap} concludes the proof that the Floer cochain complexes and
the relative invariants split as described.

It remains only to show that the splitting of Floer cochain complexes
agrees with that of Theorem~\ref{theorem:splitskh} where it describes the
addition of a single unlinked unknot component to $(XP,A,B)$. The splittings
actually agree at the chain level, which can be shown as follows.

One constructs $\psi_A=\psi_B$ to be short enough that the Floer cohomology
calculation can be simultaneously localised in both ways with the same $\sfibre{m-1}{XP}$
factor. Namely, the neighbourhood to which one localises for Theorem~\ref{theorem:splitskh}
can be found within $\mathbb{C}^4\times_{\mathbb{C}^*}(F\setminus 0)$.
\end{proof}

Using this lemma to explicitly calculate the relative invariant
used in the stabilisation map and composing it with the creation map
of Section~\ref{section:crean}, we have explicitly calculated the stabilisation
map. In particular, we have shown that, for the correct choices
of chain complexes it arises from an isomorphism of chain complexes.

Similarly we now have an expression for the destabilisation map, constructed
with respect to the same splittings. Comparing these leads us to the following
observation:

\begin{corollary}
\label{corollary:inv}
The stabilisation
and destabilisation maps on symplectic Khovanov homology at a given vertex are inverses of each other.
\end{corollary}

\subsection{The ``switching move''}

\begin{figure}[h]
\begin{center}
\scalebox{0.6}{\includegraphics{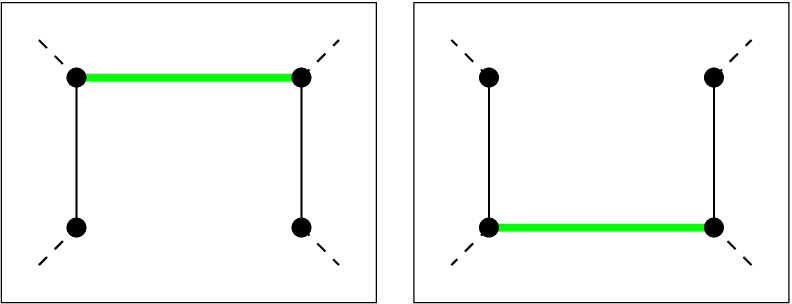}}
\end{center}
\caption{Two ways of performing the same cobordism of admissible links
locally with non-isotopic curve $\delta$ (indicated by grey lines).}
\label{fig:switchcob}
\end{figure}

The two cobordisms specified locally by Figure~\ref{fig:switchcob}
are isotopic as link cobordisms.
However, the curves specifying them are not isotopic. They are
related by a simple ``switching move'' simultaneously
sliding the curve $\delta$
along the two $\alpha$--curves with which it shares vertices.

\begin{proposition}
\label{prop:switchcob}
The two cobordisms specified locally by Figure~\ref{fig:switchcob}
induce the same maps (up to sign) on $\skh $.
\end{proposition}

We start with a weaker result in a very special case.

\begin{lemma}
\label{lemma:switchcob2}
In the case specified locally by Figure~\ref{fig:switchcob2}, the two
cobordisms induce the same maps (up to sign) on the $\mathbb{Z}$ summand
in the bottom degree of cohomology.
\end{lemma}

\begin{figure}[h]
\begin{center}
\scalebox{0.6}{\includegraphics{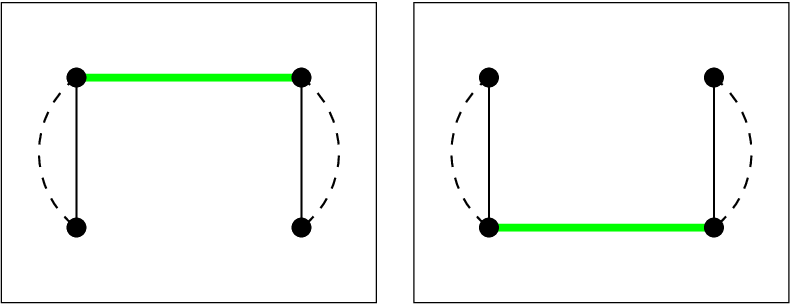}}
\end{center}
\caption{A trivial case of the switching move.}
\label{fig:switchcob2}
\end{figure}

\begin{proof}
Suppose first that we are dealing with the case specified globally by
Figure~\ref{fig:switchcob2}.

The domain of the maps is $H^*(S^2)^{\otimes 2}$ which has
a single $\mathbb{Z}$ summand in the top degree of the cohomology. Both maps
are induced by the saddle cobordism part of different stabilisations, so by the
previous section we know that there is, for each, a splitting in which
they can be written (up to sign) as:
\begin{eqnarray*}
H^*(S^2)\otimes\mathbb{Z}[X]/(X^2)&\to& H^*(S^2)\otimes\mathbb{Z} \\
a\otimes X&\longmapsto& 0 \\
a\otimes 1&\longmapsto& a\otimes 1
\end{eqnarray*}
However, the splitting in either case need not necessarily be the same.

In particular, this means that they induce isomorphisms between the copies of
$\mathbb{Z}$ in the bottom degree of cohomology on both sides. Up to a sign
ambiguity there is only one such isomorphism, so we are done.

In the slightly less trivial case where Figure~\ref{fig:switchcob2} is only
the local model we denote by $D$, the bridge diagram for the rest of the link.
The maps on symplectic Khovanov homology split (using Lemma~\ref{lemma:splitmap}) as the
identity on $\skh (D)$ tensored with the maps
$H^*(S^2)^{\otimes 2}\to H^*(S^2)$ compared in the lemma.
\end{proof}

Now, to prove Proposition~\ref{prop:switchcob}, we refer to
Figure~\ref{fig:switchcob3}. The figure has four numbered rows, each of which describes a
cobordism composed of three elementary saddle cobordisms. The thicker
grey lines indicate the cobordism to be performed as you pass to the
next diagram on that row. Several of the diagrams have two cobordisms
marked, but only when the lines specifying them do not intersect, so by
Lemma~\ref{lemma:commuting} it does not matter in which order they are
performed.

\begin{figure}[ht!]
\begin{center}
\scalebox{0.55}[0.35]{\includegraphics{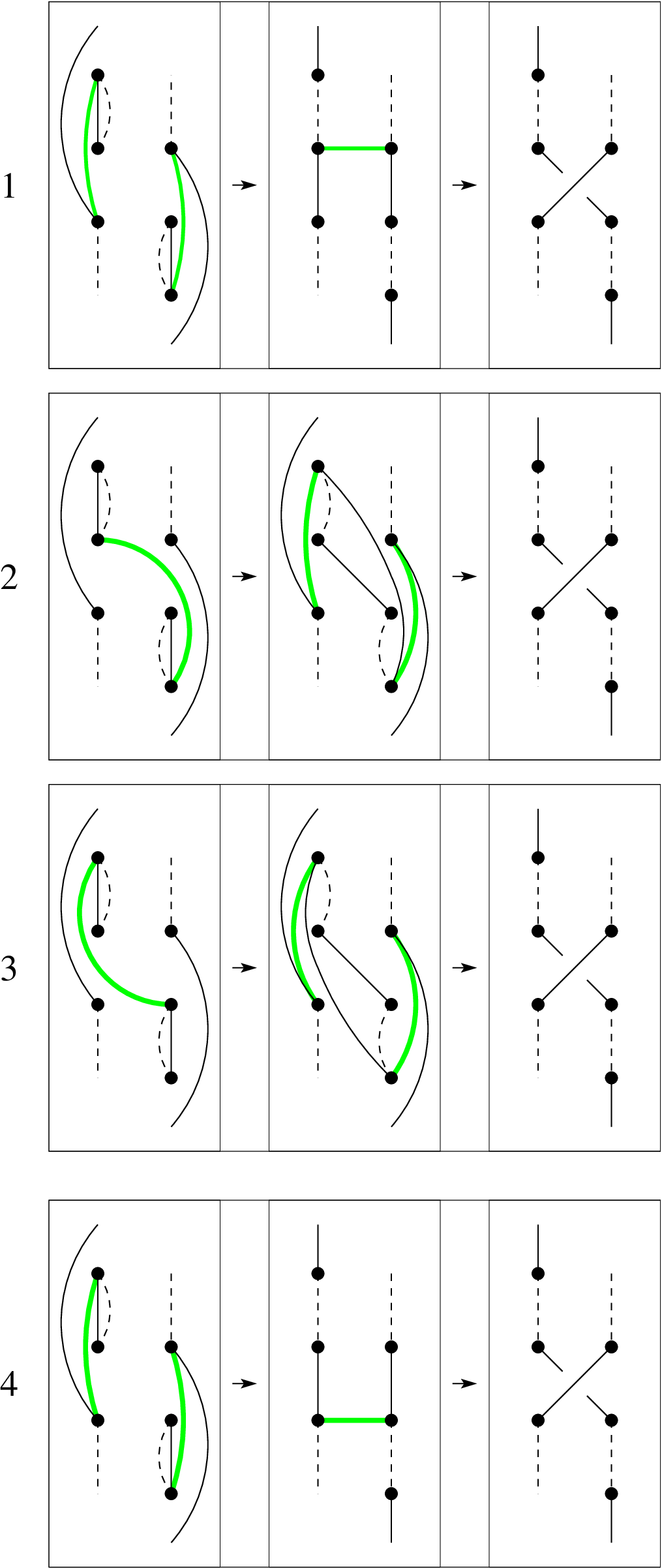}}
\end{center}
\caption{Four related cobordisms used to prove the invariance under the
switching move.}
\label{fig:switchcob3}
\end{figure}

We start with row 1. The first map is the composite of two of the
saddle cobordisms used in stabilisation. The symplectic Khovanov homology splits as
$\skh (D)\otimes H^*(S^2)^{\otimes 2}$ and all terms map to
zero, with the exception of $\skh (D)\otimes 1 \otimes 1$ which maps
isomorphically onto the symplectic Khovanov homology $\skh (D)$ of the next diagram.
The second
map in this row is the cobordism for which we would like to prove the
proposition. The final diagram in row 1 represents the output end of that cobordism.
We have taken the liberty of choosing an admissible position which is
not quite a bridge diagram, purely because it looks simpler (observe there
is a crossing between two $\alpha$--curves).

Row 2 is an application of Lemma~\ref{lemma:commuting} to swap the order
of the cobordisms in row 1. Hence, row 2 gives the same map on symplectic
Khovanov homology.

Row 3 is the same as row 2 except that the switching move has been applied
to the first cobordism. By Lemma~\ref{lemma:switchcob2} rows 2 and 3
give the same map (up to sign) on $\skh (D)\otimes 1 \otimes 1$.

Row 3 is an application of Lemma~\ref{lemma:commuting} to row 4, just as
row 2 was to row 1. Hence rows 3 and 4 give the same map.

In conclusion, rows 1 and 4 restrict to the same map (up to sign) on
$\skh (D)\otimes 1 \otimes 1$. Also the first map in both rows is
identical. Since $\skh (D)\otimes 1 \otimes 1$ maps isomorphically
onto the symplectic Khovanov homology of the second diagram, it must
be the case that the last maps in rows 1 and 4 are the same up to sign.
This concludes the proof of Proposition~\ref{prop:switchcob}.

\subsection{Symplectic Khovanov homology of crossing diagrams}
\label{section:khsympcross}

In this section, we explain how $\skh $ of a crossing diagram
is well-defined, with canonical isomorphisms, up to a possible overall sign ambiguity.
To do this one
defines $\skh $ of a crossing diagram to be $\skh (D)$ for any
bridge position which projects to that crossing diagram. The difficulty
is in specifying a consistent choice of canonical isomorphism between
the symplectic Khovanov homologies of any two bridge diagrams with
the same projection.

Two such bridge diagrams are related by
\begin{itemize}
 \item a sequence of stabilisations and destabilisations preserving the projection
 \item isotopy sliding the $\alpha$-- and $\beta$--curves and vertices
along the projection.
\end{itemize}
We take the stabilisation and destabilisation maps and the maps induced by
these isotopies to be the canonical isomorphisms.

\begin{lemma}
For a crossing diagram in which each component of the link is involved in at
least one crossing the symplectic Khovanov homology is well-defined (up to
sign). Namely, the isomorphisms mentioned above are consistent.
\end{lemma}

\begin{proof}

The condition on components being involved in crossings means that
the diagram can be decomposed into edges (by cutting both strands at each
crossing) without any closed loops remaining. Any bridge diagram projecting
to the crossing diagram will have some number of vertices on each edge (none
are possible at the crossings). Any two such bridge diagrams are isotopic
through such diagrams if and only if these numbers
agree on each edge. Moreover, any two
such isotopies are isotopic (the crossing condition is vital here) so
the maps they induce are the same.

Only the stabilisation and destabilisation maps can change the numbers of
vertices on an edge. By applying the switching move
(Proposition~\ref{prop:switchcob}) to the saddle cobordism of a
stabilisation it is immediate that it makes no difference (up to sign)
at which vertex one stabilises or destabilises. It
therefore suffices to show that a stabilisation followed by a destabilisation
at the same vertex gives the identity map. This is covered by
Corollary~\ref{corollary:inv}.
\end{proof}

Careful consideration of the switching move shows also that the isotopy
which moves two vertices from one edge past a crossing to another using
the passing move is the same as the map that destabilises one edge and
stabilises the next. Hence we can add these isotopies to the choice of 
canonical isomorphisms. This allows the condition that
each component is involved in at least one crossing to be lifted.

A consequence is that, given a neighbourhood in which an
elementary saddle cobordism is performed to a crossing diagram, it
does not matter which of the curves representing a saddle cobordism
one chooses on a bridge diagram to
define the map on symplectic Khovanov homology. Hence, the maps
from creation, annihilation and saddle cobordisms (and similarly isotopy)
are well-defined (up to sign) on the symplectic Khovanov homology of a
crossing diagram.

\subsection{General smooth cobordisms}

Given any smooth cobordism between links in
$\mathbb{R}^3\times\interval$, considerations
in Morse Theory show that it can be put in a position such that
the functional assigning to each point on the cobordism
the value of the time ($t$--coordinate on $\interval$)
has only non-degenerate critical points.
Hence any cobordism is a composite
of finitely many saddle cobordisms and creation/annihilation cobordisms.

In \cite{cartersaito:mmoves} cobordisms are further reduced into `movies',
i.e.\ finite sequences of the basic cobordisms and the Reidemeister moves
on crossing diagrams. A list of `movie moves' is given which relate
isotopic cobordisms, such that any two movies of isotopic cobordisms are
related by a finite sequence of these movie moves.

The movie moves (with the enumeration we shall refer to) are illustrated
on pages 1221 and 1222 of Jacobsson's paper \cite{jacobsson:cob}. 
To avoid confusion,
we should note that numberings in the literature vary and that the moves 9, 10 and 12
of Jacobsson are moves 11, 15 and 13, respectively, under Bar-Natan's more widely used numbering
\cite{barnatan:cob}.

In the
construction of invariants of link cobordisms for combinatorial Khovanov homology, there is a lot
of work required in checking invariance under the moves. However, in the case of
symplectic Khovanov homology, invariance under most of the movie moves is immediate,
since they compare isotopic isotopies of links
and hence induce the same maps on $\skh $. It is possible that a more intrinsically geometric
definition of the maps induced by cap/cup cobordisms would deal with the remaining moves.
An alternative definition of the maps induced by cobordisms was recently given by Rezazadegan
\cite{reza:quilts}. However, it defines the maps induced by cap/cup cobordisms in essentially the same
manner.

We now prove invariance under those movie moves, for which invariance is not immediately obvious by the above.

\begin{figure}[h]
\begin{center}
\scalebox{0.55}{\includegraphics{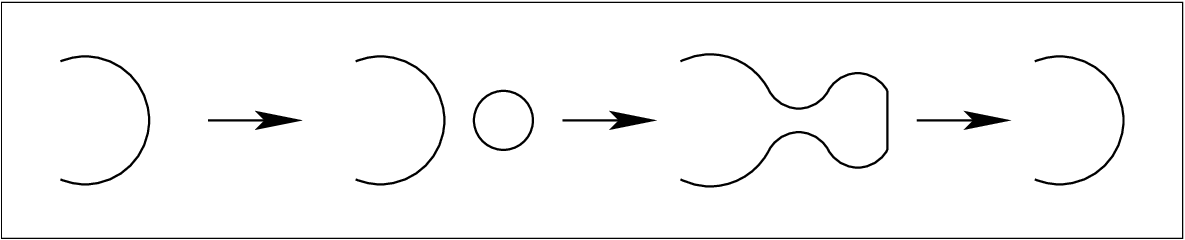}}
\end{center}
\caption{A non-trivial position of the trivial cobordism.}
\label{fig:movie9}
\end{figure}

As a warm-up, there is movie move 9, which relates the cobordism in
Figure~\ref{fig:movie9} to the trivial one. This is realised by
a stabilisation map. Hence, by definition of the canonical isomorphisms,
it gives the identity map.

Moves 10 and 12 are more interesting. They are illustrated in
Figures~\ref{fig:movie1012} as sequences of bridge diagrams with thicker
lines marking the saddle cobordism still to be performed. These lines are
removed once the cobordism is performed. The cobordisms should be read
from top to bottom in each column.

\begin{figure}[ht!]
\begin{center}
\scalebox{0.4}{\includegraphics{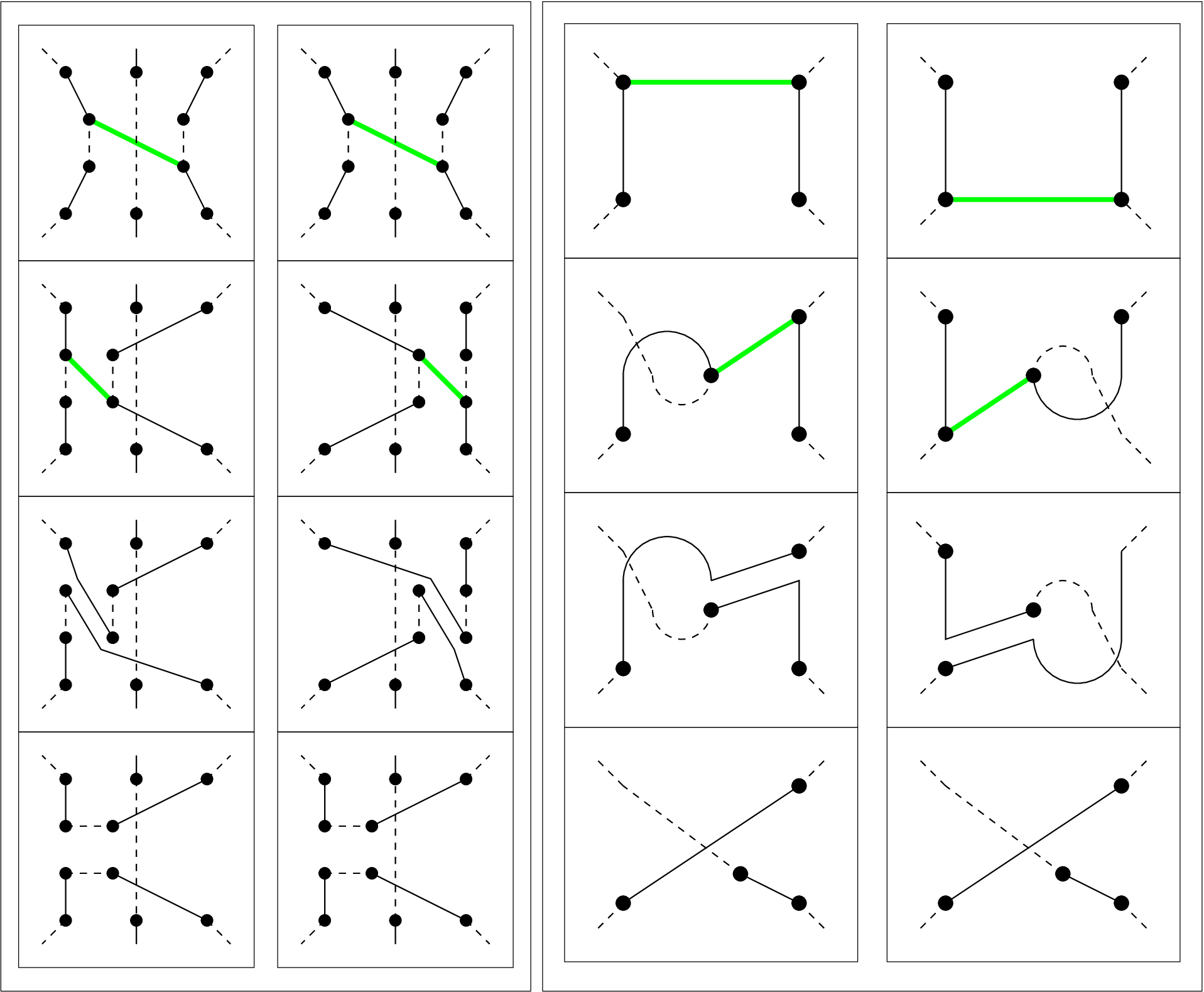}}
\end{center}
\caption{Movie moves 10 (between the two movies on the left) and 12 (between the two
movies on the right) in terms of bridge diagrams.
Each movie should be read from to to bottom.
The thicker lines indicate the cobordisms still to be performed in that movie.}
\label{fig:movie1012}
\end{figure}

Movie move 10 is covered by isotopies of admissible links together
with simultaneous isotopy of the curve $\delta$ defining a saddle
cobordism, which, by the discussion in Section~\ref{subsection:saddles}
proves invariance of the induced map (see Figure~\ref{fig:movie1012}).
Movie move 12 requires, on top of that, the use of the switching move,
Proposition~\ref{prop:switchcob} (see the first part of the movie).

We should also prove invariance under reflected and reversed movie moves,
but the proofs are identical.

As well as being invariant under the movie moves, an invariant of cobordisms up to isotopy must
be invariant under the commuting of `distant' cobordisms (i.e.\ supported
on disjoint neighbourhoods of the diagram). This is immediate from
Lemma~\ref{lemma:commuting}.

In conclusion, the maps (up to sign) on the symplectic Khovanov homology of a link,
induced by specific positions of link cobordisms, are actually invariants
of the isotopy class of smooth cobordism. This concludes the proof of
Theorem~\ref{theorem:main}.

\bibliography{papers}

\medskip

\noindent{\small\sc The Mathematical Institute, 24-29 St. Giles,
Oxford, OX1 3LB, U.K.}

\noindent{\small\sc E-mail: \tt waldron@maths.ox.ac.uk}

\end{document}